\documentclass[onefignum,onetabnum]{siamonline220329}
\pdfoutput=1

\usepackage{lipsum}
\usepackage{amsfonts, amssymb}
\usepackage{subfig}
\usepackage{multicol}
\usepackage{multirow}
\usepackage{graphicx}
\usepackage{epstopdf}
\usepackage{algorithmic}
\ifpdf
  \DeclareGraphicsExtensions{.eps,.pdf,.png,.jpg}
\else
  \DeclareGraphicsExtensions{.eps}
\fi

\usepackage{enumitem}
\setlist[enumerate]{leftmargin=.5in}
\setlist[itemize]{leftmargin=.5in}

\newsiamremark{remark}{Remark}
\newsiamremark{hypothesis}{Hypothesis}
\crefname{hypothesis}{Hypothesis}{Hypotheses}
\newsiamthm{claim}{Claim}

\headers{HessianFR: An Algorithm for Minimax Optimization}{Y. Gao, 
H. Liu, M. Ng, M. Zhou}

\title{HessianFR: An Efficient Hessian-based Follow-the-Ridge Algorithm for Minimax Optimization \thanks{Preprint version. \funding{This work 
is supported by Hong Kong Research Grant Council GRF 12300218, 12300519, 17201020, 17300021, C1013-21GF,
C7004-21GF and Joint NSFC-RGC N-HKU76921.}}}

\author{Yihang Gao\thanks{Department of Mathematics, The University of Hong Kong, Pokfulam, Hong Kong
  (\email{gaoyh@connect.hku.hk}).}
\and Huafeng Liu \thanks{Department of Mathematics, The University of Hong Kong, Pokfulam, Hong Kong
  (\email{huafeng@hku.hk}).}
\and Michael K. Ng \thanks{Department of Mathematics, The University of Hong Kong, Pokfulam, Hong Kong
  (\email{mng@maths.hku.hk}).}
  \and Mingjie Zhou \thanks{Department of Mathematics, The University of Hong Kong, Pokfulam, Hong Kong
  (\email{mjzhou@connect.hku.hk}).}}

\usepackage{amsopn}

\ifpdf
\hypersetup{
  pdftitle={HessianFR: An Efficient Algorithm on Solving Minimax Problems},
  pdfauthor={Y. Gao, M. Zhou, H. Liu, M. Ng}
}
\fi

\begin{document}

\maketitle

\begin{abstract}
Wide applications of differentiable two-player sequential games (e.g., image generation by GANs) have raised much interest and attention of researchers to study efficient and fast algorithms. Most of existing algorithms are developed based on nice properties of simultaneous games, i.e., convex-concave payoff functions, but are not applicable in solving sequential games with different settings. Some conventional gradient descent ascent algorithms theoretically and numerically fail to find the local Nash equilibrium of the simultaneous game or the local minimax (i.e., local Stackelberg equilibrium) of the sequential game. In this paper, we propose the HessianFR, an efficient Hessian-based Follow-the-Ridge algorithm with theoretical guarantees. Furthermore, the convergence of the stochastic algorithm and the approximation of Hessian inverse are exploited to improve algorithm efficiency. A series of experiments of training generative adversarial networks (GANs) have been conducted on both synthetic and real-world large-scale image datasets (e.g. MNIST, CIFAR-10 and CelebA). The experimental results demonstrate that the proposed HessianFR outperforms baselines in terms of convergence and image generation quality.
\end{abstract}



\section{Introduction}
Two-player games as extensions of minimization problems achieve wide applications especially in economics \cite{echenique2003equilibrium} and  machine learning (e.g., generative adversarial networks (GANs) \cite{goodfellow2014generative, arjovsky2017wasserstein, park2019sphere}, adversarial learning \cite{shaham2018understanding, madry2018towards} and reinforcement learning \cite{dai2018sbeed} etc.). We merely focus on differentiable two-player zero-sum games which are mathematically formulated as the following min-max optimization problem:
\begin{equation*} \min_{\mathbf{x}} \max_{\mathbf{y}} f(\mathbf{x},\mathbf{y}), \end{equation*}where $\mathbf{x} \in \mathbb{R}^{d_1}$ and $\mathbf{y} \in \mathbb{R}^{d_2}$ are two players and $f(\mathbf{x},\mathbf{y})$ is the payoff function.  

There are mainly two types of two-player games, i.e., simultaneous games and sequential games. Taking words literally, the difference between two types of games lies in the order of actions two players take. In two-player simultaneous games, $\mathbf{x}$ and $\mathbf{y}$ have the same position that they are blind to other's actions before they finish at each step. It further implies that the payoff function is convex-concave (i.e., $\min_{\mathbf{x}} \max_{\mathbf{y}} f(\mathbf{x},\mathbf{y})=\max_{\mathbf{y}} \min_{\mathbf{x}} f(\mathbf{x},\mathbf{y})$) in simultaneous learning. Sequential games strictly require the order of two players' actions. The variable $\mathbf{x}$ plays the role of a leader who aims to reduce the loss (pay) while the follower $\mathbf{y}$ tries to maximize his gains after observing the leader's action.

Most of previous works stabilized GANs training through regularization and specific modeling. Generators and discriminators in generative adversarial networks (GANs) act as the leader $\mathbf{x}$ and the follower $\mathbf{y}$ respectively. Training GANs is usually equivalent to solving a sequential min-max optimization problem \cite{goodfellow2014generative, arjovsky2017towards, arjovsky2017wasserstein, gulrajani2017improved, park2019sphere}. Optimization difficulties of GANs, especially the instabilities and nonconvergence, have been emphasized and discussed for a long time, since GANs was first proposed \cite{salimans2016improved, goodfellow2016nips, arjovsky2017wasserstein}. From the model's point of view, some proper regularization terms added into the loss function numerically stabilize training \cite{goodfellow2016nips, gulrajani2017improved, brockLRW17, miyato2018spectral, cao2018improving}. Brock et al. \cite{brockLRW17} encouraged weights to be orthonormal to mitigate the instabilities of GANs by introducing an additional regularization term. With 1-Lipschitz constraints on discriminators, Wasserstein GAN achieves much better performances than vanilla GAN in terms of generation quality and training stability \cite{arjovsky2017wasserstein, gulrajani2017improved, wei2018improving, adler2018banach}. However, they are still suffering from the same difficulties using gradient descent ascent algorithms. To escape the optimization difficulties, some recent works of generative models aim to estimate the score of the target distribution by minimizing the fisher divergence \cite{song2019generative, song2020sliced}. No free lunch in the world that it is hard to sample high quality and high dimensional data from the score, although score-based models are more easier to be optimized, because score loses some information compared with the density function. Therefore, it is urgent to develop efficient algorithms for GANs training (sequential games).

Gradient descent ascent (GDA) is the extension of gradient descent from minimization problems to min-max problems. Unfortunately, GDA has been shown to suffer from undesirable convergence and strong rotation around fixed points \cite{daskalakis2018limit}. To overcome the mentioned drawbacks, several variants are proposed. Two Time-Scale GDA \cite{heusel2017gans} and GDA-k \cite{goodfellow2014generative} are two variants of GDA and are widely used in training GANs. Extra gradient (EG) \cite{korpelevich1976extragradient,gidel2018a}, optimistic GDA (OGDA) \cite{daskalakis2018training} and consensus optimization (CO) \cite{mescheder2017numerics}, extended from algorithms solving minimization problems, improve the convergence of GDA. Unfortunately, they are designed for solving convex-concave problems (simultaneous problems).

Recently, Jin et al. \cite{jin2020local} defined the equilibrium (i.e. local minimax) for differentiable sequential games which is more appropriate than Nash equilibrium, based on the works from Evtushenko \cite{evtushenko1974some, evtushenko1974iterative}. Local minimax takes into account the sequential structure and makes use of the Schur complement of the Hessian matrix rather than merely the block diagonal. Based on the definition of local minimax, FR \cite{Wang2020On} and TGDA \cite{fiez2020implicit} are proposed recently and locally converge to local minimax. Furthermore, to accelerate the convergence, Newton-type methods are proposed in \cite{zhang2020newton}. However, some of them may be not applicable in machine learning (deep learning) tasks. For more details, we will discuss later in \cref{proposed_mothods}. 

In this paper, we propose a novel algorithm, namely HessianFR, which has better convergence than FR \cite{Wang2020On} by adding Hessian information to $\mathbf{y}$ in each update but without additional computation. Mathematically, the Hessian information reduces the condition number of Jacobian matrix and thus accelerates the convergence. Overall, equipped with the perspective, here are our contributions:
\begin{itemize}
    \item A new algorithm is proposed and is theoretically guaranteed to locally converge and only converge to local minimax with proper learning rates. 
    
    \item We theoretically and numerically study several fast computation methods including diagonal methods as well as conjugate gradient for Hessian inverse and stochastic learning for lower computation costs in each update. Both diagonal methods and conjugate gradient perform well in practice. 
    
    \item Finally, we apply our algorithm in training generative adversarial networks on synthetic dataset to show the superiority of HessianFR than other algorithms in terms of iterations and seconds for convergence. Furthermore, we  test our algorithm in stochastic setting on large scale image datasets (e.g., MNIST, CIFAR-10 and CelebA). According to numerical results, the proposed HessianFR outperforms other algorithms in terms of the image generation quality.
\end{itemize}

\section{Preliminaries}

\textbf{Notation.} In this paper, we use $\|\cdot\|_2$ to represent the euclidean norm of both vector and the corresponding matrix spectral norm. Concisely, we rewrite $\mathbf{z} = (\mathbf{x},\mathbf{y})$ and the Hessian matrix
\begin{equation*}
\nabla_{\mathbf{z}\mathbf{z}}f(\mathbf{z}) = \left [ \begin{matrix} \nabla_{\mathbf{x}\mathbf{x}}f(\mathbf{x},\mathbf{y}) & \nabla_{\mathbf{x}\mathbf{y}}f(\mathbf{x},\mathbf{y})\\ \nabla_{\mathbf{y}\mathbf{x}}f(\mathbf{x},\mathbf{y}) &  \nabla_{\mathbf{y}\mathbf{y}}f(\mathbf{x},\mathbf{y}) \end{matrix} \right] =: \left [ \begin{matrix} \mathbf{H}_{\mathbf{x}\mathbf{x}} & \mathbf{H}_{\mathbf{x}\mathbf{y}}\\ \mathbf{H}_{\mathbf{y}\mathbf{x}} & \mathbf{H}_{\mathbf{y}\mathbf{y}}  \end{matrix} \right].
\end{equation*}Sometimes, we may use $\mathbf{H}_{\mathbf{x}\mathbf{x}}(\mathbf{x}_t,\mathbf{y}_t)$ to highlight the spatial and temporal location of $\mathbf{H}_{\mathbf{x}\mathbf{x}}$ to avoid ambiguity of notations. Analogous notations holds for $\mathbf{H}_{\mathbf{x}\mathbf{y}}$, $\mathbf{H}_{\mathbf{y}\mathbf{y}}$ and so on. Usually, the spatial and temporal location $\mathbf{z}_t = (\mathbf{x}_t,\mathbf{y}_t)$ is omitted if it is clear and straightforward. To highlight the Hessian matrix evaluated at a given point $\mathbf{z}^{*}=(\mathbf{x}^{*},\mathbf{y}^{*})$ which is within our interest, we denote it as $\mathbf{H}_{\mathbf{x x}}^{*}$, etc. We denote the maximal eigenvalue, the minimal eigenvalue and the spectral radius of a matrix by $\lambda_{\text{max}}(\cdot)$, $\lambda_{\text{min}}(\cdot)$ and $\lambda(\cdot)$ respectively.

Differentiable two-player zero-sum sequential games are mathematically formulated as solving the following min-max optimization problems:
\begin{equation}
\label{min-max}
    \min_{\mathbf{x}} \max_{\mathbf{y}} f(\mathbf{x},\mathbf{y}),
\end{equation}where $\mathbf{x} \in \mathbb{R}^{d_1}$ and $\mathbf{y} \in \mathbb{R}^{d_2}$ are two players and $f(\mathbf{x},\mathbf{y})$ is the payoff function. Note that the payoff function $f(\mathbf{x},\mathbf{y})$ is nonconvex-nonconcave in sequential setting. In this paper, we mainly focus on solving min-max problem for training generative adversarial networks. 

\subsection{Why Minimax Optimization}

Most of previous works define local Nash equilibrium for min-max problems in training generative adversarial networks. We first review the definition and some properties for (local) Nash equilibrium here. Then, we demonstrate that it is overly strict to define Nash equilibrium in training generative adversarial networks.

\begin{definition}
[Local Nash equilibrium]
A point $(\mathbf{x}^{*},\mathbf{y}^{*})$ is a local Nash equilibrium for min-max problem \cref{min-max} if there exists $\delta>0$ such that 
\begin{equation*}
f(\mathbf{x}^{*},\mathbf{y}) \leq f(\mathbf{x}^{*},\mathbf{y}^{*}) \leq f(\mathbf{x},\mathbf{y}^{*}),
\end{equation*}for all $(\mathbf{x},\mathbf{y})$ satisfying $\|\mathbf{x}-\mathbf{x}^{*}\|_2 \leq \delta$ and $\|\mathbf{y}-\mathbf{y}^{*}\|_2 \leq \delta$.
\end{definition}

\begin{proposition}
[Necessary conditions for local Nash equilibrium] \label{necessary_local_Nash_equil}
The local Nash equilibrium $(\mathbf{x}^{*},\mathbf{y}^{*})$ for a twice differentiable payoff function $f(\mathbf{x},\mathbf{y})$ must satisfy the following conditions: (1) It is a critical point (i.e., $\nabla_{\mathbf{x}}f(\mathbf{x}^{*},\mathbf{y}^{*})=0$ and $\nabla_{\mathbf{y}}f(\mathbf{x}^{*},\mathbf{y}^{*})=0$); (2) $\nabla_{\mathbf{y y}} f(\mathbf{x}^{*}, \mathbf{y}^{*}) \preceq \mathbf{0}, \text { and } \nabla_{\mathbf{x x}} f(\mathbf{x}^{*}, \mathbf{y}^{*}) \succeq \mathbf{0}$.
\end{proposition}

\begin{proposition}
[Sufficient conditions for local Nash equilibrium]
\label{sufficient_local_Nash_equil}
For a twice differentiable min-max problem \cref{min-max}, if the critical point $(\mathbf{x}^{*},\mathbf{y}^{*})$ satisfies $\nabla_{\mathbf{y y}} f(\mathbf{x}^{*},\mathbf{y}^{*}) \prec \mathbf{0} \text { and } \nabla_{\mathbf{x x}} f(\mathbf{x}^{*},\mathbf{y}^{*})$ $ \succ \mathbf{0}$, then it is a (strict) local Nash equilibrium. 
\end{proposition}

Local Nash equilibroum implies that the payoff function $f(\mathbf{x},\mathbf{y})$ is locally convex-concave although it is not for simultaneous games. In other words, the local Nash equilibrium $(\mathbf{x}^{*},\mathbf{y}^{*})$ for min-max problem $\min_{\mathbf{x}} \max_{\mathbf{y}} f(\mathbf{x},\mathbf{y})$ is also a local Nash equilibrium for max-min problem $\max_{\mathbf{y}} \min_{\mathbf{x}} f(\mathbf{x},\mathbf{y})$. The Nash equilibrium is overly strict for equilibrium of sequential games. Moreover, it is hard to say whether Nash equilibrium exists for nonconvex-nonconcave min-max problems. For example, we consider a two-dimensional payoff function $f(x,y)=\sin(x+y)$. We can easily verify that no equilibrium exists by \cref{necessary_local_Nash_equil}. The necessary condition of Nash equilibrium requires the Hessian matrix $\nabla_{\mathbf{x}\mathbf{x}}f(\mathbf{x}^{*},\mathbf{y}^{*})$ and $\nabla_{\mathbf{y}\mathbf{y}}f(\mathbf{x}^{*},\mathbf{y}^{*})$ to be semi positive definite and semi negative definite respectively, regardless of the correlation term $\nabla_{\mathbf{x}\mathbf{y}}f(\mathbf{x}^{*},\mathbf{y}^{*})$, which further implies the uncorrelation of variables $\mathbf{x}$ and $\mathbf{y}$ w.r.t. $f(\mathbf{x},\mathbf{y})$ at the local Nash equilibrium $(\mathbf{x}^{*},\mathbf{y}^{*})$.

Generative adversarial networks, first proposed by Goodfellow et al. \cite{goodfellow2014generative}, achieve much attention and are widely applied in various fields and applications \cite{bowman2015generating, odena2017conditional}. In training generative adversarial networks, $\mathbf{x}$ and $\mathbf{y}$ represent the trainable parameters of neural networks for generators and discriminators respectively. In JS-GAN (vanilla GAN) \cite{goodfellow2014generative}, the discriminator parameterized by $\mathbf{y}$ measures the JS divergence between the generated distribution and the target distribution. Wasserstein GAN \cite{arjovsky2017wasserstein} adopts a weaker but softer metric which is numerically represented by 1-Lipschitz neural networks parameterized by $\mathbf{y}$. The discriminator in Sphere GAN \cite{park2019sphere} first projects data to a unit sphere and then measures the data distance on the manifold. Suppose that $\mathbf{y}^{*} = \arg \max_{\mathbf{y}} f(\mathbf{x}^{*},\mathbf{y})$ which measures the distance between the generated distribution and the target distribution, it is not necessary to be a local maximizer for all $\mathbf{x}$ satisfying $\|\mathbf{x}-\mathbf{x}^{*}\|_2 \leq \delta$. 

Based on above analysis, it is inappropriate to depict the solution of sequential games by Nash equilibrium. Jin et al. \cite{jin2020local} defined local minimax for the equilibrium of differentiable sequential games \cref{min-max} which is intuitively and theoretically more feasible.

\begin{definition}
[Local minimax] 
\label{def_local_minmax}
A point $(\mathbf{x}^{*},\mathbf{y}^{*})$ is a local minimax for min-max problem \cref{min-max} if there exists $\delta_0>0$ and a continuous function $h(\delta)$ with $h(\delta) \to 0$ as $\delta \to 0$, such that 
\begin{equation*}
f(\mathbf{x}^{*},\mathbf{y}) \leq f(\mathbf{x}^{*},\mathbf{y}^{*}) \leq \max_{\mathbf{y}^{\prime}:\|\mathbf{y}^{\prime}-\mathbf{y}^{*}\|_2 \leq h(\delta)} f(\mathbf{x},\mathbf{y}^{\prime}),
\end{equation*}
for all $0<\delta \leq \delta_0$, $\|\mathbf{x}-\mathbf{x}^{*}\|_2 \leq \delta$ and $\|\mathbf{y}-\mathbf{y}^{*}\|_2 \leq \delta$.

By implicit function theorem, the definition can be further clarified \cite{Wang2020On}. (1) $y^{*}$ is a local maximum of $f(\mathbf{x}^{*}, \cdot)$; (2) $\mathbf{x}^{*}$ is a local minimum of $f(\mathbf{x},r(\mathbf{x}))$ where $r(\mathbf{x})$ is an implicit function defined by $\nabla_{\mathbf{y}}f(\mathbf{x},r(\mathbf{x}))=0$ in a neighborhood of $(\mathbf{x}^{*},\mathbf{y}^{*})$ with $r(\mathbf{x}^{*})=\mathbf{y}^{*}$.
\end{definition}

Here, local minimax does not require $\mathbf{y}^{*}$ to be the local maximum of $f(\mathbf{x},\cdot)$ for all $\mathbf{x}$ in a neighborhood of $\mathbf{x}^{*}$. Note that the local maximum of $f(\mathbf{x},\cdot)$ is allowed to change slightly (by properties of $h(\delta)$) with $\mathbf{x}$. For more details, please refer to \cite{jin2020local}. Similar to local Nash equilibrium, necessary and sufficient conditions for local minimax are established in \cite{jin2020local}.

\begin{proposition}
[Necessary conditions for local minimax]
\label{necessary_local_minimax}
The local minimax $(\mathbf{x}^{*},\mathbf{y}^{*})$ for a twice differentiable payoff function $f(\mathbf{x},\mathbf{y})$ must satisfies the following conditions: (1) It is a critical point (i.e., $\nabla_{\mathbf{x}}f(\mathbf{x}^{*},\mathbf{y}^{*})=0$ and $\nabla_{\mathbf{y}}f(\mathbf{x}^{*},\mathbf{y}^{*})=0$); (2) $\nabla_{\mathbf{y y}} f(\mathbf{x}^{*},\mathbf{y}^{*}) \preceq \mathbf{0}, \text { and } \nabla_{\mathbf{x x}} f-\nabla_{\mathbf{x}\mathbf{y}}f(\nabla_{\mathbf{y}\mathbf{y}}f)^{-1}\nabla_{\mathbf{y}\mathbf{x}}f|_{(\mathbf{x}^{*},\mathbf{y}^{*})} \succeq \mathbf{0}$.
\end{proposition}

\begin{proposition}
[Sufficient conditions for local minimax]
\label{sufficient_local_minimax}
For a twice differentiable min-max problem \cref{min-max}, if the critical point $(\mathbf{x}^{*},\mathbf{y}^{*})$ satisfies $\nabla_{\mathbf{y y}} f(\mathbf{x}^{*},\mathbf{y}^{*}) \prec \mathbf{0}, \text { and } \nabla_{\mathbf{x x}} f-\nabla_{\mathbf{x}\mathbf{y}}f(\nabla_{\mathbf{y}\mathbf{y}}f)^{-1}\nabla_{\mathbf{y}\mathbf{x}}f|_{(\mathbf{x}^{*},\mathbf{y}^{*})} \succ \mathbf{0}$, then it is a (strict) local minimax. 
\end{proposition}

Comparing \cref{necessary_local_Nash_equil} and \cref{sufficient_local_Nash_equil} with \cref{necessary_local_minimax} and \cref{sufficient_local_minimax}, the main difference lies in that local minimax utilizes the Schur complement of the Hessian matrix while local Nash equilibrium merely focus on the block diagonal matrix. Intuitively, the Schur complement has more information than the block diagonal matrix because the latter ignores the correlation of two variables while the former adopts the whole Hessian information. 

\subsection{Why not GDA and its Variants}
Two Time-Scale GDA (i.e., TTUR in \cite{heusel2017gans}) and GDA-k (adopted in most of GANs training \cite{goodfellow2014generative, arjovsky2017wasserstein, gulrajani2017improved}) are two most popular variants of GDA. The global convergence of Two Time-Scale GDA and GDA-k is less than satisfactory. Two Time-Scale GDA perhaps converges to an undesired point which is neither Nash equilibrium nor local minimax. And GDA-k is convergent if it satisfies Max-Oracle which is extremely strict in practice \cite{jin2020local}. Fortunately, GDA has local convergence properties \cite{zhang2020newton} for local minimax, which may explain the success of GANs trained by GDA in varies tasks. We further find that Extra Gradient (EG) \cite{korpelevich1976extragradient} which is derived for solving convex-concave problems (simultaneous games), is not suitable for solving sequential min-max problem. For more details of (local) convergence of GDA (and its variants) to local minimax, please see \cref{appendix_local_cong_GDA}.

\subsection{Follow-the-Ridge (FR) Algorithm}
Follow-the-Ridge (FR) algorithm \cite{Wang2020On} is the first work studying min-max problem \cref{min-max} based on local minimax. We briefly introduce the main idea of FR algorithm here. Suppose that $(\mathbf{x}_t,\mathbf{y}_t)$ is on the ridge, i.e., $\nabla_{\mathbf{y}}f(\mathbf{x}_t,r(\mathbf{x}_t))=0$ and $r(\mathbf{x}_t)=\mathbf{y}_t$ where $r(\mathbf{x})$ is the implicit function defined in  \cref{def_local_minmax}. In each step $t$, $\mathbf{x}_t$ is updated by gradient descent, i.e., $\mathbf{x}_{t+1}=\mathbf{x}_t - \eta_{\mathbf{x}} \nabla_{\mathbf{x}}f(\mathbf{x}_t,\mathbf{y}_t)$. The leader $\mathbf{x}$ always cannot foresee the follower's action, that's why simple gradient descent is adopted for updating $\mathbf{x}$. However, the follower $\mathbf{y}_t$ witnesses the update of $\mathbf{x}$ and he hopes to take advantage of the additional information and stay on the ridge (i.e., satisfying $\nabla_{\mathbf{y}}f(\mathbf{x},r(\mathbf{x}))=0$). With the current state that $\nabla_{\mathbf{y}}f(\mathbf{x}_t,r(\mathbf{x}_t))=0$ and the known update $\mathbf{x}_{t+1}=\mathbf{x}_t - \eta_{\mathbf{x}} \nabla_{\mathbf{x}}f(\mathbf{x}_t,\mathbf{y}_t)$, the follower $\mathbf{y}_{t+1}$ hopes to satisfy $\nabla_{\mathbf{y}}f(\mathbf{x}_{t+1},\mathbf{y}_{t+1})=0$. Taylor expansion for $\nabla_{\mathbf{y}}f(\mathbf{x}_{t+1},\mathbf{y}_{t+1})$ implies that  
\begin{equation}
    \begin{split}
        0 & = \nabla_{\mathbf{y}}f(\mathbf{x}_{t+1},\mathbf{y}_{t+1})\\
        & \approx \nabla_{\mathbf{y}}f(\mathbf{x}_{t},\mathbf{y}_{t}) + \mathbf{H}_{\mathbf{y}\mathbf{x}}(\mathbf{x}_{t+1} - \mathbf{x}_t) + \mathbf{H}_{\mathbf{y}\mathbf{y}}(\mathbf{y}_{t+1} - \mathbf{y}_t)\\
        & =\mathbf{H}_{\mathbf{y}\mathbf{x}}(-\eta_{\mathbf{x}}\nabla_{\mathbf{x}}f(\mathbf{x}_t,\mathbf{y}_t)) + \mathbf{H}_{\mathbf{y}\mathbf{y}}(\mathbf{y}_{t+1} - \mathbf{y}_t).
    \end{split}
\end{equation}
Therefore, the correction term $\eta_{\mathbf{x}}\mathbf{H}_{\mathbf{y}\mathbf{y}}^{-1}\mathbf{H}_{\mathbf{y}\mathbf{x}}\nabla_{\mathbf{x}}f(\mathbf{x}_t,\mathbf{y}_t)$ brings $\mathbf{y}_t$ to $\mathbf{y}_{t+1}$ with $(\mathbf{x}_{t+1}, \mathbf{y}_{t+1})$ staying on the ridge. If $(\mathbf{x}_t,\mathbf{y}_t)$ is not on the ridge, then gradient ascent update $\eta_{\mathbf{y}}\nabla_{\mathbf{y}}f(\mathbf{x}_t,\mathbf{y}_t)$ for $\mathbf{y}_t$ is necessary for moving $(\mathbf{x}_{t+1},\mathbf{y}_{t+1})$ closer to the ridge. In conclusion, FR algorithm has an additional correction term compared with GDA to keep parallel to or on the ridge. The algorithm is listed as follows:
\begin{equation*}
\begin{array}{l}
\mathbf{x}_{t+1} \leftarrow \mathbf{x}_{t}-\eta_{\mathbf{x}} \nabla_{\mathbf{x}} f\left(\mathbf{x}_{t}, \mathbf{y}_{t}\right), \\
\mathbf{y}_{t+1} \leftarrow \mathbf{y}_{t}+\eta_{\mathbf{y}} \nabla_{\mathbf{y}} f\left(\mathbf{x}_{t}, \mathbf{y}_{t}\right)+\eta_{\mathbf{x}} \mathbf{H}_{\mathbf{y y}}^{-1} \mathbf{H}_{\mathbf{y x}} \nabla_{\mathbf{x}} f\left(\mathbf{x}_{t}, \mathbf{y}_{t}\right),
\end{array}
\end{equation*}where $\eta_{\mathbf{x}}$ and $\eta_{\mathbf{y}}$ are learning rates for $\mathbf{x}$ and $\mathbf{y}$ respectively.

\section{HessianFR}
\label{proposed_mothods}
In this section, we first develop the HessianFR algorithm based on the FR algorithm and the Newton method. Theoretically, we show that the proposed HessianFR algorithm has greater convergence rates to strict local minimax than the FR algorithm. To reduce the computational costs in large scale problems, we extend the deterministic HessianFR to the stochastic HessianFR with theoretical convergence guarantees. Furthermore, we discuss several computation methods for Hessian inverse for the proposed HessianFR.

\subsection{Deterministic Algorithm} In this part, we introduce the motivation for the proposed algorithm and then local convergence is developed. Theoretical analysis conveys that the proposed HessianFR is better than FR \cite{Wang2020On} when $\mathbf{H}_{\mathbf{y y}}^{*}$ is ill-conditioned and is not worse than GDN \cite{zhang2020newton}. However, HessianFR is more computationally friendly than GDN in implementation.

\subsubsection{Motivation} According to the convergence analysis of FR algorithm in \cite{Wang2020On}, the theoretical convergence rate is related to the condition number of $\mathbf{H}_{\mathbf{x}\mathbf{x}}^{*} - \mathbf{H}_{\mathbf{x}\mathbf{y}}\mathbf{H}_{\mathbf{y}\mathbf{y}}^{*-1}\mathbf{H}_{\mathbf{y}\mathbf{x}}^{*}$ and $\mathbf{H}_{\mathbf{y}\mathbf{y}}^{*}$, i.e., the Jacobian matrix of FR at local minimax $\mathbf{z}^{*}=(\mathbf{x}^{*},\mathbf{y}^{*})$ is similar to
\begin{equation}
\mathbf{I} - \eta_{\mathbf{x}} \left[\begin{matrix} \mathbf{H}_{\mathbf{x}\mathbf{x}}^{*} - \mathbf{H}_{\mathbf{x}\mathbf{y}}^{*}\mathbf{H}_{\mathbf{y}\mathbf{y}}^{*-1}\mathbf{H}_{\mathbf{y}\mathbf{x}}^{*} & \mathbf{H}_{\mathbf{x}\mathbf{y}}^{*}\\ & -c\mathbf{H}_{\mathbf{y}\mathbf{y}}^{*} \end{matrix} \right],
\end{equation}which implies that the maximal eigenvalue is
\begin{equation*}
    1-\eta_{\mathbf{x}} \min\{\lambda_{\text{min}}(\mathbf{H}_{\mathbf{x}\mathbf{x}}^{*} - \mathbf{H}_{\mathbf{x}\mathbf{y}}^{*}\mathbf{H}_{\mathbf{y}\mathbf{y}}^{*-1}\mathbf{H}_{\mathbf{y}\mathbf{x}}^{*}), \lambda_{\text{min}}(-c\mathbf{H}_{\mathbf{y}\mathbf{y}}^{*})\},
\end{equation*}
where 
\begin{equation*}
    \eta_{\mathbf{x}} < \frac{2}{\max\{\lambda_{\text{max}}(\mathbf{H}_{\mathbf{x}\mathbf{x}}^{*} - \mathbf{H}_{\mathbf{x}\mathbf{y}}^{*}\mathbf{H}_{\mathbf{y}\mathbf{y}}^{*-1}\mathbf{H}_{\mathbf{y}\mathbf{x}}^{*}), \lambda_{\text{max}}(-c\mathbf{H}_{\mathbf{y}\mathbf{y}}^{*})\}}
\end{equation*}
and $c=\eta_{\mathbf{y}}/\eta_{\mathbf{x}}$. How to improve the convergence of FR without extra computation (or accepted computation costs)?  A direct way is to 
decrease the condition number of two block diagonal of the Jacobian matrix, i.e.,  $\mathbf{H}_{\mathbf{x}\mathbf{x}}^{*} - \mathbf{H}_{\mathbf{x}\mathbf{y}}^{*}\mathbf{H}_{\mathbf{y}\mathbf{y}}^{*-1}\mathbf{H}_{\mathbf{y}\mathbf{x}}^{*}$ and $\mathbf{H}_{\mathbf{y}\mathbf{y}}^{*}$. Furthermore, we hope the follower $\mathbf{y}$ to be close to the optimal one, i.e., $\mathbf{y}_{t+1} \approx \arg \max f(\mathbf{x}_{t+1},\mathbf{y})$. Actions of $\mathbf{y}$ is essential for min-max problem and fast convergence is expected. In training GANs, weak discriminators will lead to disasters and failures. Therefore, instead of improving the updates for $\mathbf{x}$, acceleration of $\mathbf{y}$ seems to be more reasonable. Observe that the correction term for $\mathbf{y}$ involves the inverse of Hessian $\mathbf{H}_{\mathbf{y}\mathbf{y}}^{*-1}$. It reminds us of the Newton method that it outperforms gradient-based (first-order) methods theoretically and numerically under some regularization conditions. Taking the advantage of the Newton method, an extra correction term (Newton step) can accelerate the convergence of $\mathbf{y}$, i.e.,
\begin{equation}
    \mathbf{y}_{t+1} \leftarrow \mathbf{y}_{t} +\eta_{\mathbf{y}1} \nabla_{\mathbf{y}} f\left(\mathbf{x}_{t}, \mathbf{y}_{t}\right) - \eta_{\mathbf{y}2} \mathbf{H}_{\mathbf{y y}}^{-1}\nabla_{\mathbf{y}} f\left(\mathbf{x}_{t}, \mathbf{y}_{t}\right)+\eta_{\mathbf{x}} \mathbf{H}_{\mathbf{y y}}^{-1} \mathbf{H}_{\mathbf{y x}} \nabla_{\mathbf{x}} f\left(\mathbf{x}_{t}, \mathbf{y}_{t}\right),
\end{equation}where $\eta_{\mathbf{y}1}$ and $\eta_{\mathbf{y}2}$ are learning rates for gradient ascent term and Newton correction term respectively. Moreover, compared with FR, HessianFR merely requires additional computation of $\nabla_{\mathbf{y}} f\left(\mathbf{x}_{t}, \mathbf{y}_{t}\right)-\mathbf{H}_{\mathbf{y x}} \nabla_{\mathbf{x}} f\left(\mathbf{x}_{t}, \mathbf{y}_{t}\right)$. The extra computation costs are acceptable and can be ignored in practice. The algorithm is named "HessianFR" because it has Hessian (Newton) step 
in the context of FR algorithm. The pseudocode for HessianFR is shown in \cref{alg:HessianFR}. Here, we use finite difference to compute $\mathbf{H}_{\mathbf{y x}}\nabla_{\mathbf{x}}f(\mathbf{x}_t,\mathbf{y}_t)$, i.e.,

\begin{equation}
    \mathbf{H}_{\mathbf{y x}}\nabla_{\mathbf{x}}f(\mathbf{x}_t,\mathbf{y}_t) \approx \frac{\nabla_{\mathbf{y}}f(\mathbf{x}_t+\alpha\nabla_{\mathbf{x}}f(\mathbf{x}_t,\mathbf{y}_t),\mathbf{y}_t) - \nabla_{\mathbf{y}}f(\mathbf{x}_t,\mathbf{y}_t)}{\alpha},
\end{equation}
when $\alpha \approx 0$. As for the computation of $\mathbf{H}_{\mathbf{y y}}^{-1}$, we will discuss it in \cref{comput_hessian_inverse}.

\begin{algorithm}[tb]
   \caption{HessianFR to solve \cref{min-max}}
   \label{alg:HessianFR}
\begin{algorithmic}
\STATE Given learning rates $\eta_{\mathbf{x}}$, $\eta_{\mathbf{y}1}$ and $\eta_{\mathbf{y}2}$; time steps $T$.
\FOR{$t=0$ {\bfseries to} $T$}
\STATE $\begin{array}{l}
\mathbf{x}_{t+1} \leftarrow \mathbf{x}_{t}-\eta_{\mathbf{x}} \nabla_{\mathbf{x}} f\left(\mathbf{x}_{t}, \mathbf{y}_{t}\right) \\
\mathbf{y}_{t+1} \leftarrow \mathbf{y}_{t} +\eta_{\mathbf{y}1} \nabla_{\mathbf{y}} f\left(\mathbf{x}_{t}, \mathbf{y}_{t}\right) - \eta_{\mathbf{y}2} \mathbf{H}_{\mathbf{y y}}^{-1}\nabla_{\mathbf{y}} f\left(\mathbf{x}_{t}, \mathbf{y}_{t}\right) +\eta_{\mathbf{x}} \mathbf{H}_{\mathbf{y y}}^{-1} \mathbf{H}_{\mathbf{y x}} \nabla_{\mathbf{x}} f\left(\mathbf{x}_{t}, \mathbf{y}_{t}\right)
\end{array}$
\ENDFOR
\end{algorithmic}
\end{algorithm}

\subsubsection{Relation to other algorithms} HessianFR is actually an generalized algorithm of FR \cite{Wang2020On} and GDN \cite{zhang2020newton}. It is equivalent to FR if we set $\eta_{\mathbf{y}2}=0$. The update rule for GDN is as follows:
\begin{equation*}
\begin{array}{l}
\mathbf{x}_{t+1} \leftarrow \mathbf{x}_{t}-\eta_{\mathbf{x}} \nabla_{\mathbf{x}} f\left(\mathbf{x}_{t}, \mathbf{y}_{t}\right), \\
\mathbf{y}_{t+1} \leftarrow \mathbf{y}_{t} - (\nabla_{\mathbf{y y}} f)^{-1} \nabla_{\mathbf{y}} f\left(\mathbf{x}_{t+1}, \mathbf{y}_{t}\right).
\end{array}
\end{equation*}Note that by Taylor expansion, we have
\begin{equation}
    \begin{split}
        \mathbf{y}_{t+1} & = \mathbf{y}_{t} - (\nabla_{\mathbf{y y}} f)^{-1} \nabla_{\mathbf{y}} f\left(\mathbf{x}_{t+1}, \mathbf{y}_{t}\right) \\
        & = \mathbf{y}_{t} - (\nabla_{\mathbf{y y}} f)^{-1} \nabla_{\mathbf{y}} f\left(\mathbf{x}_{t}-\eta_{\mathbf{x}} \nabla_{\mathbf{x}} f\left(\mathbf{x}_{t}, \mathbf{y}_{t}\right), \mathbf{y}_{t}\right)\\
        & \approx \mathbf{y}_{t} - (\nabla_{\mathbf{y y}} f)^{-1} \nabla_{\mathbf{y}} f\left(\mathbf{x}_{t}, \mathbf{y}_{t}\right) + \eta_{\mathbf{x}} (\nabla_{\mathbf{y y}} f)^{-1} \nabla_{\mathbf{y}\mathbf{x}} f \nabla_{\mathbf{x}} f\left(\mathbf{x}_{t}, \mathbf{y}_{t}\right),
    \end{split}
\end{equation}which is a special case of  (equivalent to) HessianFR if $\eta_{\mathbf{y}1}=0$ and $\eta_{\mathbf{y}2}=1$. Theoretical and numerical comparisons of those three related algorithms will be discussed later.

\subsubsection{Convergence Analysis} We first prove the local convergence properties of HessianFR that it converges and only converges to a local minimax. Then, the theoretical convergence rates are compared for HessianFR, FR and GDN. With a proper choice of learning rates, our HessianFR is better than FR and is comparable with GDN in theory.

Here, we first introduce the definition of strict stable points of an algorithm, which is useful for convergence analysis in optimization.

\begin{definition}
[Strict stable point of an algorithm] For an algorithm defined as $\mathbf{z}_{t+1} = g(\mathbf{z}_t)$, we call a point $\mathbf{z}^{*}$ is a strict stable point of the algorithm if
\begin{equation*}
\lambda(\mathbf{J}^{*}) < 1,
\end{equation*}where $\mathbf{J}^{*}=\nabla_{\mathbf{z z}}g(\mathbf{z}^{*})$ is the Jacobian matrix at $\mathbf{z}^{*}$ and $\lambda(\cdot)$ denotes the spectral radius. 
\end{definition}

\begin{theorem}
[Local convergence of HessianFR]
\label{converg}
With a proper choice of learning rates, all strict local minimax points are strict stable fixed points of HessianFR. Moreover, any strict stable fixed point is a strict local minimax.
\end{theorem}

\begin{proof}
Let $c_1 = \eta_{\mathbf{y}1}/\eta_{\mathbf{x}} > 0$ and $c_2 = \eta_{\mathbf{y}2}/\eta{\mathbf{x}} \geq 0$. The Jacobian matrix of HessianFR at local minimax $\mathbf{z}^{*}=(\mathbf{x}^{*},\mathbf{y}^{*})$ is 
\begin{equation}
\label{J_HFR}
    \mathbf{J}^{*}_{\text{HFR}} = \mathbf{I} - \eta_{\mathbf{x}} \left[\begin{matrix} \mathbf{I} & \\ -\mathbf{H}_{\mathbf{y y}}^{*-1}\mathbf{H}_{\mathbf{y x}}^{*} & -c_1\mathbf{I} + c_2 \mathbf{H}_{\mathbf{y y}}^{*-1} \end{matrix}\right] \left[\begin{matrix}\mathbf{H}_{\mathbf{x x}}^{*}   & \mathbf{H}_{\mathbf{x y}}^{*}\\ \mathbf{H}_{\mathbf{y x}}^{*} & \mathbf{H}_{\mathbf{y y}}^{*}   \end{matrix}\right].
\end{equation}Observe that the matrix
\begin{equation*}
\left[\begin{matrix}\mathbf{I} & \\ \mathbf{H}_{\mathbf{y y}}^{*-1}\mathbf{H}_{\mathbf{y x}}^{*} & \mathbf{I} \end{matrix} \right]
\end{equation*}is invertible, then $\mathbf{J}^{*}_{\text{HFR}}$ is similar to
\begin{equation*}
    \begin{split}
        \mathbf{M}^{*}_{\text{HFR}} & = \left[\begin{matrix}\mathbf{I} & \\ \mathbf{H}_{\mathbf{y y}}^{*-1}\mathbf{H}_{\mathbf{y x}}^{*} & \mathbf{I} \end{matrix} \right] \mathbf{J}^{*} \left[\begin{matrix}\mathbf{I} & \\ -\mathbf{H}_{\mathbf{y y}}^{*-1}\mathbf{H}_{\mathbf{y x}}^{*} & \mathbf{I} \end{matrix} \right] \\
        & = \mathbf{I} - \eta_{\mathbf{x}} \left[ \begin{matrix}\mathbf{H}_{\mathbf{x x}}^{*} - \mathbf{H}_{\mathbf{x y}}^{*}\mathbf{H}_{\mathbf{y y}}^{*-1}\mathbf{H}_{\mathbf{y x}}^{*} & \mathbf{H}_{\mathbf{x y}}^{*} \\ & -c_1\mathbf{H}_{\mathbf{y y}}^{*} + c_2 \mathbf{I} \end{matrix}\right],
    \end{split}
\end{equation*}Therefore, $\mathbf{J}^{*}_{\text{HFR}}$ and $\mathbf{M}^{*}_{\text{HFR}} $ share the same eigenvalues. By conditions (\cref{sufficient_local_minimax}) for strict local minimax, we have $\mathbf{H}_{\mathbf{x x}}^{*} - \mathbf{H}_{\mathbf{x y}}^{*}\mathbf{H}_{\mathbf{y y}}^{*-1}\mathbf{H}_{\mathbf{y x}}^{*} \succ \mathbf{0}$ and $-c_1\mathbf{H}_{\mathbf{y y}}^{*} + c_2 \mathbf{I} \succ \mathbf{0}$. To guarantee the convergence of HessianFR, we require $\lambda(\mathbf{J}^{*}_{\text{HFR}})=\lambda(\mathbf{M}^{*}_{\text{HFR}})<1$, i.e.,
\begin{equation*}
-\mathbf{I} \prec \mathbf{I} - \eta_{\mathbf{x}} \left[ \begin{matrix}\mathbf{H}_{\mathbf{x x}}^{*} - \mathbf{H}_{\mathbf{x y}}^{*}\mathbf{H}_{\mathbf{y y}}^{*-1}\mathbf{H}_{\mathbf{y x}}^{*} & \mathbf{H}_{\mathbf{x y}}^{*} \\ & -c_1\mathbf{H}_{\mathbf{y y}}^{*} + c_2 \mathbf{I} \end{matrix}\right] \prec \mathbf{I}
\end{equation*}and
\begin{equation*}
\mathbf{0} \prec \eta_{\mathbf{x}} \left[ \begin{matrix}\mathbf{H}_{\mathbf{x x}}^{*} - \mathbf{H}_{\mathbf{x y}}^{*}\mathbf{H}_{\mathbf{y y}}^{*-1}\mathbf{H}_{\mathbf{y x}}^{*} & \mathbf{H}_{\mathbf{x y}}^{*} \\ & -c_1\mathbf{H}_{\mathbf{y y}}^{*} + c_2 \mathbf{I} \end{matrix}\right] \prec 2\mathbf{I},
\end{equation*}which implies that 
\begin{equation*}
\eta_{\mathbf{x}} < \frac{2}{\max\{\lambda_{\text{max}}(\mathbf{H}_{\mathbf{x x}}^{*} - \mathbf{H}_{\mathbf{x y}}^{*}\mathbf{H}_{\mathbf{y y}}^{*-1}\mathbf{H}_{\mathbf{y x}}^{*}), \lambda_{\text{max}}(-c_1\mathbf{H}_{\mathbf{y y}}^{*} + c_2 \mathbf{I})\}}
\end{equation*}and
\begin{equation*}
\lambda(\mathbf{J}^{*}_{\text{HFR}}) <1.
\end{equation*}Note that the spectral radius is the infimum of all norms and according to  \cite{horn2012matrix}, there exist a matrix norm $\| \cdot \|$ and a constant $\Tilde{\lambda}\in(0,1)$ such that the Jacobbian matrix 
\begin{equation}
\label{jacobian_norm}
    \|\mathbf{J}^{*}_{\text{HFR}}\| = 1 - \Tilde{\lambda} < 1.
\end{equation}Consider the Taylor expansion of HessianFR defined as $\mathbf{z}_{t+1}=g(\mathbf{z}_t)$, we have
\begin{equation*}
g(\mathbf{z}_{t}) = g(\mathbf{z}^{*}) + \mathbf{J}^{*}_{\text{HFR}}(\mathbf{z}_{t} - \mathbf{z}^{*}) + o(\|\mathbf{z}_{t} - \mathbf{z}^{*}\|).
\end{equation*}There exist a small enough neighborhood of $\mathbf{z}^{*}$ with radius $\delta>0$ such that if $\|\mathbf{z} - \mathbf{z}^{*}\|<\delta$, we have $\|o(\|\mathbf{z} - \mathbf{z}^{*}\|)\|<\frac{\Tilde{\lambda}}{2}\|\mathbf{z} - \mathbf{z}^{*}\|$. Therefore,
\begin{equation*}
    \begin{split}
        \|\mathbf{z}_{t+1}-\mathbf{z}^{*}\| & = \|g(\mathbf{z}_{t})-g(\mathbf{z}^{*})\|\\
        & \leq \|\mathbf{J}^{*}_{\text{HFR}}\| \cdot \|\mathbf{z}_{t} - \mathbf{z}^{*}\| + \|o(\|\mathbf{z} - \mathbf{z}^{*}\|)\|\\
        & \leq (1-\frac{\Tilde{\lambda}}{2}) \cdot \|\mathbf{z}_{t} - \mathbf{z}^{*}\|,
    \end{split}
\end{equation*}which complete the proof of local convergence of HessianFR.

Conversely, if a point $\mathbf{z}^{*}=(\mathbf{x}^{*},\mathbf{y}^{*})$ is a strictly stable fixed point of HessianFR, i.e., $\mathbf{z}^{*}$ is a critical point and the Jacobian at $\mathbf{z}^{*}$ satisfies that  $\lambda(\mathbf{J}^{*}) < 1$. It further implies that $\mathbf{H}_{\mathbf{x x}}^{*} - \mathbf{H}_{\mathbf{x y}}^{*}\mathbf{H}_{\mathbf{y y}}^{*-1}\mathbf{H}_{\mathbf{y x}}^{*} \succ \mathbf{0}$ and  $-c_1\mathbf{H}_{\mathbf{y y}}^{*} + c_2 \mathbf{I} \succ \mathbf{0}$ for any $c_1>0$ and $c_2 \geq 0$. Therefore, $\mathbf{H}_{\mathbf{x x}}^{*} - \mathbf{H}_{\mathbf{x y}}^{*}\mathbf{H}_{\mathbf{y y}}^{*-1}\mathbf{H}_{\mathbf{y x}}^{*} \succ \mathbf{0}$ and  $-\mathbf{H}_{\mathbf{y y}}^{*} \succ \mathbf{0}$, which conclude that $\mathbf{z}^{*}$ is a strict local minimax. 
\end{proof}

Note that, the convergence rate of HessianFR can be roughly estimated by eigenvalues of $\mathbf{J}^{*}_{\text{HFR}}$ that
\begin{equation}
\lambda(\mathbf{J}^{*}_{\text{HFR}}) \leq 1-2\kappa_{\text{HFR}},
\end{equation}
where
\begin{equation*}
\kappa_{\text{HFR}}=\frac{\min\{\lambda_{\text{min}}(\mathbf{H}_{\mathbf{x x}}^{*} - \mathbf{H}_{\mathbf{x y}}^{*}\mathbf{H}_{\mathbf{y y}}^{*-1}\mathbf{H}_{\mathbf{y x}}^{*}), \lambda_{\text{min}}(-c_1\mathbf{H}_{\mathbf{y y}}^{*}+c_2\mathbf{I})\}}{\max\{\lambda_{\text{max}}(\mathbf{H}_{\mathbf{x x}}^{*} - \mathbf{H}_{\mathbf{x y}}^{*}\mathbf{H}_{\mathbf{y y}}^{*-1}\mathbf{H}_{\mathbf{y x}}^{*}), \lambda_{\text{max}}(-c_1\mathbf{H}_{\mathbf{y y}}^{*}+c_2\mathbf{I})\}}.
\end{equation*}Let $c_2=0$, the Jacobian matrix $\mathbf{J}^{*}_{\text{FR}}$ of FR at the local minimax $\mathbf{z}^{*}$ is similar to
\begin{equation}
\mathbf{M}^{*}_{\text{FR}} = \mathbf{I} - \eta_{\mathbf{x}} \left[ \begin{matrix}\mathbf{H}_{\mathbf{x x}}^{*} - \mathbf{H}_{\mathbf{x y}}^{*}\mathbf{H}_{\mathbf{y x}}^{*-1}\mathbf{H}_{\mathbf{y x}}^{*} & \mathbf{H}_{\mathbf{x y}}^{*} \\ & -c_1\mathbf{H}_{\mathbf{y y}}^{*} \end{matrix}\right].
\end{equation}and the corresponding spectral radius is
\begin{equation}
\lambda(\mathbf{J}^{*}_{\text{FR}}) \leq 1-2\kappa_{\text{FR}},
\end{equation}with
\begin{equation*}
\kappa_{\text{FR}} = \frac{\min\{\lambda_{\text{min}}(\mathbf{H}_{\mathbf{x x}}^{*} - \mathbf{H}_{\mathbf{x y}}^{*}\mathbf{H}_{\mathbf{y y}}^{*-1}\mathbf{H}_{\mathbf{y x}}^{*}), \lambda_{\text{min}}(-c_1\mathbf{H}_{\mathbf{y y}}^{*})\}}{\max\{\lambda_{\text{max}}(\mathbf{H}_{\mathbf{x x}}^{*} - \mathbf{H}_{\mathbf{x y}}^{*}\mathbf{H}_{\mathbf{y y}}^{*-1}\mathbf{H}_{\mathbf{y x}}^{*}), \lambda_{\text{max}}(-c_1\mathbf{H}_{\mathbf{y y}}^{*})\}}.
\end{equation*}If $\mathbf{H}_{\mathbf{y y}}^{*}$ is ill-conditioned that $\lambda_{\text{max}}(-\mathbf{H}_{\mathbf{y y}}^{*})$ is large but $\lambda_{\text{min}}(-\mathbf{H}_{\mathbf{y y}}^{*})$ is small, then $\kappa_{\text{FR}}$ is small. With proper choice of $c_2$ such that 
\begin{equation}
    \frac{\lambda_{\text{min}}(-c_1\mathbf{H}_{\mathbf{y y}}^{*}+c_2\mathbf{I})}{\lambda_{\text{max}}(-c_1\mathbf{H}_{\mathbf{y y}}^{*}+c_2\mathbf{I})} > \frac{\lambda_{\text{min}}(-c_1\mathbf{H}_{\mathbf{y y}}^{*})}{\lambda_{\text{max}}(-c_1\mathbf{H}_{\mathbf{y y}}^{*})},
\end{equation}
then $\kappa_{\text{HFR}} > \kappa_{\text{FR}}$ and thus $\lambda(\mathbf{J}^{*}_{\text{HFR}})<\lambda(\mathbf{J}^{*}_{\text{FR}})$ which implies the improved theoretical convergence of HessianFR over FR.

Let $c_1=0$ and $c_2=1/\eta_{\mathbf{x}}$, then the Jacobian matrix $\mathbf{J}^{*}_{\text{GDN}}$ of GDN at the local minimax $\mathbf{z}^{*}$ is similar to
\begin{equation}
\mathbf{M}^{*}_{\text{GDN}} = \mathbf{I} - \eta_{\mathbf{x}} \left[ \begin{matrix}\mathbf{H}_{\mathbf{x x}}^{*} - \mathbf{H}_{\mathbf{x y}}^{*}\mathbf{H}_{\mathbf{y x}}^{*-1}\mathbf{H}_{\mathbf{y x}}^{*} & \mathbf{H}_{\mathbf{x y}}^{*} \\ & 1/\eta_{\mathbf{x}} \mathbf{I} \end{matrix}\right],
\end{equation}which implies that 
\begin{equation}
\lambda(\mathbf{J}^{*}_{\text{GDN}}) \leq 1-2\kappa_{\text{GDN}},
\end{equation}where
\begin{equation*}
\kappa_{\text{GDN}} = \frac{\lambda_{\text{min}}(\mathbf{H}_{\mathbf{x x}}^{*} - \mathbf{H}_{\mathbf{x y}}^{*}\mathbf{H}_{\mathbf{y y}}^{*-1}\mathbf{H}_{\mathbf{y x}}^{*})}{\lambda_{\text{max}}(\mathbf{H}_{\mathbf{x x}}^{*} - \mathbf{H}_{\mathbf{x y}}^{*}\mathbf{H}_{\mathbf{y y}}^{*-1}\mathbf{H}_{\mathbf{y x}}^{*})}.
\end{equation*}With proper choice of $c_1$ and $c_2$ (always exist) such that  
\begin{equation*}
    \lambda_{\text{min}}(-c_1\mathbf{H}_{\mathbf{y y}}^{*}+c_2\mathbf{I})\geq \lambda_{\text{min}}(\mathbf{H}_{\mathbf{x x}}^{*} - \mathbf{H}_{\mathbf{x y}}^{*}\mathbf{H}_{\mathbf{y y}}^{*-1}\mathbf{H}_{\mathbf{y x}}^{*})
\end{equation*}
and 
\begin{equation*}
    \lambda_{\text{max}}(-c_1\mathbf{H}_{\mathbf{y y}}^{*}+c_2\mathbf{I})\leq \lambda_{\text{max}}(\mathbf{H}_{\mathbf{x x}}^{*} - \mathbf{H}_{\mathbf{x y}}^{*}\mathbf{H}_{\mathbf{y y}}^{*-1}\mathbf{H}_{\mathbf{y x}}^{*}),
\end{equation*}
then the theoretical convergence rate of HessianFR is equal to that of GDN. 

In conclusion, HessianFR is better than FR when $\mathbf{H}_{\mathbf{y y}}^{*}$ is ill-conditioned and not worse than GDN theoretically. Note that setting learning rate to be $1$ (in GDN) may be infeasible in deep learning applications. Although GDN outperforms FR in the sense of theoretical local convergence, it requires strict pretraining which is hard to achieve in practice. We will discussed it in the numerical part.

\subsubsection{Preconditioning} Preconditioning is a popular method to accelerate the convergence in machine learning and numerical linear algebra. Suppose that in each step of HessianFR, $(\mathbf{x},\mathbf{y})$ is preconditioned by a pair of diagonal matrix $(\mathbf{P}_1,\mathbf{P}_2)$ that are positive definite and bounded, then the convergence properties of HessianFR in \cref{converg} still hold. In the numerical part of this paper, we adopt the same preconditioning strategy as Adam \cite{kingma2014adam}.

\begin{proposition}
[Convergence of HessianFR with preconditioning] Suppose that the gradient in HessianFR is preconditioned by symmetric bounded positive definite matrix pairs $(\mathbf{P}_1,\mathbf{P}_2)$, i.e., $\nabla_{\mathbf{x}}f(\mathbf{x}_t,\mathbf{y}_t)$ and $\nabla_{\mathbf{y}}f(\mathbf{x}_t,\mathbf{y}_t)$ are replaced by $\mathbf{P}_1\nabla_{\mathbf{x}}f(\mathbf{x}_t,\mathbf{y}_t)$ and $\mathbf{P}_2\nabla_{\mathbf{y}}f(\mathbf{x}_t,\mathbf{y}_t)$ respectively. Then, \cref{converg} still holds for HessianFR with preconditioning. 
\end{proposition}

\begin{proof}
The update rule for HessianFR with preconditioning (HFR-P) is as follows:
\begin{equation*}
\begin{split}
\mathbf{x}_{t+1} & = \mathbf{x}_{t}-\eta_{\mathbf{x}} \mathbf{P}_1\nabla_{\mathbf{x}} f\left(\mathbf{x}_{t}, \mathbf{y}_{t}\right), \\
\mathbf{y}_{t+1} & = \mathbf{y}_{t} +\eta_{\mathbf{y}1} \mathbf{P}_2\nabla_{\mathbf{y}} f\left(\mathbf{x}_{t}, \mathbf{y}_{t}\right) - \eta_{\mathbf{y}2} \mathbf{H}_{\mathbf{y y}}^{-1}\mathbf{P}_2\nabla_{\mathbf{y}} f\left(\mathbf{x}_{t}, \mathbf{y}_{t}\right) +\eta_{\mathbf{x}} \mathbf{H}_{\mathbf{y y}}^{-1} \mathbf{H}_{\mathbf{y x}} \mathbf{P}_1\nabla_{\mathbf{x}} f\left(\mathbf{x}_{t}, \mathbf{y}_{t}\right).
\end{split} \end{equation*}The Jacobian for HessianFR with preconditioning is 
\begin{equation*}
 \mathbf{J}^{*}_{\text{HFR-P}} =  \mathbf{I} - \eta_{\mathbf{x}} \left[\begin{matrix} \mathbf{I} & \\ -\mathbf{H}_{\mathbf{y y}}^{*-1}\mathbf{H}_{\mathbf{y x}}^{*} & -c_1\mathbf{I} + c_2 \mathbf{H}_{\mathbf{y y}}^{*-1} \end{matrix}\right] \left[ \begin{matrix} \mathbf{P}_1 & \\ & \mathbf{P}_2 \end{matrix}\right] \left[\begin{matrix}\mathbf{H}_{\mathbf{x x}}^{*}   & \mathbf{H}_{\mathbf{x y}}^{*}\\ \mathbf{H}_{\mathbf{y x}}^{*} & \mathbf{H}_{\mathbf{y y}}^{*}   \end{matrix}\right],
\end{equation*}which is similar to 
\begin{equation*}
    \begin{split}
 \mathbf{M}^{*}_{\text{HFR-P}} 
        & = \left[\begin{matrix}\mathbf{I} & \\ \mathbf{H}_{\mathbf{y y}}^{*-1}\mathbf{H}_{\mathbf{y x}}^{*} & \mathbf{I} \end{matrix} \right] \mathbf{J}^{*}_{\text{HFR-P}} \left[\begin{matrix}\mathbf{I} & \\ -\mathbf{H}_{\mathbf{y y}}^{*-1}\mathbf{H}_{\mathbf{y x}}^{*} & \mathbf{I} \end{matrix} \right]\\
        &  = \mathbf{I} - \eta_{\mathbf{x}} \left[\begin{matrix} \mathbf{I} & \\   & -c_1\mathbf{I} + c_2 \mathbf{H}_{\mathbf{y y}}^{*-1} \end{matrix}\right] \left[ \begin{matrix} \mathbf{P}_1 & \\ & \mathbf{P}_2 \end{matrix}\right]\left[ \begin{matrix}\mathbf{H}_{\mathbf{x x}}^{*} - \mathbf{H}_{\mathbf{x y}}^{*}\mathbf{H}_{\mathbf{y y}}^{*-1}\mathbf{H}_{\mathbf{y x}}^{*} & \mathbf{H}_{\mathbf{x y}}^{*} \\ & \mathbf{H}_{\mathbf{y y}}^{*} \end{matrix}\right]\\
        & = \mathbf{I} - \eta_{\mathbf{x}} \left[ \begin{matrix}\mathbf{P}_1(\mathbf{H}_{\mathbf{x x}}^{*} - \mathbf{H}_{\mathbf{x y}}^{*}\mathbf{H}_{\mathbf{y y}}^{*-1}\mathbf{H}_{\mathbf{y x}}^{*}) & \mathbf{P}_1\mathbf{H}_{\mathbf{x y}}^{*} \\ & (-c_1\mathbf{I} + c_2 \mathbf{H}_{\mathbf{y y}}^{*-1})\mathbf{P}_2\mathbf{H}_{\mathbf{y y}}^{*} \end{matrix}\right].
    \end{split}
\end{equation*}

Note that the matrix 
\begin{equation*}
    \mathbf{P}_1(\mathbf{H}_{\mathbf{x x}}^{*} - \mathbf{H}_{\mathbf{x y}}^{*}\mathbf{H}_{\mathbf{y y}}^{*-1}\mathbf{H}_{\mathbf{y x}}^{*})
\end{equation*}
is similar to \begin{equation*}
    \mathbf{P}_1^{1/2}(\mathbf{H}_{\mathbf{x x}}^{*} - \mathbf{H}_{\mathbf{x y}}^{*}\mathbf{H}_{\mathbf{y y}}^{*-1}\mathbf{H}_{\mathbf{y x}}^{*})\mathbf{P}_1^{1/2},
\end{equation*} 
which is positive definite if $\mathbf{P}_1$ is symmetric positive definite. Moreover, the matrix
\begin{equation*}
    \left(-c_1\mathbf{I} + c_2 \mathbf{H}_{\mathbf{y y}}^{*-1} \right)\mathbf{P}_2\mathbf{H}_{\mathbf{y y}}^{*}
\end{equation*}
is similar to \begin{equation*}
    \mathbf{H}_{\mathbf{y y}}^{*}(-c_1\mathbf{I} + c_2 \mathbf{H}_{\mathbf{y y}}^{*-1})\mathbf{P}_2 = (-c_1\mathbf{H}_{\mathbf{y y}}^{*} + c_2 \mathbf{I})\mathbf{P}_2,
\end{equation*} 
where both $-c_1\mathbf{H}_{\mathbf{y y}}^{*} + c_2 \mathbf{I}$ and $\mathbf{P}_2$ are symmetric positive definite. We deduce that all eigenvalues of 
\begin{equation*}
    \left(-c_1\mathbf{H}_{\mathbf{y y}}^{*} + c_2 \mathbf{I}\right)\mathbf{P}_2,
\end{equation*}
which is similar to 
\begin{equation*}
    \mathbf{P}_2^{1/2}\left(-c_1\mathbf{H}_{\mathbf{y y}}^{*} + c_2 \mathbf{I}\right)\mathbf{P}_2^{1/2},
\end{equation*}
are positive and real. Therefore, eigenvalues of two matrices 
\begin{equation*}
    \mathbf{P}_1 \left (\mathbf{H}_{\mathbf{x x}}^{*} - \mathbf{H}_{\mathbf{x y}}^{*}\mathbf{H}_{\mathbf{y y}}^{*-1}\mathbf{H}_{\mathbf{y x}}^{*} \right )
\end{equation*}
and 
\begin{equation*}
    \left (-c_1\mathbf{I} + c_2 \mathbf{H}_{\mathbf{y y}}^{*-1} \right )\mathbf{P}_2\mathbf{H}_{\mathbf{y y}}^{*}
\end{equation*}
are all real and positive. Furthermore, \cref{converg} holds for HessianFR with preconditioning if  
\begin{equation}
    \eta_{\mathbf{x}} < \frac{2}{\max \left\{\lambda_{\text{max}} \left (\mathbf{P}_1 \left (\mathbf{H}_{\mathbf{x x}}^{*} - \mathbf{H}_{\mathbf{x y}}^{*}\mathbf{H}_{\mathbf{y y}}^{*-1}\mathbf{H}_{\mathbf{y x}}^{*} \right) \right), \lambda_{\text{max}} \left ( \left (-c_1\mathbf{I} + c_2 \mathbf{H}_{\mathbf{y y}}^{*-1} \right)\mathbf{P}_2\mathbf{H}_{\mathbf{y y}}^{*} \right) \right \}}.
\end{equation}
\end{proof}

\subsection{Stochastic Learning}
Large scale data sets and parameters appear in deep learning. In generative adversarial networks, we usually need to estimate millions of parameters and the size of dataset can also be thousands (e.g., MNIST and CIFAR-10) and millions (e.g., CelebA). Instead of solving a deterministic optimization problem, stochastic (mini-batch) learning is adopted to lower computational costs and the storage, but can still approximate the exact solution. In this section, we mainly analyze the convergence properties of HessianFR in stochastic setting and the analysis can also be extended for FR.

\subsubsection{Motivation} We first derive the stochastic algorithm for training generative adversarial networks. For simplicity, GANs model is formulated as
\begin{equation*}
\min_{G \in \mathcal{G}} \max_{F \in \mathcal{F}} \frac{1}{m}\sum_{j=1}^{m} F(G(\mathbf{Z}_j)) - \frac{1}{n}\sum_{k=1}^n F(\mathbf{X}_k),
\end{equation*}where $G \in \mathcal{G}$ and $F \in \mathcal{F}$ are generators and discriminators respectively; $\{\mathbf{Z}_j\}_{j=1}^{m}$ is the set of noises (usually be Gaussian or uniform) while $\{\mathbf{X}_k\}_{k=1}^{n}$ is the observed data from the target distribution. Let $i=(j-1)*m+k$ and $f_i(G,F) = F(G(\mathbf{Z}_j)) - F(\mathbf{X}_k)$, then the object function of GANs model is rewritten into 
\begin{equation*}
\min_{G \in \mathcal{G}} \max_{F \in \mathcal{F}} \frac{1}{mn} \sum_{i=1}^{mn} f_i(G,F).
\end{equation*}Similarly, we can easily check that JS-GAN \cite{goodfellow2014generative}, WGAN (WGAN-clip \cite{arjovsky2017wasserstein}, WGAN-GP \cite{gulrajani2017improved}, WGAN-spectral \cite{miyato2018spectral}, etc.) and other GAN variants satisfy the above property.  Therefore, with finite training data, the optimization problem \cref{min-max} is rewritten as 
\begin{equation}
\label{stochastic_min_max}
    \min_{\mathbf{x}} \max_{\mathbf{y}} f(\mathbf{x}, \mathbf{y}) = \frac{1}{n}\sum_{i=1}^n f_i(\mathbf{x}, \mathbf{y}),
\end{equation}where $f_i$ is the payoff function for the $i$-th training data. The stochastic payoff function in each step $t$ is 
\begin{equation}
\hat{f}(\mathbf{x},\mathbf{y}) = \frac{1}{|\mathcal{S}_t|} \sum_{i \in \mathcal{S}_t} f_i(\mathbf{x},\mathbf{y}),
\end{equation}where $\mathcal{S}_t$ is the set of sampling index at time $t$. The pseudocode stochastic HessianFR to solve \cref{stochastic_min_max} is shown in \cref{alg:stochastic_HessianFR}.

\begin{algorithm}[tb]
   \caption{Stochastic HessianFR to solve \cref{stochastic_min_max}}
   \label{alg:stochastic_HessianFR}
\begin{algorithmic}
\STATE Given learning rates $\eta_{\mathbf{x}}$, $\eta_{\mathbf{y}1}$ and $\eta_{\mathbf{y}2}$; time steps $T$.
\FOR{$t = 0, \cdots, T$}
    \STATE Sample a minibatch $\mathcal{S}_t \subseteq \{1, \cdots n\}$ and the object function becomes
    \begin{equation*}
    \hat{f}(\mathbf{x},\mathbf{y}) = \frac{1}{|\mathcal{S}_t|} \sum_{i \in \mathcal{S}_t} f_i(\mathbf{x},\mathbf{y}),
    \end{equation*}
    \STATE $\begin{array}{l} \mathbf{x}_{t+1} = \mathbf{x}_{t} - \eta_{\mathbf{x}} \nabla_{\mathbf{x}}\hat{f}(\mathbf{x}_t,\mathbf{y}_t) \\ \mathbf{y}_{t+1} = \mathbf{y}_{t} + \eta_{\mathbf{y}1} \nabla_{\mathbf{y}}\hat{f}(\mathbf{x}_t,\mathbf{y}_t) - \eta_{\mathbf{y}2}\widehat{\mathbf{H}}^{-1}_{\mathbf{y}\mathbf{y}}\nabla_{\mathbf{y}}\hat{f}(\mathbf{x}_t,\mathbf{y}_t) +\eta_{\mathbf{x}} \widehat{\mathbf{H}}^{-1}_{\mathbf{y}\mathbf{y}} \widehat{\mathbf{H}}_{\mathbf{y}\mathbf{x}} \nabla_{\mathbf{x}}\hat{f}(\mathbf{x}_t,\mathbf{y}_t)\end{array}$
\ENDFOR
\end{algorithmic}
\end{algorithm}

\subsubsection{Convergence Analysis} We have proved the convergence of HessianFR in solving deterministic min-max problem \cref{min-max}. In this part, we derive a similar convergence property for stochastic HessianFR in solving \cref{stochastic_min_max} under some mild conditions (\cref{standard_assumption}). Suppose that $\mathbf{z}^{*}=(\mathbf{x}^{*},\mathbf{y}^{*})$ is a strict local minimax of the min-max problem \cref{stochastic_min_max} and $\Omega(\mathbf{z}^{*})$ is a neighborhood of $\mathbf{z}^{*}$.

\begin{assumption}
[Standard assumptions for smoothness] \label{standard_assumption} The gradient and Hessian of each objective function $f_i(\mathbf{x},\mathbf{y})$ are bounded in $\Omega(\mathbf{z}^{*}):=\mathcal{B}_2(\mathbf{z}^{*}, r) = \{ \mathbf{z}:\|\mathbf{z}-\mathbf{z}^{*} \|_2 \leq r\}$, i.e., we assume that the following inequalities hold,
\begin{equation}
   \begin{split}
       & \|\nabla_{\mathbf{x}}f_i(\mathbf{x},\mathbf{y})\|_2 \leq \rho_{\mathbf{x}},\\
       & \|\nabla_{\mathbf{y}}f_i(\mathbf{x},\mathbf{y})\|_2 \leq \rho_{\mathbf{y}},\\
       & \|\nabla_{\mathbf{x}\mathbf{y}}f_i(\mathbf{x},\mathbf{y})\|_2 = \|\nabla_{\mathbf{y}\mathbf{x}}f_i(\mathbf{x},\mathbf{y})\|_2 \leq \rho_{\mathbf{x}\mathbf{y}},\\
       & \|\nabla_{\mathbf{y}\mathbf{y}}f_i(\mathbf{x},\mathbf{y})\|_2 \leq \rho_{\mathbf{y}\mathbf{y}},
   \end{split} 
\end{equation}for all $i \in [n]$ and $\mathbf{z}=(\mathbf{x},\mathbf{y}) \in \Omega(\mathbf{z}^{*})$. Furthermore, by the definition that $f(\mathbf{x},\mathbf{y})=\frac{1}{n}\sum_{i=1}^n f_i(\mathbf{x},\mathbf{y})$, $f(\mathbf{x},\mathbf{y})$ also satisfies above inequalities, i.e.,
\begin{equation}
   \begin{split}
       & \|\nabla_{\mathbf{x}}f(\mathbf{x},\mathbf{y})\|_2 \leq \rho_{\mathbf{x}},\\
       & \|\nabla_{\mathbf{y}}f(\mathbf{x},\mathbf{y})\|_2 \leq \rho_{\mathbf{y}},\\
       & \|\nabla_{\mathbf{x}\mathbf{y}}f(\mathbf{x},\mathbf{y})\|_2 = \|\nabla_{\mathbf{y}\mathbf{x}}f(\mathbf{x},\mathbf{y})\|_2 \leq \rho_{\mathbf{x}\mathbf{y}},\\
       & \|\nabla_{\mathbf{y}\mathbf{y}}f(\mathbf{x},\mathbf{y})\|_2 \leq \rho_{\mathbf{y}\mathbf{y}}.
   \end{split} 
\end{equation}
\end{assumption}

The above assumption is standard and mild for smooth functions defined on a compact set when we concentrate on the local convergence. Before developing the convergence analysis for the proposed stochastic HessianFR, we first introduce some useful lemmas in the stochastic learning literature. The payoff function in training generative adversarial networks are
\begin{equation*}
f(\mathbf{x},\mathbf{y}) = \frac{1}{n} \sum_{i =1}^{n} f_i(\mathbf{x},\mathbf{y}),
\end{equation*} and the object function of each mini-batch at time $t$ is
\begin{equation*}
\hat{f}(\mathbf{x},\mathbf{y}) = \frac{1}{|\mathcal{S}_t|} \sum_{i \in \mathcal{S}_t} f_i(\mathbf{x},\mathbf{y}).
\end{equation*}

\begin{lemma}
[Matrix Hoeffding's inequality with uniform sampling for Hermitian Hessian matrix]
\label{hermitian_matrix_hoeffding}
Let $0< \epsilon, \delta < 1$ and 
\begin{equation*}
|\mathcal{S}| \geq \frac{16\rho_{\mathbf{y}\mathbf{y}}^2}{\epsilon^2} \log \frac{2d_2}{\delta}.
\end{equation*}where $\mathcal{S}$ is the set of index. Suppose we pick elements of $\mathcal{S}$ uniformly at random with or without replacement, then we have 
\begin{equation*}
\mathbf{Pr}\left( \left \|\nabla_{\mathbf{y}\mathbf{y}}\hat{f}(\mathbf{x},\mathbf{y}) - \nabla_{\mathbf{y}\mathbf{y}}f(\mathbf{x},\mathbf{y}) \right\|_2 \leq \epsilon \right) \geq 1-\delta.
\end{equation*}
\end{lemma}

\begin{proof}
Denote $\mathbf{H}_j(\mathbf{x},\mathbf{y})=\nabla_{\mathbf{y}\mathbf{y}}f_j(\mathbf{x},\mathbf{y})$, then with uniform sampling,
\begin{equation*}
\mathbf{Pr}\left (\mathbf{H}_j(\mathbf{x},\mathbf{y})=\nabla_{\mathbf{y}\mathbf{y}}f_i(\mathbf{x},\mathbf{y})\right) = \frac{1}{n}.
\end{equation*}Define $\mathbf{X}_j = \mathbf{H}_j(\mathbf{x},\mathbf{y}) - \nabla_{\mathbf{y}\mathbf{y}}f(\mathbf{x},\mathbf{y})$ with $j \in \mathcal{S}$, then $\mathbf{X}_j$ are independent Hermitian matrices with dimension $d_2 \times d_2$ and $\mathbf{E} (\mathbf{X}_j)=0$. Therefore,
\begin{equation}
    \begin{split}
        \left \|\mathbf{X}_j\right \| & = \left \|\nabla_{\mathbf{y}\mathbf{y}}f_j(\mathbf{x},\mathbf{y})-\frac{1}{n}\sum_{i=1}^n \nabla_{\mathbf{y}\mathbf{y}}f_i(\mathbf{x},\mathbf{y})\right \|_2\\
        & = \left \|\frac{n-1}{n}\nabla_{\mathbf{y}\mathbf{y}}f_j(\mathbf{x},\mathbf{y}) - \frac{1}{n} \sum_{i \neq j} \nabla_{\mathbf{y}\mathbf{y}}f_i(\mathbf{x},\mathbf{y}) \right \|_2\\
        & \leq \frac{n-1}{n} \cdot \rho_{\mathbf{y}\mathbf{y}} + \frac{n-1}{n} \cdot \rho_{\mathbf{y}\mathbf{y}} \leq 2\rho_{\mathbf{y}\mathbf{y}}
    \end{split}
\end{equation}and
\begin{equation*}
\|\mathbf{X}_j^2\|_2 \leq 4\rho_{\mathbf{y}\mathbf{y}}^2.
\end{equation*}

Define 
\begin{equation*}
    \begin{split}
        & \widehat{\mathbf{H}} = \frac{1}{|\mathcal{S}|} \sum_{j \in \mathcal{S}} \nabla_{\mathbf{y}\mathbf{y}}f_j(\mathbf{x},\mathbf{y})=\frac{1}{|\mathcal{S}|} \sum_{j \in \mathcal{S}} \mathbf{H}_j(\mathbf{x},\mathbf{y}),\\
        & \widehat{\mathbf{X}} = \sum_{j \in \mathcal{S}}\mathbf{X}_j = |\mathcal{S}|\cdot (\widehat{\mathbf{H}}-\mathbf{H}),
    \end{split}
\end{equation*}
where $\mathbf{H}=\nabla_{\mathbf{y}\mathbf{y}}f(\mathbf{x},\mathbf{y})$. By Matrix Bernstein Inequality \cite{gross2011recovering, tropp2012user, kohler2017sub}, we have
\begin{equation*}
\mathbf{Pr}\left( \|\widehat{\mathbf{X}}\|_2 \geq \epsilon \right) \leq 2d_2 \cdot \exp \left (\frac{-\epsilon^2}{16 |\mathcal{S}|\rho_{\mathbf{y}\mathbf{y}}^2}\right ),
\end{equation*}and thus
\begin{equation*}
\mathbf{Pr}\left(\|\widehat{\mathbf{H}}-\mathbf{H}\|_2 \geq \epsilon \right ) \leq 2d_2  \cdot \exp \left (\frac{-\epsilon^2 \cdot |\mathcal{S}|}{16 \rho_{\mathbf{y}\mathbf{y}}^2}\right ),
\end{equation*}which completes the proof.
\end{proof}

\begin{lemma}
[Matrix Hoeffding's inequality with uniform sampling for Rectangular Hessian matrix] 
\label{rectangular_matrix_hoeffding}
Let $0<\epsilon \leq \rho_{\mathbf{x}\mathbf{y}}$, $0<\delta <1$ and 
\begin{equation*}
|\mathcal{S}| \geq \frac{16 \rho_{\mathbf{x}\mathbf{y}}^2}{\epsilon^2} \log \frac{d_1 + d_2}{\delta}.
\end{equation*}where $\mathcal{S}$ is the set of index. Suppose we pick elements of $\mathcal{S}$ uniformly at random with or without replacement, then we have 
\begin{equation*}
\mathbf{Pr}\left( \left \|\nabla_{\mathbf{x}\mathbf{y}}\hat{f}(\mathbf{x},\mathbf{y}) - \nabla_{\mathbf{x}\mathbf{y}}f(\mathbf{x},\mathbf{y}) \right\|_2 \leq \epsilon \right) \geq 1-\delta
\end{equation*}and
\begin{equation*}
\mathbf{Pr}\left( \left \|\nabla_{\mathbf{y}\mathbf{x}}\hat{f}(\mathbf{x},\mathbf{y}) - \nabla_{\mathbf{y}\mathbf{x}}f(\mathbf{x},\mathbf{y}) \right\|_2 \leq \epsilon \right) \geq 1-\delta.
\end{equation*}
\end{lemma}

\begin{proof}
Similar proof can be achieved as \cref{hermitian_matrix_hoeffding} by using Bernstein inequality for Rectangular matrices \cite{tropp2012user}.
\end{proof}

\begin{lemma}
[Vector Hoeffding's inequality with uniform sampling]
\label{vector_matrix_hoeffding}
Let $0< \epsilon, \delta < 1$ and 
\begin{equation*}
|\mathcal{S}| \geq \frac{32\rho_{\mathbf{x}}^2}{\epsilon^2} \cdot \left (\frac{1}{4}-\log \delta \right).
\end{equation*}Suppose we pick elements of $\mathcal{S}$ uniformly at random with or without replacement, then we have 
\begin{equation*}
\mathbf{Pr}\left( \left \|\nabla_{\mathbf{x}}\hat{f}(\mathbf{x},\mathbf{y}) - \nabla_{\mathbf{x}}f(\mathbf{x},\mathbf{y}) \right\|_2 \leq \epsilon \right) \geq 1-\delta.
\end{equation*}Similarly, if
\begin{equation*}
|\mathcal{S}| \geq \frac{32\rho_{\mathbf{y}}^2}{\epsilon^2} \cdot \left (\frac{1}{4}-\log \delta \right),
\end{equation*}then
\begin{equation*}
\mathbf{Pr}\left( \left \|\nabla_{\mathbf{y}}\hat{f}(\mathbf{x},\mathbf{y}) - \nabla_{\mathbf{y}}f(\mathbf{x},\mathbf{y}) \right\|_2 \leq \epsilon \right) \geq 1-\delta.
\end{equation*}
\end{lemma}

\begin{proof}
Similar proof can be achieved as \cref{hermitian_matrix_hoeffding} by using Bernstein inequality for vectors \cite{gross2011recovering, kohler2017sub}.
\end{proof}

\begin{theorem}
[Local convergence of stochastic HessianFR] 
\label{converg_stochastic_HessianFR}
Let $0 < \epsilon \leq \rho_{\mathbf{x}\mathbf{y}}$, $0< \delta < 1$ and $T>0$ to be a positive integer denoting the number of time steps in \cref{alg:stochastic_HessianFR}. We further let
\begin{equation}
\label{sampling_numbers}
    \begin{array}{ll}
      |\mathcal{S}_t| \geq \max \{ & \frac{16\rho_{\mathbf{y}\mathbf{y}}^2}{\epsilon^2} \log \frac{8d_2T}{\delta}, \frac{16 \rho_{\mathbf{x}\mathbf{y}}^2}{\epsilon^2} \log \frac{4(d_1 + d_2)T}{\delta},  \\
    & \frac{32\rho_{\mathbf{x}}^2}{\epsilon^2} \cdot \left (\frac{1}{4}+\log4T-\log \delta \right),\\ 
     &\frac{32\rho_{\mathbf{y}}^2}{\epsilon^2} \cdot \left (\frac{1}{4}+\log4T-\log \delta \right)  \}
    \end{array}
\end{equation}and suppose that we pick elements of $\mathcal{S}_t$ uniformly at random with or without replacement in each step for $t=0, \cdots T$. There exists a small enough $\Tilde{r} \leq r$ and a constant $C$ (depends only on constants defined in  \cref{standard_assumption} but independent of $T$) such that if $\|\mathbf{z}_0 - \mathbf{z}^{*} \|_2 \leq \Tilde{r}$, then with probability at least $1-\delta$, we have
\begin{equation}
    \|\mathbf{z}_T - \mathbf{z}^{*} \|_2 \leq \sigma^T \cdot \frac{w_2}{w_1} \cdot \|\mathbf{z}_0 - \mathbf{z}^{*} \|_2 + C \cdot \frac{1-\sigma^T}{1-\sigma} \cdot \epsilon,
\end{equation}for some $\sigma<1$ and $w_1, w_2 >0$ (independent of $\epsilon$, $\delta$ and $T$), if $\epsilon$ is small enough.
\end{theorem}

\begin{proof}
By the smoothness of the object function (twice continuously differentiable), there exists a region (a neighborhood of the strict local minimax point $\mathbf{z}^{*}$) $\mathcal{B}_2(\mathbf{z}^{*}, \Tilde{r}) = \{ \mathbf{z}:\|\mathbf{z}-\mathbf{z}^{*} \|_2 \leq \Tilde{r}\}$, such that the Taylor expansion of the algorithm in the region satisfies
\begin{equation}
\label{taylor_HFR}
    g(\mathbf{z}) = g(\mathbf{z}^{*}) + \mathbf{J}_{\text{HFR}}(\mathbf{z} - \mathbf{z}^{*})
\end{equation}and
\begin{equation*}
\|\mathbf{J}_{\text{HFR}}\| \leq \sigma < 1,
\end{equation*}where the norm $\|\cdot\|$ is defined in \cref{jacobian_norm},
and
the negative Hessian matrix of the object function w.r.t. $\mathbf{y}$ in the neighborhood of the strict local minimax $\mathbf{z}^{*}$ is positive definite, i.e.,
\begin{equation*}
\mu_1 \mathbf{I} \preceq -\mathbf{H}_{\mathbf{y}\mathbf{y}} \preceq \mu_2 \mathbf{I}.
\end{equation*}By the equivalence of norms in finite dimensions, there exist positive constants $w_1$ and $w_2$ such that
\begin{equation*}
w_1 \|\mathbf{z} \|_2 \leq \|\mathbf{z} \| \leq w_2 \|\mathbf{z} \|_2, \hspace{1em} \forall \mathbf{z}.
\end{equation*}

The updates in \cref{alg:stochastic_HessianFR} can be rewritten as
\begin{equation}
\label{update_stochatic}
         \left [ \begin{matrix}
            \mathbf{x}_{t+1} \\
            \mathbf{y}_{t+1}
        \end{matrix} \right ]  = 
        \left [ \begin{matrix}
            \mathbf{x}_{t} \\
            \mathbf{y}_{t}
        \end{matrix} \right ] - \eta_{\mathbf{x}} \left [ \begin{matrix}
        \mathbf{I} & \\
        -\mathbf{H}^{-1}_{\mathbf{y}\mathbf{y}}\mathbf{H}_{\mathbf{y}\mathbf{x}}& c_1 \mathbf{H}^{-1}_{\mathbf{y}\mathbf{y}}-c_2 \mathbf{I}
        \end{matrix} \right] \cdot \left [ 
        \begin{matrix}
         \nabla_{\mathbf{x}} f(\mathbf{x}_t,\mathbf{y}_t)\\
          \nabla_{\mathbf{y}} f(\mathbf{x}_t,\mathbf{y}_t)
        \end{matrix}
        \right] 
         + \mathcal{E}_t,
\end{equation}where
\begin{equation*}
    \mathbf{\mathcal{E}}_t = \eta_{\mathbf{x}} \cdot \left[ \begin{matrix}
     \mathcal{E}_{tx}\\
     \mathcal{E}_{ty1}+\mathcal{E}_{ty2}+\mathcal{E}_{ty3}
    \end{matrix} \right ]
\end{equation*} 
with

\begin{equation*}
    \mathcal{E}_{tx}=\nabla_{\mathbf{x}} f(\mathbf{x}_t,\mathbf{y}_t)-\nabla_{\mathbf{x}} \hat{f}(\mathbf{x}_t,\mathbf{y}_t),
\end{equation*}

\begin{equation*}
    \mathcal{E}_{ty1}=\widehat{\mathbf{H}}^{-1}_{\mathbf{y}\mathbf{y}}\widehat{\mathbf{H}}_{\mathbf{y}\mathbf{x}}\nabla_{\mathbf{x}} \hat{f}(\mathbf{x}_t,\mathbf{y}_t)-\mathbf{H}^{-1}_{\mathbf{y}\mathbf{y}}\mathbf{H}_{\mathbf{y}\mathbf{x}}\nabla_{\mathbf{x}} f(\mathbf{x}_t,\mathbf{y}_t),
\end{equation*}

\begin{equation*}
    \mathcal{E}_{ty2}=c_1\mathbf{H}^{-1}_{\mathbf{y}\mathbf{y}}\nabla_{\mathbf{y}}f(\mathbf{x}_t,\mathbf{y}_t)-c_1\widehat{\mathbf{H}}^{-1}_{\mathbf{y}\mathbf{y}}\nabla_{\mathbf{y}}\hat{f}(\mathbf{x}_t,\mathbf{y}_t),
\end{equation*}
and
\begin{equation*}
    \mathcal{E}_{ty3}=c_2\nabla_{\mathbf{y}}\hat{f}(\mathbf{x}_t,\mathbf{y}_t) - c_2 \nabla_{\mathbf{y}}f(\mathbf{x}_t,\mathbf{y}_t).
\end{equation*}

If $\mathcal{S}_t$ satisfies \cref{sampling_numbers}, then by \cref{hermitian_matrix_hoeffding}, \cref{rectangular_matrix_hoeffding} and \cref{vector_matrix_hoeffding} that with probability at least $1-\delta$, the following inequalities hold simultaneously:
\begin{equation}
\begin{split}
    & \left \|\nabla_{\mathbf{y}\mathbf{y}}\hat{f}(\mathbf{x}_t,\mathbf{y}_t) - \nabla_{\mathbf{y}\mathbf{y}}f(\mathbf{x}_t,\mathbf{y}_t) \right\|_2 \leq \epsilon,\\
    & \left \|\nabla_{\mathbf{x}\mathbf{y}}\hat{f}(\mathbf{x}_t,\mathbf{y}_t) - \nabla_{\mathbf{x}\mathbf{y}}f(\mathbf{x}_t,\mathbf{y}_t) \right\|_2 \leq \epsilon,\\
    & \left \|\nabla_{\mathbf{x}}\hat{f}(\mathbf{x}_t,\mathbf{y}_t) - \nabla_{\mathbf{x}}f(\mathbf{x}_t,\mathbf{y}_t) \right\|_2 \leq \epsilon,\\
    & \left \|\nabla_{\mathbf{y}}\hat{f}(\mathbf{x}_t,\mathbf{y}_t) - \nabla_{\mathbf{y}}f(\mathbf{x}_t,\mathbf{y}_t) \right\|_2 \leq \epsilon.
\end{split}
\end{equation}for all $t=0, \cdots, T$. Suppose $\epsilon \leq \mu_1/2 < \mu_1$ is small enough such that 
\begin{equation*}
(\mu_1-\epsilon) \mathbf{I} \preceq -\widehat{\mathbf{H}}_{\mathbf{y}\mathbf{y}},
\end{equation*}and thus
\begin{equation*}
\| \widehat{\mathbf{H}}_{\mathbf{y}\mathbf{y}}^{-1}\|_2 \leq (\mu_1 - \epsilon)^{-1}.
\end{equation*}Furthermore, 
\begin{equation*}
    \begin{split}
        \|\mathbf{H}_{\mathbf{y}\mathbf{y}}^{-1} - \widehat{\mathbf{H}}_{\mathbf{y}\mathbf{y}}^{-1}\|_2 & = \|\widehat{\mathbf{H}}_{\mathbf{y}\mathbf{y}}^{-1}(\widehat{\mathbf{H}}_{\mathbf{y}\mathbf{y}} - \mathbf{H}_{\mathbf{y}\mathbf{y}})\mathbf{H}_{\mathbf{y}\mathbf{y}}^{-1}\|_2 \\
        & \leq \|\widehat{\mathbf{H}}_{\mathbf{y}\mathbf{y}}^{-1}\|_2 \cdot \|\widehat{\mathbf{H}}_{\mathbf{y}\mathbf{y}} - \mathbf{H}_{\mathbf{y}\mathbf{y}}\|_2 \cdot \|\mathbf{H}_{\mathbf{y}\mathbf{y}}^{-1}\|_2 \\
        & \leq (\mu_1 - \epsilon)^{-1} \cdot \mu_1^{-1} \cdot \epsilon.
    \end{split}
\end{equation*}In the next step, we can bound the error term $\mathcal{E}_t$ w.r.t. $\epsilon$. For $\mathcal{E}_{tx}$,
\begin{equation}
\label{Error_x}
    \|\mathcal{E}_{tx}\|_2 = \|\nabla_{\mathbf{x}}f(\mathbf{x}_t,\mathbf{y}_t) - \nabla_{\mathbf{x}}\hat{f}(\mathbf{x}_t,\mathbf{y}_t) \|_2 \leq \epsilon.
\end{equation}For $\mathcal{E}_{ty1}$,
\begin{equation*}
    \begin{split}
        \mathcal{E}_{ty1} & = \widehat{\mathbf{H}}^{-1}_{\mathbf{y}\mathbf{y}}\widehat{\mathbf{H}}_{\mathbf{y}\mathbf{x}}\nabla_{\mathbf{x}} \hat{f}(\mathbf{x}_t,\mathbf{y}_t)-\mathbf{H}^{-1}_{\mathbf{y}\mathbf{y}}\mathbf{H}_{\mathbf{y}\mathbf{x}}\nabla_{\mathbf{x}} f(\mathbf{x}_t,\mathbf{y}_t)\\
        & = \left ( \widehat{\mathbf{H}}^{-1}_{\mathbf{y}\mathbf{y}}\widehat{\mathbf{H}}_{\mathbf{y}\mathbf{x}} - \mathbf{H}^{-1}_{\mathbf{y}\mathbf{y}}\widehat{\mathbf{H}}_{\mathbf{y}\mathbf{x}} + \mathbf{H}^{-1}_{\mathbf{y}\mathbf{y}}\widehat{\mathbf{H}}_{\mathbf{y}\mathbf{x}} - \mathbf{H}^{-1}_{\mathbf{y}\mathbf{y}}\mathbf{H}_{\mathbf{y}\mathbf{x}} \right)\cdot \nabla_{\mathbf{x}} \hat{f}(\mathbf{x}_t,\mathbf{y}_t)\\ & \hspace{1em}+\mathbf{H}^{-1}_{\mathbf{y}\mathbf{y}}\mathbf{H}_{\mathbf{y}\mathbf{x}}\left(\nabla_{\mathbf{x}} \hat{f}(\mathbf{x}_t,\mathbf{y}_t) - \nabla_{\mathbf{x}} f(\mathbf{x}_t,\mathbf{y}_t)\right),\\
    \end{split}
\end{equation*}and therefore,
\begin{equation}
\label{Error_y1}
    \begin{split}
     \|\mathcal{E}_{ty1}\|_2 & \leq \left(\|\mathbf{H}_{\mathbf{y}\mathbf{y}}^{-1} - \widehat{\mathbf{H}}_{\mathbf{y}\mathbf{y}}^{-1}\|_2 \cdot \|\widehat{\mathbf{H}}_{\mathbf{y}\mathbf{x}}\|_2 + \|\mathbf{H}_{\mathbf{y}\mathbf{y}}^{-1}\|_2 \cdot \|\widehat{\mathbf{H}}_{\mathbf{y}\mathbf{x}}-\mathbf{H}_{\mathbf{y}\mathbf{x}}\|_2 \right) \cdot \|\nabla_{\mathbf{x}} \hat{f}(\mathbf{x}_t,\mathbf{y}_t)\|\\
        & \hspace{1em} + \|\mathbf{H}^{-1}_{\mathbf{y}\mathbf{y}}\|_2 \cdot  \|\mathbf{H}_{\mathbf{y}\mathbf{x}}\|_2 \cdot \|\nabla_{\mathbf{x}} \hat{f}(\mathbf{x}_t,\mathbf{y}_t) - \nabla_{\mathbf{x}} f(\mathbf{x}_t,\mathbf{y}_t)\|_2\\
        & \leq \left((\mu_1 - \epsilon)^{-1} \cdot \mu_1^{-1} \cdot \rho_{\mathbf{x}\mathbf{y}}\cdot \rho_{\mathbf{x}}+\mu_1^{-1} \cdot \rho_{\mathbf{x}}+\mu_1^{-1} \cdot \rho_{\mathbf{x}\mathbf{y}}\right) \cdot \epsilon.
    \end{split}
\end{equation}

For $\mathcal{E}_{ty2}$,
\begin{equation*}
    \begin{split}
        \mathcal{E}_{ty2} & = c_1\mathbf{H}^{-1}_{\mathbf{y}\mathbf{y}}\nabla_{\mathbf{y}}f(\mathbf{x}_t,\mathbf{y}_t)-c_1\widehat{\mathbf{H}}^{-1}_{\mathbf{y}\mathbf{y}}\nabla_{\mathbf{y}}\hat{f}(\mathbf{x}_t,\mathbf{y}_t)\\
        & = c_1 \cdot \left(\mathbf{H}^{-1}_{\mathbf{y}\mathbf{y}}\nabla_{\mathbf{y}}f(\mathbf{x}_t,\mathbf{y}_t) - \mathbf{H}^{-1}_{\mathbf{y}\mathbf{y}}\nabla_{\mathbf{y}}\hat{f}(\mathbf{x}_t,\mathbf{y}_t)  +\mathbf{H}^{-1}_{\mathbf{y}\mathbf{y}}\nabla_{\mathbf{y}}\hat{f}(\mathbf{x}_t,\mathbf{y}_t) - \widehat{\mathbf{H}}^{-1}_{\mathbf{y}\mathbf{y}}\nabla_{\mathbf{y}}\hat{f}(\mathbf{x}_t,\mathbf{y}_t)\right)
    \end{split}
\end{equation*}and
\begin{equation}
\label{Error_y2}
    \begin{split}
        \|\mathcal{E}_{ty2}\|_2 & \leq c_1 \cdot \left( \|\mathbf{H}^{-1}_{\mathbf{y}\mathbf{y}}\|_2 \cdot \|\nabla_{\mathbf{y}}f(\mathbf{x}_t,\mathbf{y}_t) - \nabla_{\mathbf{y}}\hat{f}(\mathbf{x}_t,\mathbf{y}_t)\|_2  + \|\mathbf{H}_{\mathbf{y}\mathbf{y}}^{-1} - \widehat{\mathbf{H}}_{\mathbf{y}\mathbf{y}}^{-1}\|_2 \cdot \|\nabla_{\mathbf{y}}\hat{f}(\mathbf{x}_t,\mathbf{y}_t)\|_2 \right)\\
        & \leq c_1 \cdot \left(\mu_1^{-1} + (\mu_1 - \epsilon)^{-1} \cdot \mu_1^{-1}\cdot \rho_{\mathbf{y}}   \right) \cdot \epsilon.
    \end{split}
\end{equation}

For $\mathcal{E}_{ty3}$,
\begin{equation}
    \label{Error_y3}
    \|\mathcal{E}_{ty3}\|_2 = c_2 \cdot \|\nabla_{\mathbf{y}}f(\mathbf{x}_t,\mathbf{y}_t) - \nabla_{\mathbf{y}}\hat{f}(\mathbf{x}_t,\mathbf{y}_t)\|_2 \leq c_2 \cdot \epsilon.
\end{equation} Combining \cref{Error_x}, \cref{Error_y1}, \cref{Error_y2} and \cref{Error_y3}, we have
\begin{equation*}
\begin{split}
    \|\mathcal{E}_t\|_2 & \leq \eta_{\mathbf{x}} \cdot \left(\|\mathcal{E}_{tx}\|_2+\|\mathcal{E}_{ty1}\|_2+\|\mathcal{E}_{ty2}\|_2+\|\mathcal{E}_{ty3}\|_2\right)\\
    & \leq \eta_{\mathbf{x}} \cdot C_1 \cdot \epsilon,
\end{split}
\end{equation*}for a constant $C_1$ independent of $\epsilon$ (depends on constants defined in \cref{standard_assumption}). If $\epsilon$ is small enough such that
\begin{equation*}
\epsilon \leq \frac{ 1-\sigma}{\eta_{\mathbf{x}} \cdot C_1 \cdot w_2} \cdot r,
\end{equation*}then 
\begin{equation}
\label{bound_for_error}
    \|\mathcal{E}_t\| \leq w_2 \cdot \|\mathcal{E}_t\|_2 \leq w_2 \cdot \eta_{\mathbf{x}} \cdot C_1 \cdot \epsilon \leq (1-\sigma) \cdot r.
\end{equation}

Return back to \cref{update_stochatic}, with \cref{taylor_HFR} it can be rewritten as
\begin{equation}
    \mathbf{z}_{t+1} - \mathbf{z}^{*} = \mathbf{J}_{\text{HFR}} \cdot (\mathbf{z}_{t} - \mathbf{z}^{*}) + \mathcal{E}_t,
\end{equation}where $\mathbf{J}_{\text{HFR}}$ is the Jacobian matrix defined in \cref{taylor_HFR} and
\begin{equation*}
\|\mathbf{z}_{t+1} - \mathbf{z}^{*}\| \leq \sigma \cdot \|\mathbf{z}_{t} - \mathbf{z}^{*}\| + \|\mathcal{E}_t\|,
\end{equation*}if $\mathbf{z}_t \in \Omega(\mathbf{z}^{*})$.
Therefore, we can easily prove by induction that $\|\mathbf{z}_t - \mathbf{z}^{*}\|_2 \leq r$ for $t=0,\cdots, T$, if \cref{bound_for_error} holds and  $\|\mathbf{z}_0 - \mathbf{z}^{*}\| \leq w_1 r$. It implies that $\mathbf{z}_t \in \Omega(\mathbf{z}^{*})$ and
\begin{equation*}
\|\mathbf{z}_{t+1} - \mathbf{z}^{*}\| \leq \sigma \cdot \|\mathbf{z}_{t} - \mathbf{z}^{*}\| + \|\mathcal{E}_t\|,
\end{equation*}for $t=0,\cdots, T$. We just let $\Tilde{r} \leq w_1/w_2 \cdot r$, therefore, $\|\mathbf{z}_0 - \mathbf{z}^{*}\| \leq w_2 \cdot \|\mathbf{z}_0 - \mathbf{z}^{*}\|_2 \leq w_2 \cdot \Tilde{r} \leq w_1 r$. Based on above analysis, we have
\begin{equation*}
    \begin{split}
        \|\mathbf{z}_T - \mathbf{z}^{*}\| & \leq \sigma \cdot \|\mathbf{z}_{T-1} - \mathbf{z}^{*}\| + \|\mathcal{E}_{T-1}\| \\
        & \leq \cdots \\
        & \leq \sigma^T \cdot \|\mathbf{z}_{0} - \mathbf{z}^{*}\| + \sigma^{T-1} \cdot \|\mathcal{E}_{0}\| + \sigma^{T-2} \cdot \|\mathcal{E}_{1}\|  + \cdots + \sigma \cdot \|\mathcal{E}_{T-2}\| + \|\mathcal{E}_{T-1}\|
    \end{split}
\end{equation*}and
\begin{equation*}
    \begin{split}
        \|\mathbf{z}_T - \mathbf{z}^{*}\|_2 & \leq \sigma^T \cdot \frac{w_2}{w_1} \cdot \|\mathbf{z}_0 - \mathbf{z}^{*}\|_2 + \frac{1-\sigma^T}{1-\sigma} \cdot \frac{w_2}{w_1} \cdot \eta_{\mathbf{x}} \cdot C_1 \cdot \epsilon \\
        & =: \sigma^T \cdot \frac{w_2}{w_1} \cdot \|\mathbf{z}_0 - \mathbf{z}^{*}\|_2 + C \cdot \frac{1-\sigma^T}{1-\sigma} \cdot \epsilon,
    \end{split}
\end{equation*}where $C = \eta_{\mathbf{x}} \cdot C_1 \cdot w_2/w_1$ is a constant independent of $\epsilon$ and $\sigma$.

\end{proof}

\subsection{Computation of the Hessian Inverse}
\label{comput_hessian_inverse}
The proposed HessianFR algorithm requires the computation of Hessian inverse, as is shown in \cref{alg:HessianFR} and Algorithm 3.2. However, exact computation of Hessian inverse is computationally expensive and is infeasible in many applications. In numerical linear algebra, several methods are popular in approximating Hessian inverse, e.g., iterative methods and low rank approximation methods. We will discuss them as follows. 

\subsubsection{Conjugate Gradient (CG) Method}
Both Wang et al. \cite{Wang2020On} and Zhang et al. \cite{zhang2020newton} adopted conjugate gradient (CG) method in FR and GDN to approximate the Hessian inverse due to the mathematical guarantees. The convergence of CG for positive definite matrices are well studied \cite{saad2003iterative}. To guarantee the positive definiteness of the Hessian matrix, besides adding an $L_2$ regularization term in the loss function, we approximately solve the following linear system: 
\begin{equation}
    \mathbf{H}_{\mathbf{y y}}^2 \mathbf{b} = \mathbf{H}_{\mathbf{y y}}(c_2 \nabla_{\mathbf{y}}f(\mathbf{x}_t,\mathbf{y}_t)-\mathbf{H}_{\mathbf{y x}}\nabla_{\mathbf{x}}f(\mathbf{x}_t,\mathbf{y}_t)),
\end{equation}
where we regard $\mathbf{b}$ as an approximation to 
\begin{equation}
    c_2 \mathbf{H}_{\mathbf{y y}}^{-1}\nabla_{\mathbf{y}}f(\mathbf{x}_t,\mathbf{y}_t)-\mathbf{H}_{\mathbf{y y}}^{-1}\mathbf{H}_{\mathbf{y x}}\nabla_{\mathbf{x}}f(\mathbf{x}_t,\mathbf{y}_t).
\end{equation}
Moreover, instead of obtaining the full Hessian matrix, CG requires to calculate the Hessian vector product which is approximately of two times computation than back-propagation (calculation of gradients) in deep learning. 

In numerical parts, we also compare HessianFR with different numbers of CG updates for approximating Hessian inverse. We find that larger numbers of CG updates in each steps contribute to better approximation of Hessian inverse, but require much more computation.

\subsubsection{Diagonal Method (DG)} CG is computationally expensive compared with first order methods (e.g., GDA). We hope to develop a fast HessianFR algorithm that can solve large scale problems (in terms of both size of dataset and parameters) with comparable computational costs as GDA. Recall that the initialization for LBFGS \cite{liu1989limited} in minimization problem $\min_{\mathbf{y}}f(\mathbf{y})$ is as follows:
\begin{equation}
    \mathbf{H}_{0}^{-1} = \frac{<\nabla_{\mathbf{y}}f(\mathbf{y_t})-\nabla_{\mathbf{y}}f(\mathbf{y_{t-1}}),\mathbf{y}_t - \mathbf{y}_{t-1}>}{\|\nabla_{\mathbf{y}}f(\mathbf{y_t})-\nabla_{\mathbf{y}}f(\mathbf{y_{t-1}})\|_2^2} \cdot \mathbf{I},
\end{equation}which is computationally cheap. We directly adopt the same way to approximate the Hessian inverse in practice, i.e., the Hessian matrix w.r.t. $\mathbf{y}$ at $\mathbf{z}_t=(\mathbf{x}_t,\mathbf{y}_t)$ is approximated by
\begin{equation}
\label{diagonal_mathod}
    \mathbf{H}_{\mathbf{y y}}^{-1} \approx \frac{<\nabla_{\mathbf{y}}f(\mathbf{x}_t,\mathbf{y_t})-\nabla_{\mathbf{y}}f(\mathbf{x}_t, \mathbf{y_{t-1}}),\mathbf{y}_t - \mathbf{y}_{t-1}>}{\|\nabla_{\mathbf{y}}f(\mathbf{x}_t,\mathbf{y_t})-\nabla_{\mathbf{y}}f(\mathbf{x}_t,\mathbf{y_{t-1}})\|_2^2} \cdot \mathbf{I}.
\end{equation} Note that it is a very rough estimate for Hessian inverse $\mathbf{H}_{\mathbf{y y}}^{-1}$. Numerically, we find that it performs worse than CG (even with one time update), however, they are comparable in stochastic problems. It is reasonable that exact computations of Hessian inverse may not make sense because we introduce stochasticity and similar results are also found in \cite{yao2020adahessian, xu2020second} and \cite{schaul2013no}. Recently, Yao et al. \cite{yao2020adahessian} proposed a Hessian based gradient descent method which approximates Hessian by a diagonal matrix using Hutchinson’s method \cite{bekas2007estimator}. However, it is still computationally expensive because each steps of Hutchinson’s method requires several times of Hessian vector products and diagonal matrix may not approximate Hessian matrix well if the Hessian matrix is not diagonally dominated. Both Hutchinson’s method and CG requires the computation of Hessian vector product but Hutchinson’s method merely approximate the diagonal of the Hessian while CG approximates the whole matrix. Therefore, it does not make sense to use Hutchinson’s method to approximate diagonal of Hessian in the proposed algorithm. For fast computation, we adopt the rough estimate \cref{diagonal_mathod} for Hessian inverse.

\section{Experiments}
\label{numerical_experiments}
In this part, we conduct experiments on toy examples (deterministic optimization) and large scale datasets (stochastic optimization) to demonstrate the performance of the proposed HessianFR algorithm. Please refer to the supplementary material for experimental details. Visualized results are displayed in \cref{vis_cifar10_gen_data}, \cref{vis_mnist_gen_data} and \cref{vis_celeba_gen_data}.

\subsection{Toy Examples: Low Dimensional Examples}
To show the superiority of local minimax based algorithms, we first compare FR, GDN as well as the proposed HessianFR algorithm with gradient descent ascent (GDA), optimistic GDA (OGDA) \cite{daskalakis2018training}, extra gradient (EG) \cite{korpelevich1976extragradient}, symplectic gradient adjustment (SGA) \cite{balduzzi2018mechanics} and consensus optimization (CO) \cite{mescheder2017numerics} in solving low dimentional toy examples. 

Consider the min-max problem of the following three two-dimensional functions:
\begin{equation}
    g_{1}(x, y)=-3 x^{2}-y^{2}+4 x y,
\end{equation}
\begin{equation}
    g_{2}(x, y)=3 x^{2}+y^{2}+4 x y,
\end{equation} 
and 
\begin{equation}
    g_{3}(x, y)=\left(4 x^{2}-\left(y-3 x+0.05 x^{3}\right)^{2}-0.1 y^{4}\right) \cdot e^{-0.01\left(x^{2}+y^{2}\right)}.
\end{equation}

Trajectory of those algorithms are available in \cref{low_dim}. For $g_1$, GDA, OGDA and EG diverge from $(0,0)$ which is a local minimax point. $(0,0)$ is not a local minimax point of $g_2$, however, only HessianFR, FR and GDN avoid convergence to the undesired critical point. In $g_3$, only local minimax based algorithms (HessianFR, FR and GDN) converge to the local minimax point $(0,0)$ while all other algorithms suffer from cycling. Numerical results here suggest that local minimax based algorithms outperform GDA and its variants. Even in those simple cases, GDA and its variants (e.g., OGDA, EG, GDA and CO) either fail to converge or converge to undesirable points while local minimax based algorithms (e.g., HessianFR, FR and GDN) exhibit outstanding performances.

\begin{figure*}
  \centering
  \subfloat[GDA, OGDA and EG diverge.] {
     \label{appendix_low_dim:a}     
    \includegraphics[width=0.6\textwidth]{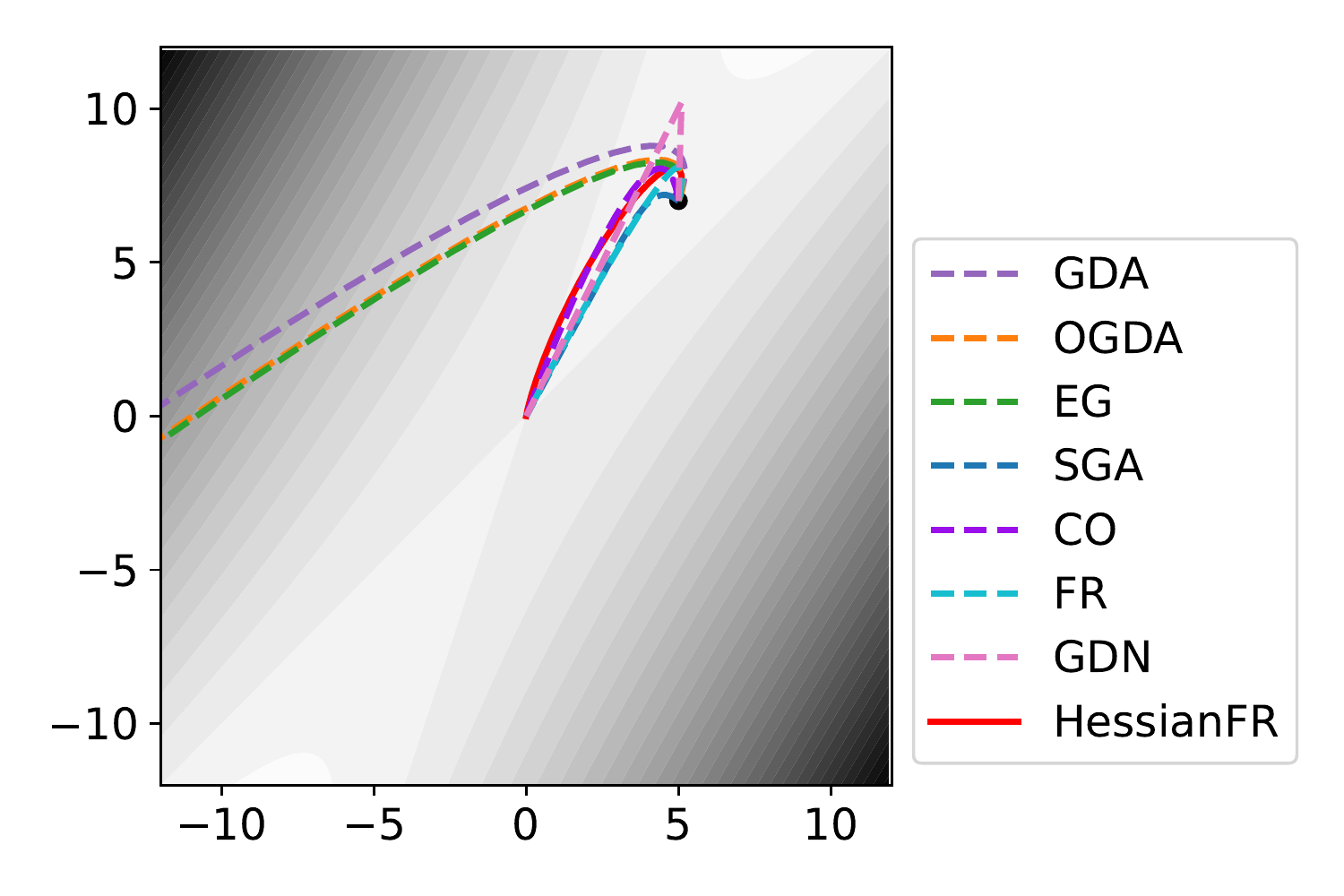}
    }\   
    \subfloat[Only HessianFR, FR and GDN escape from non-local minimax.] {
    \label{appendix_low_dim:b}     
    \includegraphics[width=0.6\textwidth]{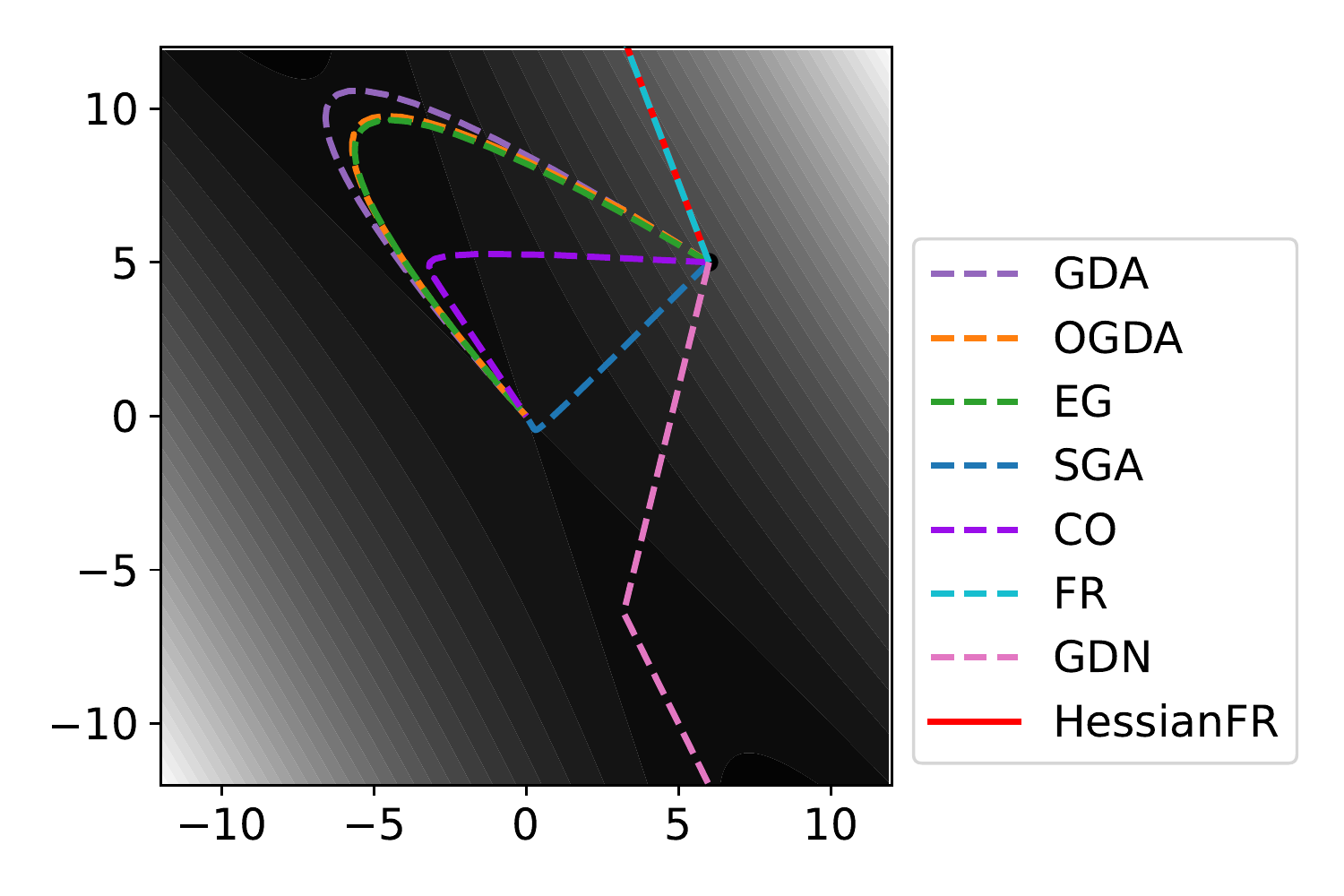}
    }\
    \subfloat[Only HessianFR, FR and GDN converge without cycling.]{
    \label{appendix_low_dim:c}   
    \includegraphics[width=0.6\textwidth]{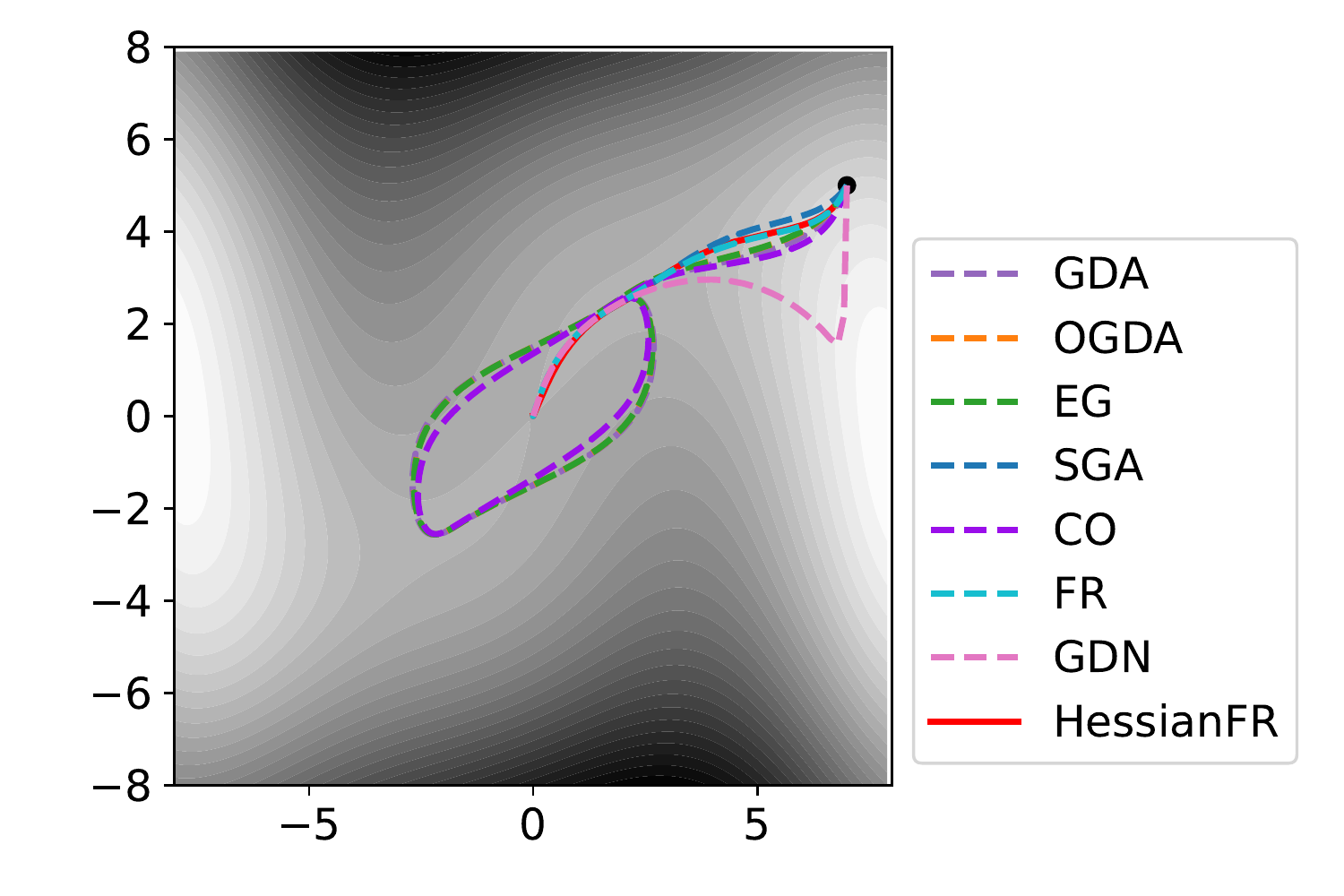}
    }\
    \caption{Trajectory of HessianFR and other algorithms in the three low-dimensional toy examples. \textbf{Left:} for $g_1$, $(0,0)$ is a local minimax. $\textbf{Middle:}$ for $g_2$, $(0,0)$ is \textbf{NOT} a local minimax but algorithms except HessianFR, FR as well as GDN converge to the non-local minimax point. \textbf{Right:} for $g_3$, $(0,0)$ is a local minimax while algorithms denoted in dashed lines are not convergent and  cycling around.}
    \label{low_dim}
    \vskip -0.1in
\end{figure*}

\subsection{Toy Examples: GANs for Synthetic Dataset}
We extend simple low dimensional examples to a more challenging task that trains generative adversarial networks on a synthetic dataset (mixtures of gaussian).  The training data is sampled independently from $\frac{1}{3}\mathcal{N}(-4, 0.3^2)+\frac{1}{3}\mathcal{N}(0, 0.3^2)+\frac{1}{3}\mathcal{N}(4, 0.3^2)$, which is the target distribution of the task. The source distribution, which is used for generation, is 3-dimensional standard normal distribution (i.e., $\mathcal{N}(\mathbf{0},\mathbf{I}_3)$).  Vanilla GAN (JS-GAN proposed by Goodfellow et al. \cite{goodfellow2014generative}) with fully connect networks is adopted here and we add $L_2$ regularization term to discriminators for training stability. We evaluate the convergence of HessianFR, FR, GDN, GDA, GDA-2 and EG in terms of gradient norms for both generators and discriminators. We first use GDA-2 (with learning rates $\eta_{\mathbf{x}}=\text{1E-04}$ for generators and $\eta_{\mathbf{y}}=\text{1E-04}$ for discriminators) to pretrain the model for 10K iterations and then run above algorithms (with fine tuned learning rates) for 100K iterations.   For more details about experimental settings and contents, please see the supplementary material.

Convergence rates (in terms of gradient norms) w.r.t. iteration steps and time for both generators and discriminators are shown in \cref{mixg1d_gen_disc_grad_itr} and \cref{appendix_mixg1d_gen_disc_grad_time} respectively. Here, we adopt 5 CG updates in each step for HessianFR-CG and FR-CG.  We have the following observations. (1) Local minimax based algorithms, e.g., HessianFR, FR and GDN  (both CG and DG based), achieve much better convergence than GDA (and its variants). It further demonstrates the advantages of local minimax than local Nash equilibrium in training GANs (differentiable sequential games). (2) GDN-CG is a little bit slower than HessianFR-CG and FR-CG mainly because it requires strict pretraining. (3) HessianFR-CG outperforms all other algorithms in terms of iteration steps for convergence while HessianFR-DG converges with least seconds. The success of HessianFR (and FR) is mainly due to the correction term and the convergence is further sped up with the Hessian information than FR.

\begin{figure}[htb]
  \centering
  \subfloat[Generators.] {
     \label{mixg1d_gen_disc_grad_itr:a}     
    \includegraphics[width=0.47\textwidth]{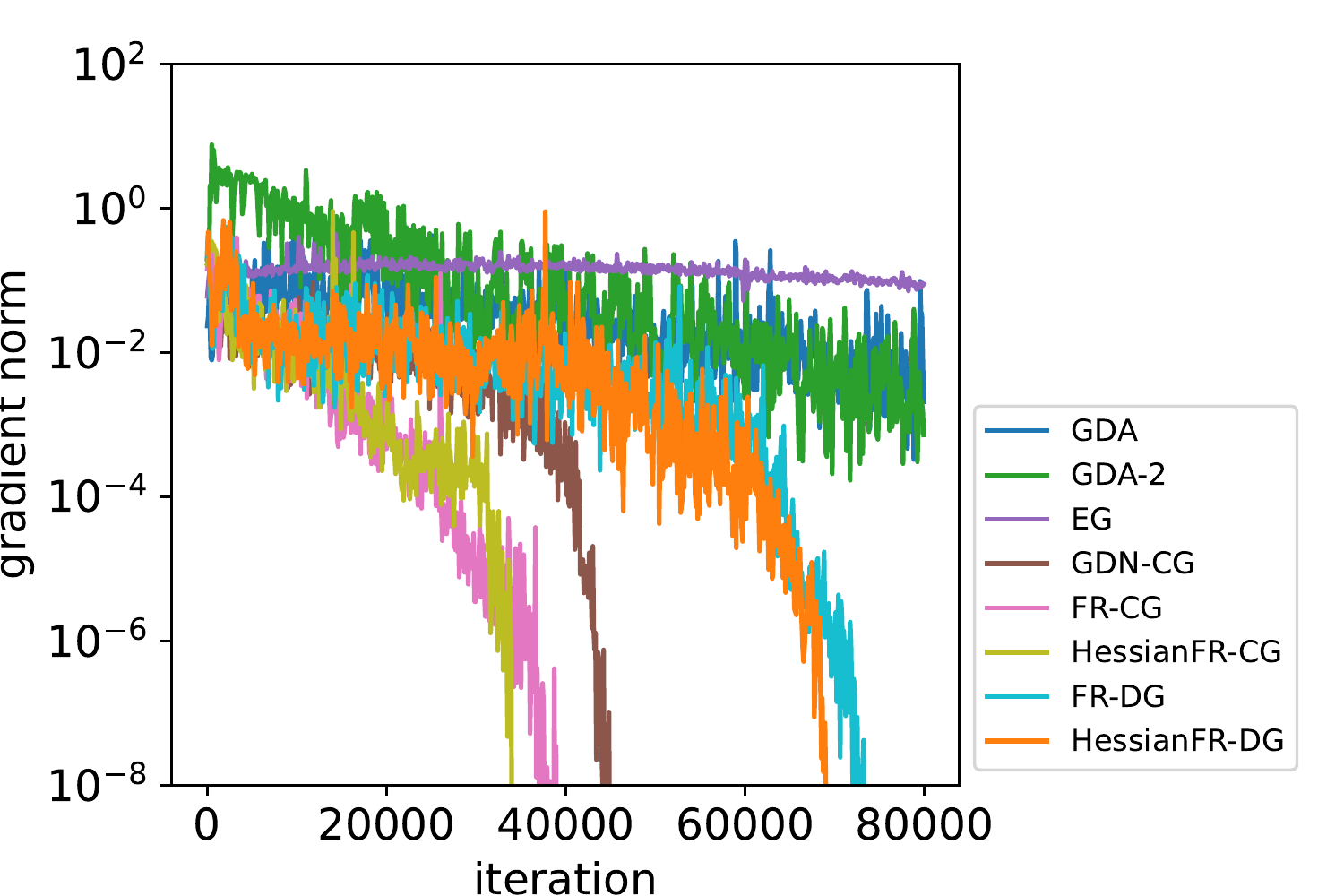}
    }\   
    \subfloat[Discriminators.] {
    \label{mixg1d_gen_disc_grad_itr:b}     
    \includegraphics[width=0.47\textwidth]{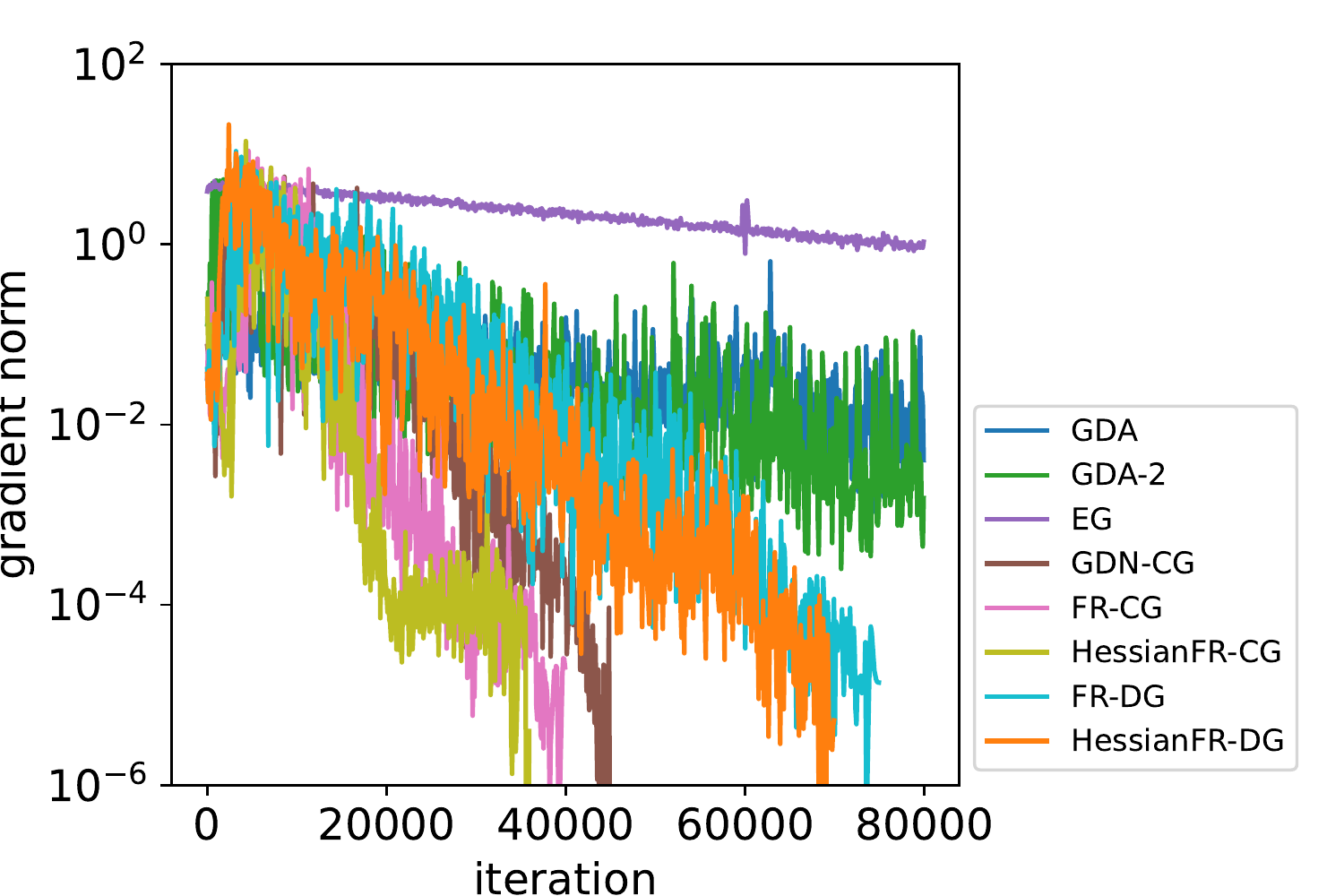}
    }\\
    \caption{Comparison of different algorithms on 1-d Mixture of Gaussian data (gradient norms) in terms of iteration steps. Gradient norms of generators and discriminators during training.}
    \label{mixg1d_gen_disc_grad_itr}
    \vskip -0.2in
\end{figure}

\begin{figure}[htb]
  \centering
  \subfloat[Generators.] {
     \label{appendix_mixg1d_gen_disc_grad_time:a}     
    \includegraphics[width=0.47\textwidth]{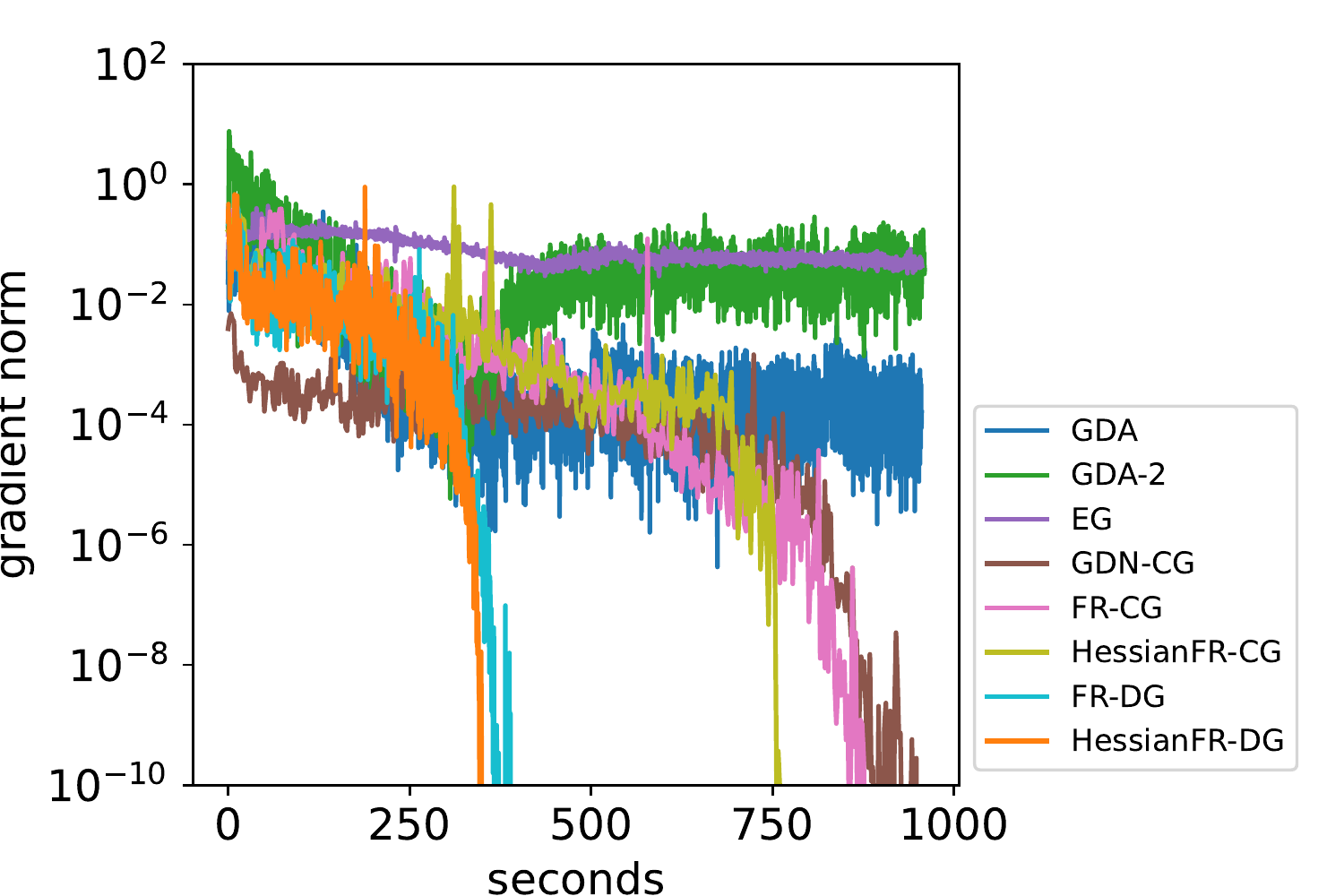}
    }\   
    \subfloat[Discriminators.] {
    \label{appendix_mixg1d_gen_disc_grad_time:b}     
    \includegraphics[width=0.47\textwidth]{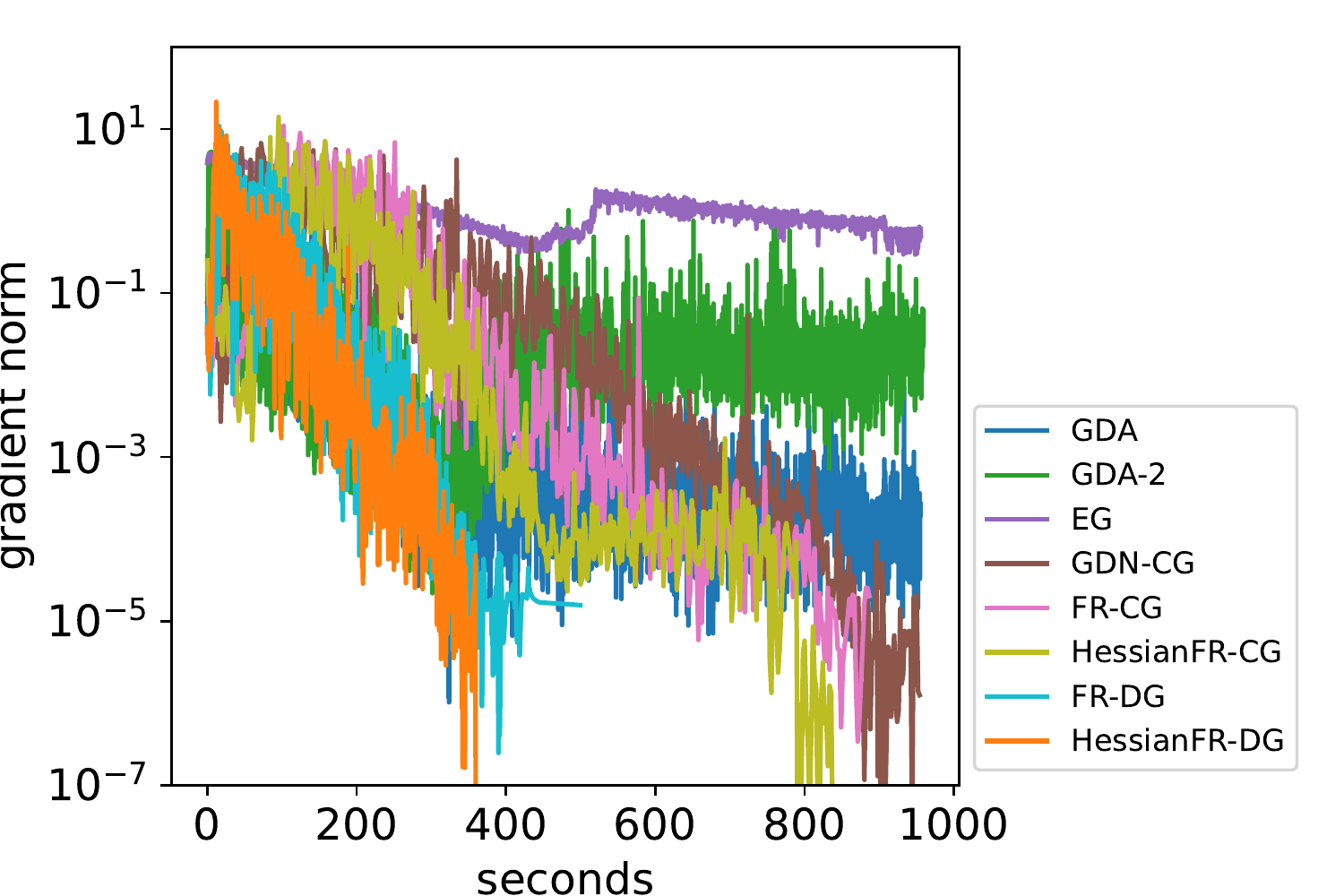}
    }\\
    \caption{Comparisons of different algorithms on 1-d Mixture of Gaussian data (gradient norms) in terms of time. Gradient norms of generators and discriminators during training.}
    \label{appendix_mixg1d_gen_disc_grad_time}
\end{figure}

\begin{figure}[htb]
  \centering
  \subfloat[Generators.] {
     \label{appendix_mixg1d_HessianFR_gen_disc_grad_itr:a}     
    \includegraphics[width=0.47\textwidth]{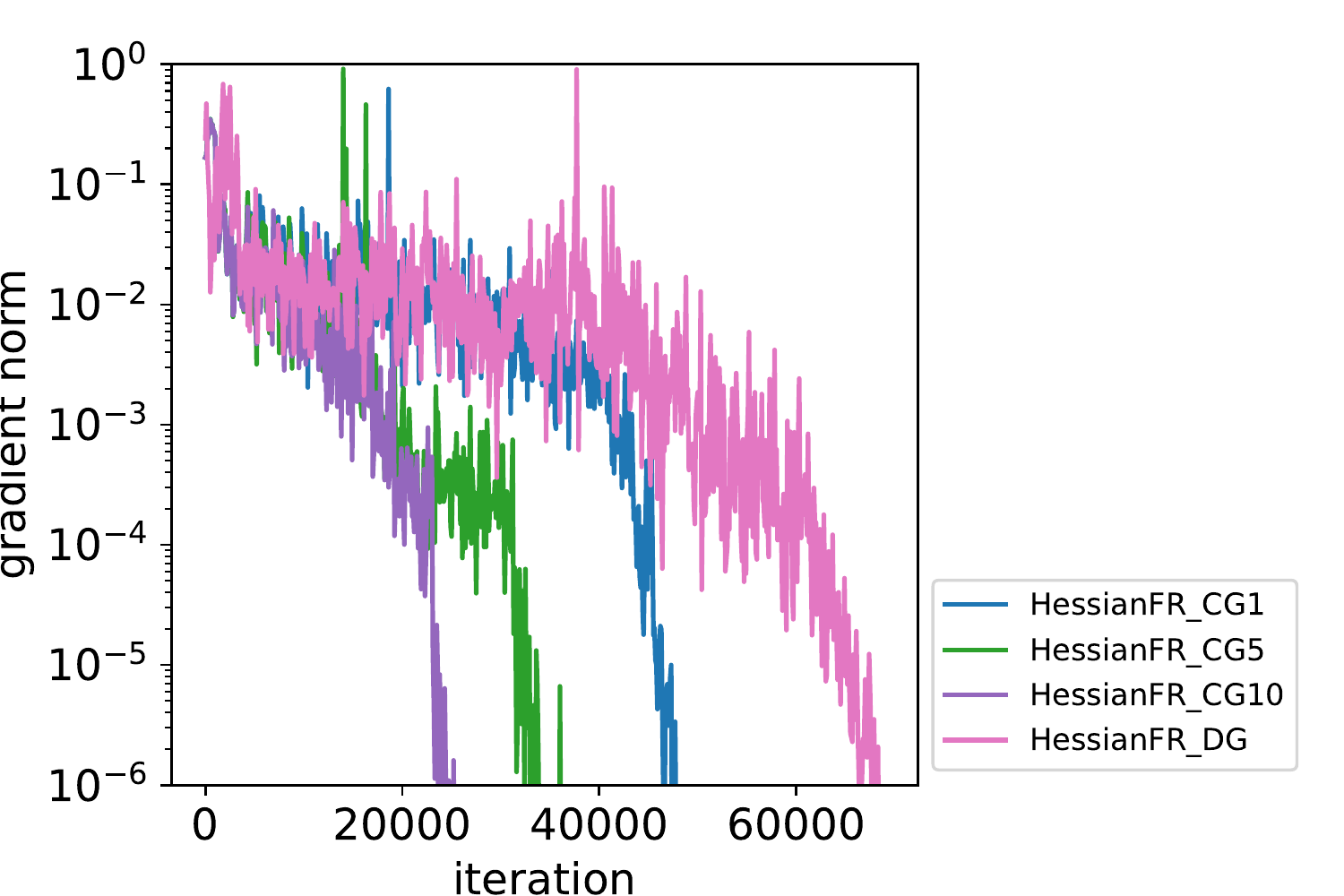}
    }\   
    \subfloat[Discriminators.] {
    \label{appendix_mixg1d_HessianFR_gen_disc_grad_itr:b}     
    \includegraphics[width=0.47\textwidth]{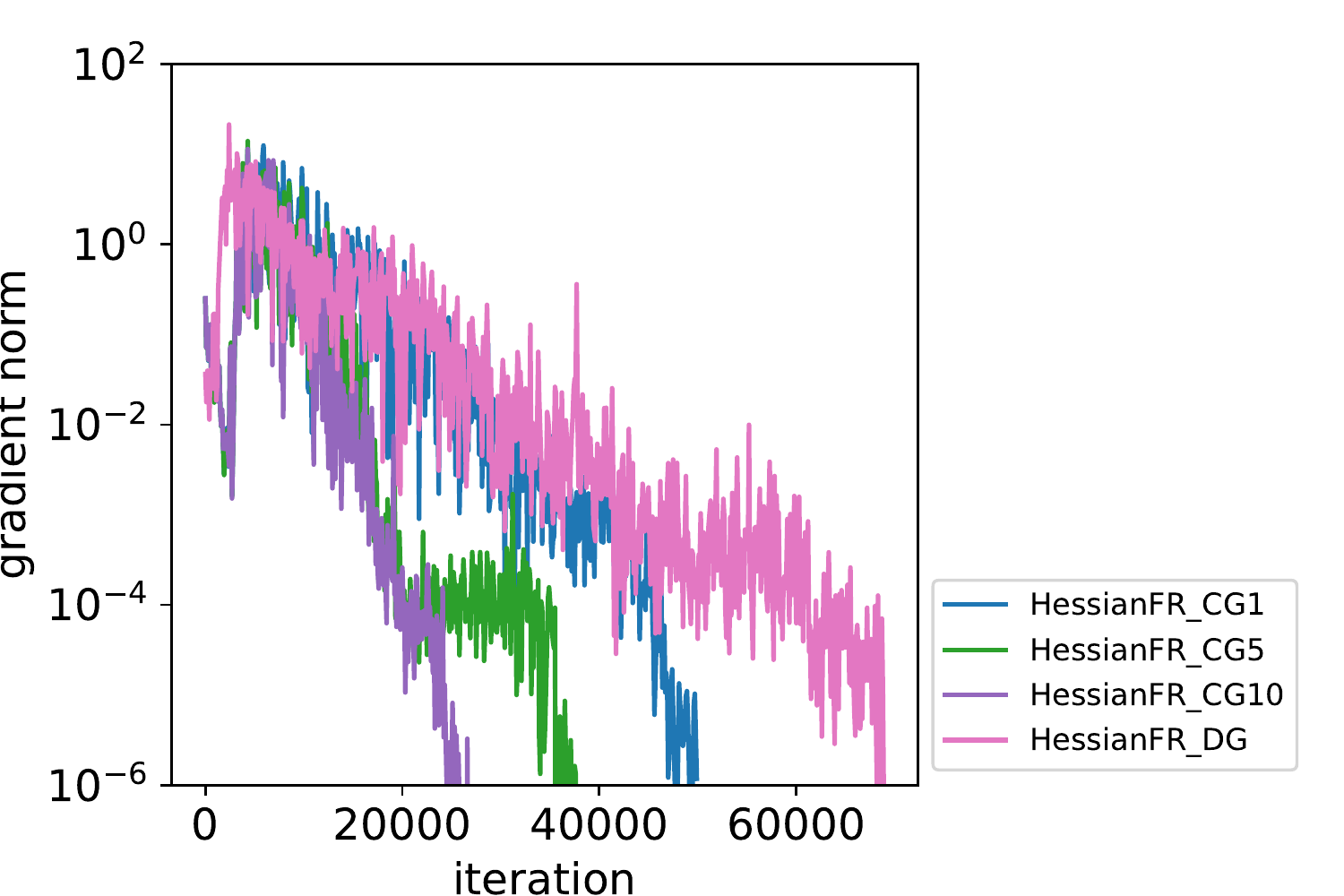}
    }\\
    \caption{Comparisons of HessianFR (with different methods in approximating inverse of Hessian matrix) on 1-d Mixture of Gaussian data (gradient norms) in terms of iteration steps. Gradient norms of generators and discriminators during training.}
    \label{appendix_mixg1d_HessianFR_gen_disc_grad_itr}
\end{figure}

\begin{figure}[htb]
  \centering
  \subfloat[Generators.] {
     \label{appendix_mixg1d_HessianFR_gen_disc_grad_time:a}     
    \includegraphics[width=0.47\textwidth]{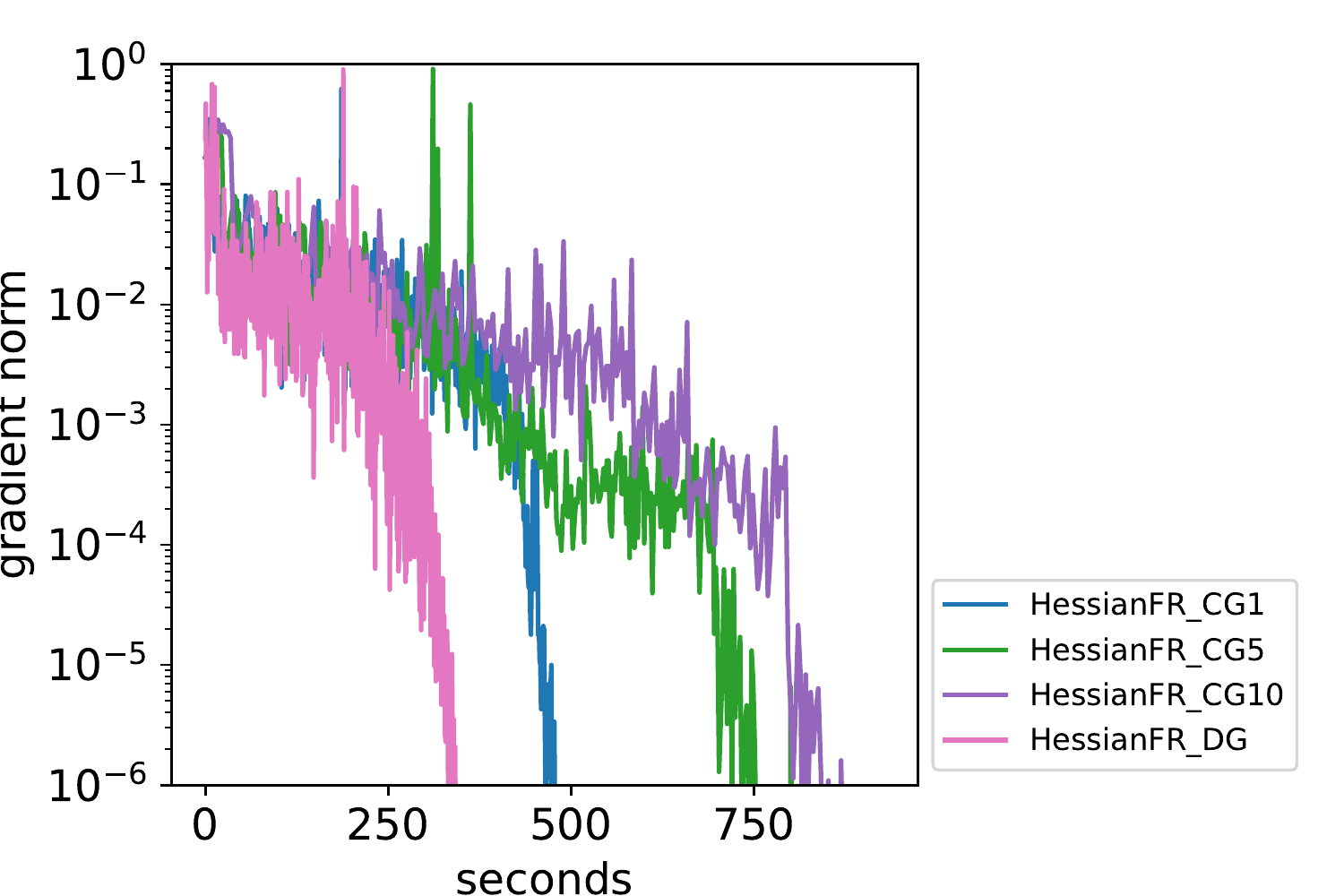}
    }\   
    \subfloat[Discriminators.] {
    \label{appendix_mixg1d_HessianFR_gen_disc_grad_time:b}     
    \includegraphics[width=0.47\textwidth]{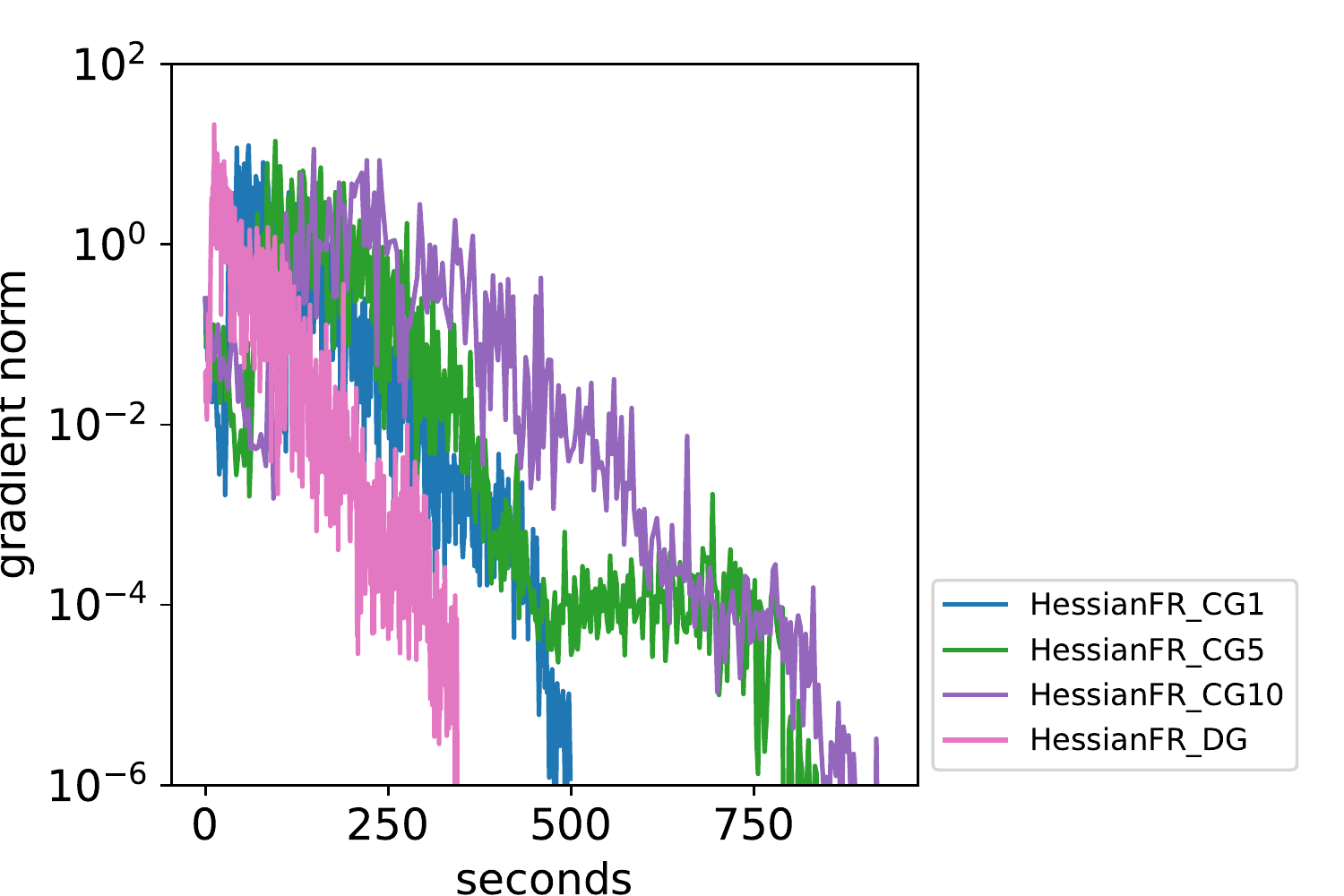}
    }\\
    \caption{Comparisons of HessianFR (with different methods in approximating inverse of Hessian matrix) on 1-d Mixture of Gaussian data (gradient norms) in terms of time. Gradient norms of generators and discriminators during training.}
    \label{appendix_mixg1d_HessianFR_gen_disc_grad_time}
\end{figure}

We further numerically investigate the convergence of HessianFR with different methods to approximate the Hessian inverse. We compare the Diagonal method (DG) and the Conjugate Gradient method (CG1, CG5 and CG10). Here, CG5 represents 5 CG updates in each step and similar settings hold for CG1 and CG10. Intuitively, more CG steps imply more accurate approximations to Hessian inverse but are more computationally expensive while DG is the most rough estimation but is much cheaper. Results displayed in  \cref{appendix_mixg1d_HessianFR_gen_disc_grad_itr} and \cref{appendix_mixg1d_HessianFR_gen_disc_grad_time} validate our expectation. Diagonal method serves the worst approximation to Hessian inverse that HessianFR-DG requires much more iterations for convergence. After all, DG has only one freedom in approximation but CG makes use of the  Hessian information (Hessian-vector products). More steps for CG result in better approximation to Hessian inverse and thus lead to fewer iterations for convergence. However, CG steps are computationally expensive. \cref{appendix_mixg1d_HessianFR_gen_disc_grad_time} shows that even though DG requires more iterations for convergence, it is the most time saving. Therefore, we need to consider the trade off between the accuracy and computational costs in practice. 

Similar experiments as mixtures of gaussian are conducted for swiss roll data. Please refer to the supplementary materials for details. HessianFR outperforms other algorithms in terms of iterations (HessianFR-CG) and seconds (HessianFR-DG) for convergence.

\subsection{Large-scale Datasets}
To further show the performance of the proposed algorithm in real applications, we apply it to train generative adversarial networks on large scale datasets (e.g., MNIST \cite{lecun1998mnist}, CIFAR-10 \cite{krizhevsky2009learning} and CelebA \cite{liu2015faceattributes}). Here, for higher quality generation, we adopt WGAN-GP \cite{gulrajani2017improved} with standard CNNs and ResNet architectures. Fréchet inception distance (FID) \cite{heusel2017gans} is utilized to measure the quality of generated images and lower scores imply better generation. Learning rates for all algorithms are fine tuned. We recommend learning rates to be $\eta_{\mathbf{x}}=\text{2E-05}$ for generators and $(\eta_{\mathbf{y}1},\eta_{\mathbf{y}2})=(\text{5E-05,2E-05})$ for discriminators. 
We put all experimental details in the supplementary material.

As reported in \cref{fid_all_dataset}, the proposed HessianFR algorithm (with DG and CG1) outperforms GDA-5 (which is the most popular algorithm in GANs training), EG (an extension of GDA) and FR in terms of image generation quality (FID), in several tasks of large scale stochastic (mini-batch) learning. HessianFR-CG5 is much more computationally expensive than others and the accurate computation of Hessian inverse is weaken by the stochasticity. Note that the computational cost of HessianFR-DG is comparable to that of GDA-5 while HessianFR-CG1 is less than twice that of GDA-5.  \cref{wgan_fid_seconds} and  \cref{wgan_fid_seconds2} display the trajectory of FID scores during training. HessianFR-DG and HessianFR-CG1 are comparable but both outperform other baselines. Some generated images are shown in \cref{vis_cifar10_gen_data}, \cref{vis_mnist_gen_data} and \cref{vis_celeba_gen_data}.

\begin{table}[t]
\caption{FID scores of models optimized by different optimizers (within 100K iterations) on MNIST, CIFAR-10 and CelebA datasets. Here, ``Time" indicates the running time (seconds) of 100 iterations}
\label{fid_all_dataset}
\begin{center}
\small
\begin{tabular}{l|cc|cc|cc|cc}
\hline
& \multicolumn{2}{c|}{MNIST (CNN)} & \multicolumn{2}{c|}{CIFAR-10 (CNN)} & \multicolumn{2}{c|}{CIFAR-10 (ResNet)} & \multicolumn{2}{c}{CelebA (ResNet)}    \\ \hline
Optimizer & FID  & Time & FID &  Time & FID  &  Time  & FID  &  Time   \\
\hline
GDA-5 & 5.04 & 5.50 & 28.43 & 17.40 & 23.48 & 44.81 & 19.34  & 98.74\\
EG & 4.83 & 3.82 & 33.01 & 12.85 & 25.38 & 29.32 & 24.45 & 68.79\\
FR-DG & 4.74 & 4.68 & 27.01 & 14.96 & 22.42  & 38.37& 13.32  & 80.44 \\
FR-CG1 & 4.85 & 13.14 & \textbf{26.38} & 34.80 & \underline{19.68} & 88.70& \underline{11.41} & 174.22\\
HessianFR-DG & 4.79 & 4.68 &  26.59 & 14.96& 20.83 & 38.37 & 12.34  & 80.44\\
HessianFR-CG1 & \textbf{4.65} & 13.14  & \underline{26.44} & 34.80 & \textbf{18.12} & 88.70 & \textbf{10.99} & 174.22 \\
HessianFR-CG5 & \underline{4.68} & 27.86 & - & - & - & - & - & - \\
\hline
\end{tabular}
\end{center}
\end{table}

\begin{figure}[htb]
    \centering
    \includegraphics[width=0.7\textwidth]{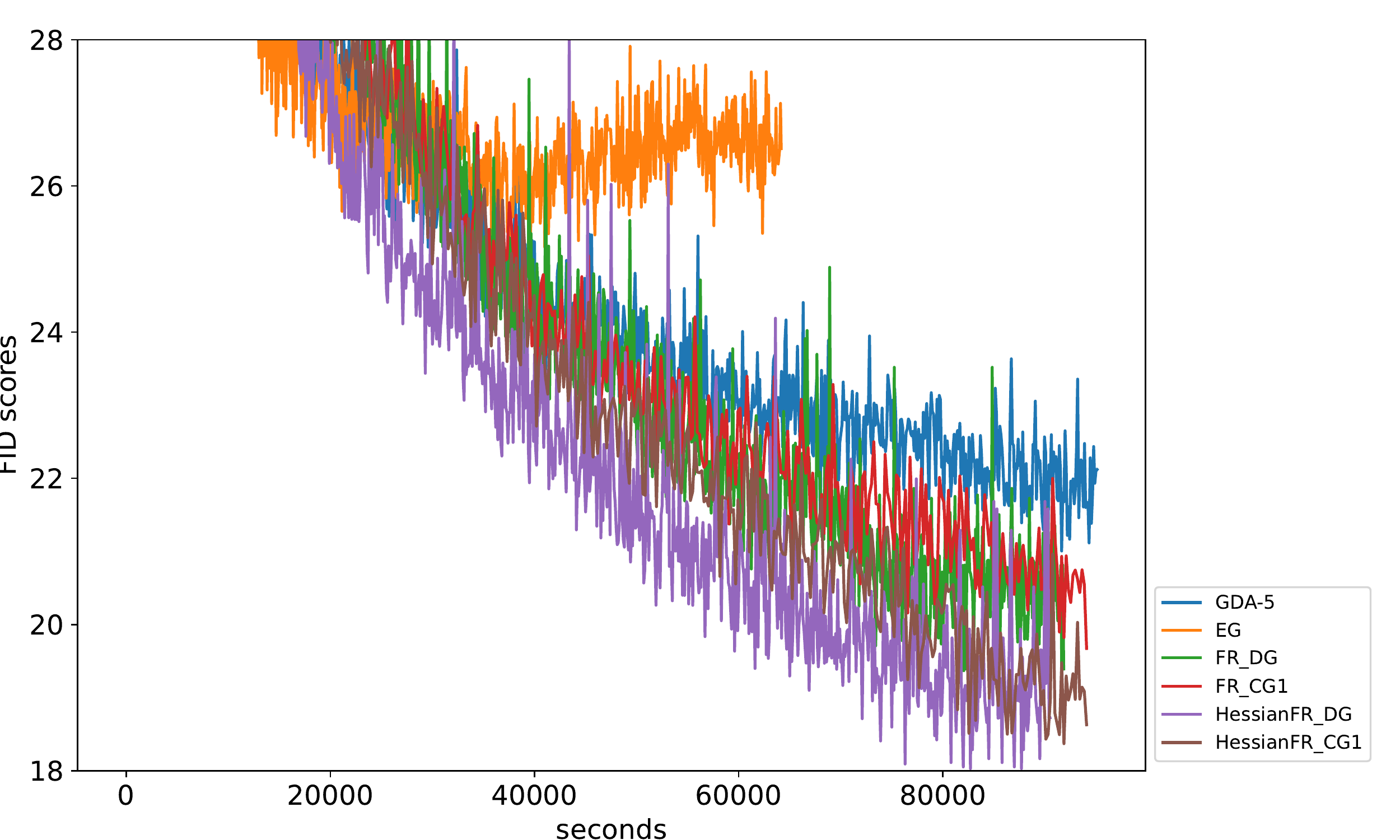}
    \caption{Comparison of different algorithms on CIFAR-10 data (ResNet) in terms of seconds. The trajectory of FID scores during training.}
    \label{wgan_fid_seconds}
\end{figure}

\begin{figure}[htb]
  \centering
 \includegraphics[width=0.7\textwidth]{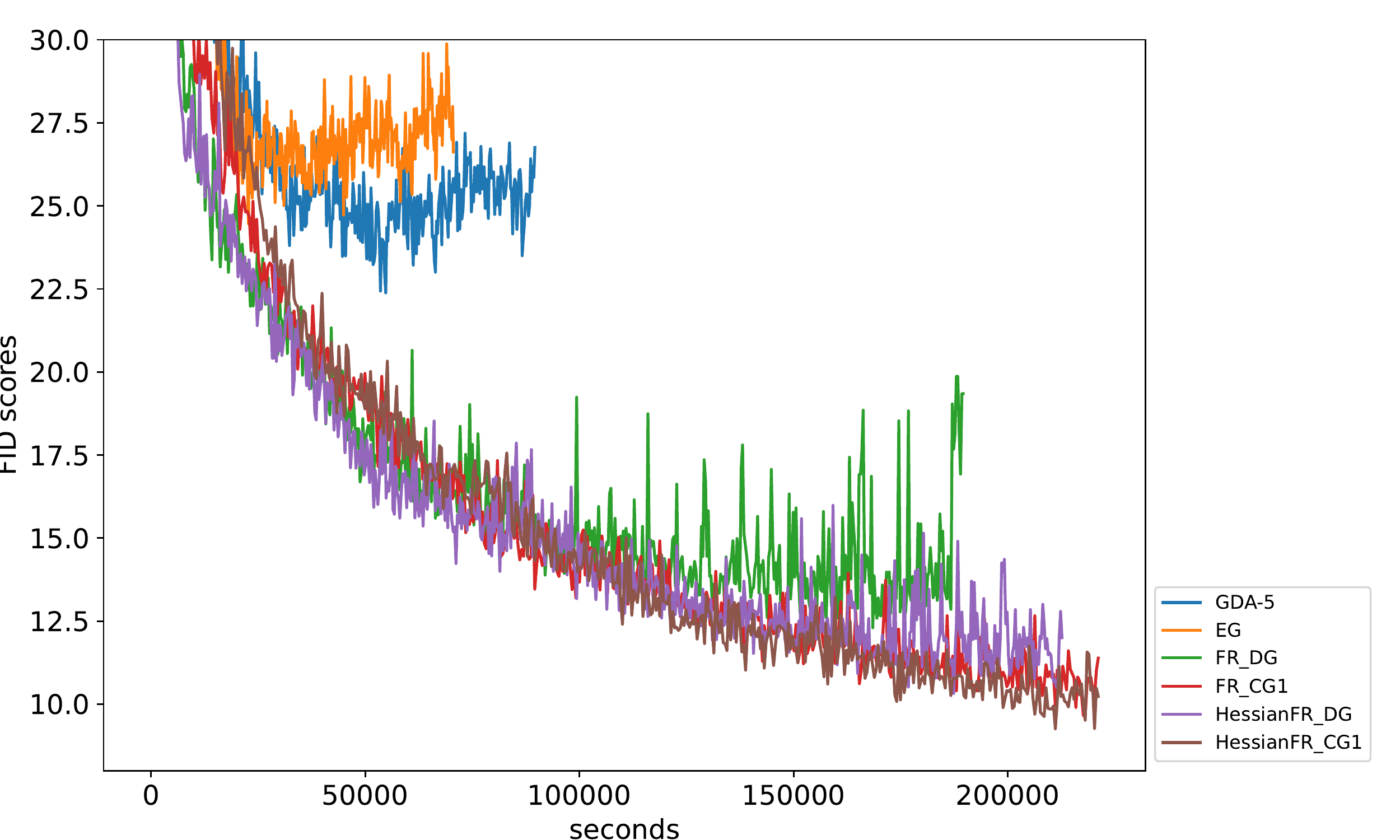}
    \caption{Comparison of different algorithms on CelebA data (ResNet) in terms of seconds. The trajectory of FID scores during training.}
    \label{wgan_fid_seconds2}
\end{figure}

\begin{figure*}[h!]
  \centering
  \subfloat[HessianFR-DG (Standard CNN).] {
     \label{vis_cifar10_gen_data:a}     
    \includegraphics[width=0.45\textwidth]{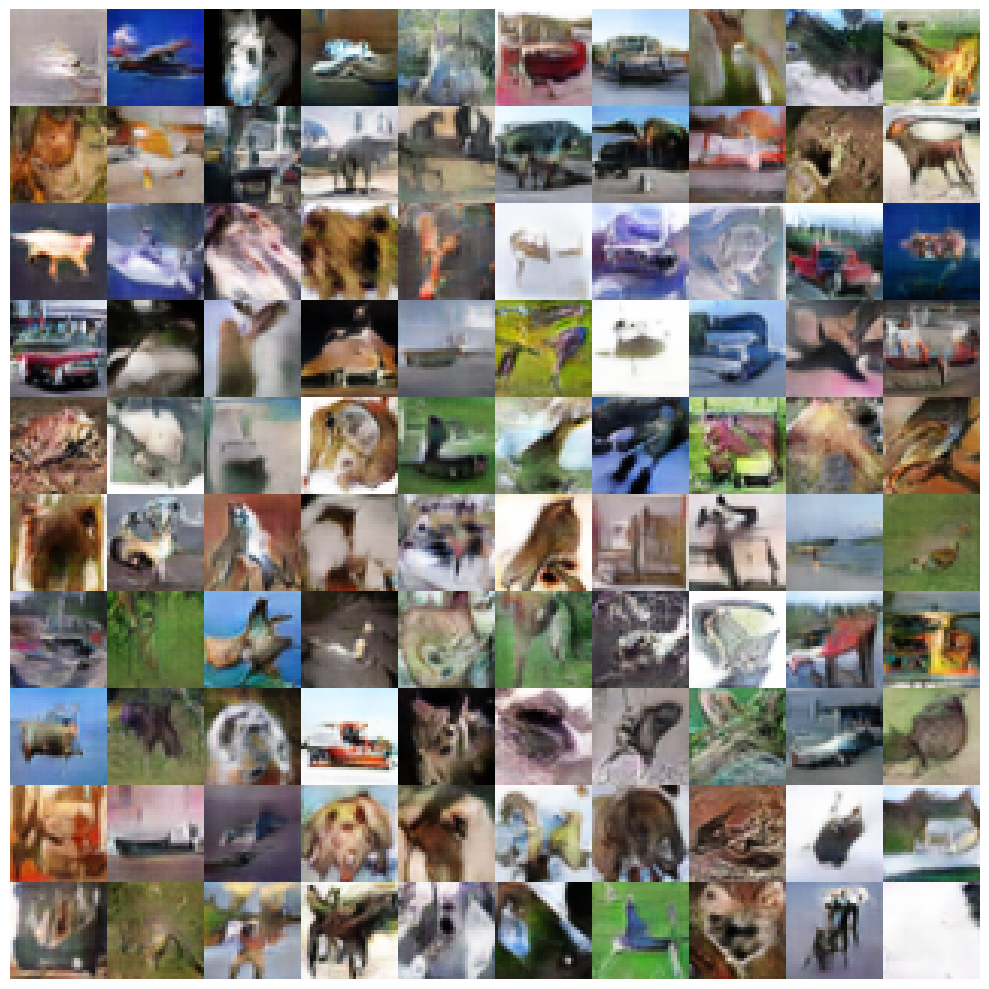}
    }
    \hfill
    \subfloat[HessianFR-CG1 (Standard CNN).] {
    \label{vis_cifar10_gen_data:b}     
    \includegraphics[width=0.45\textwidth]{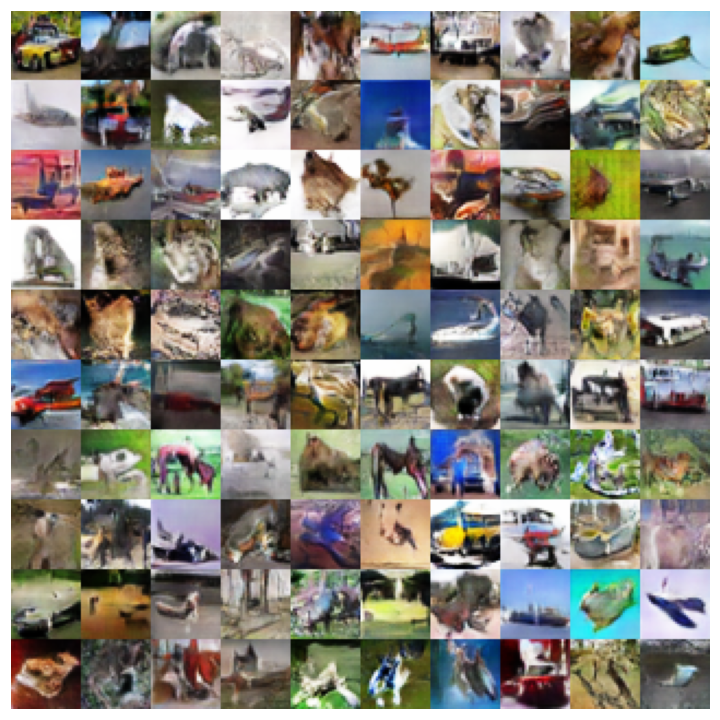}
    }
    \newline
    \subfloat[HessianFR-DG (ResNet).]{
    \label{vis_cifar10_gen_data:c}   
    \includegraphics[width=0.45\textwidth]{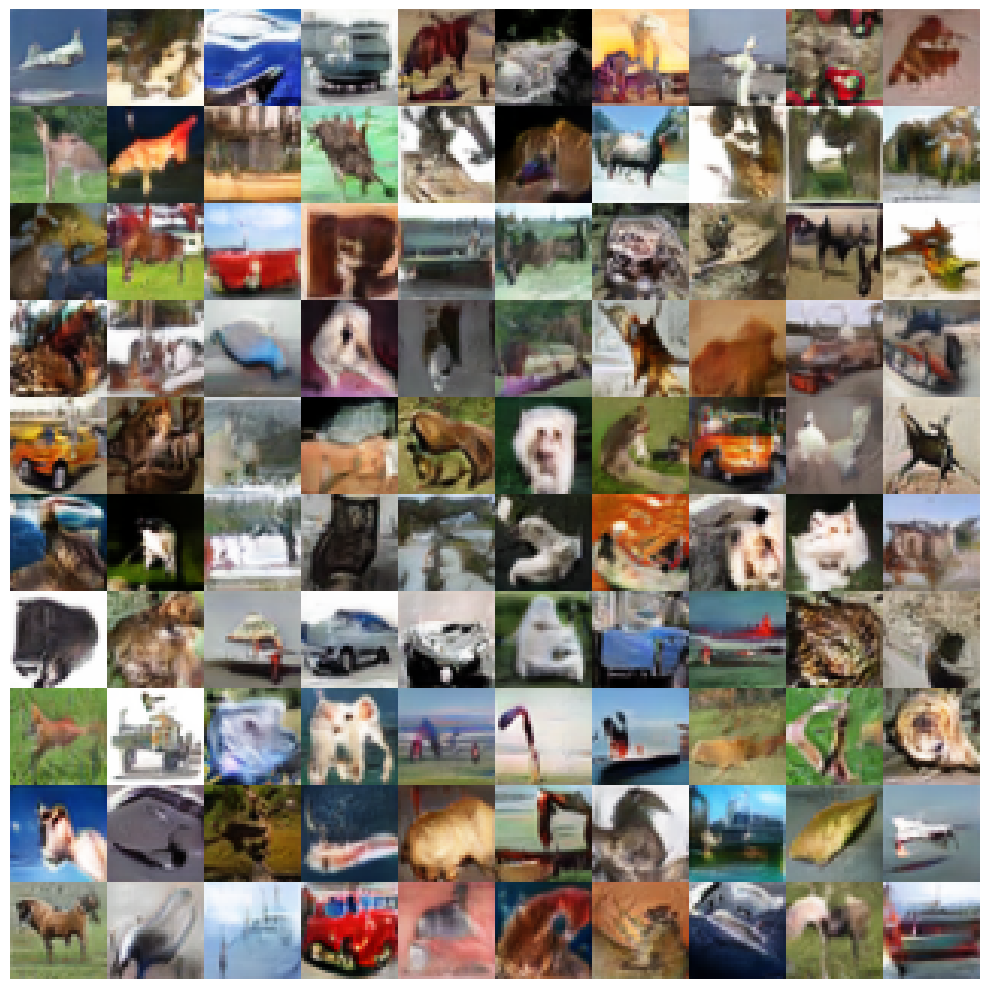}
    }
    \hfill
    \subfloat[HessianFR-CG1 (ResNet).]{
    \label{vis_cifar10_gen_data:d}   
    \includegraphics[width=0.45\textwidth]{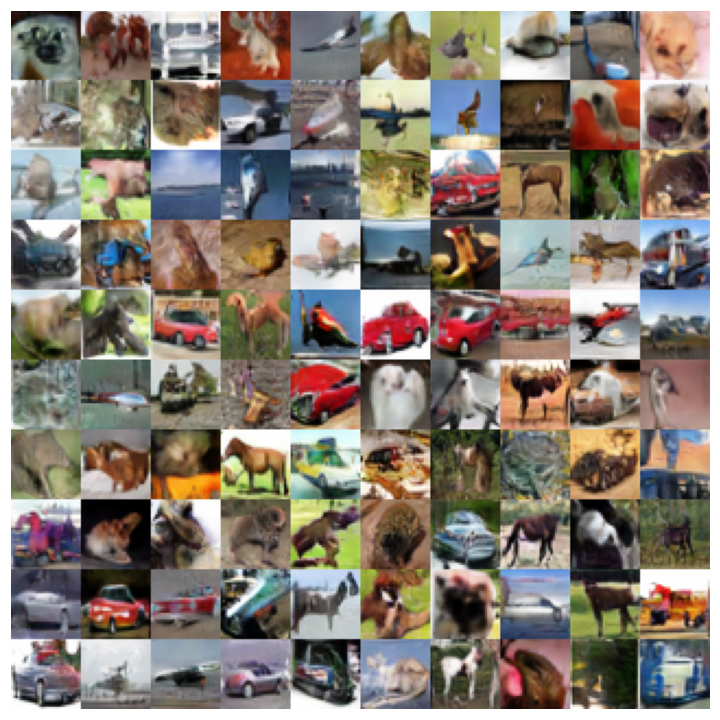}
    }
    \caption{Random samples from the generated distribution trained on CIFAR-10 dataset.}
    \label{vis_cifar10_gen_data}
\end{figure*}

\begin{figure*}[h!]
  \centering
  \subfloat[HessianFR-DG (Standard CNN).] {
     \label{vis_mnist_gen_data:a}     
    \includegraphics[width=0.45\textwidth]{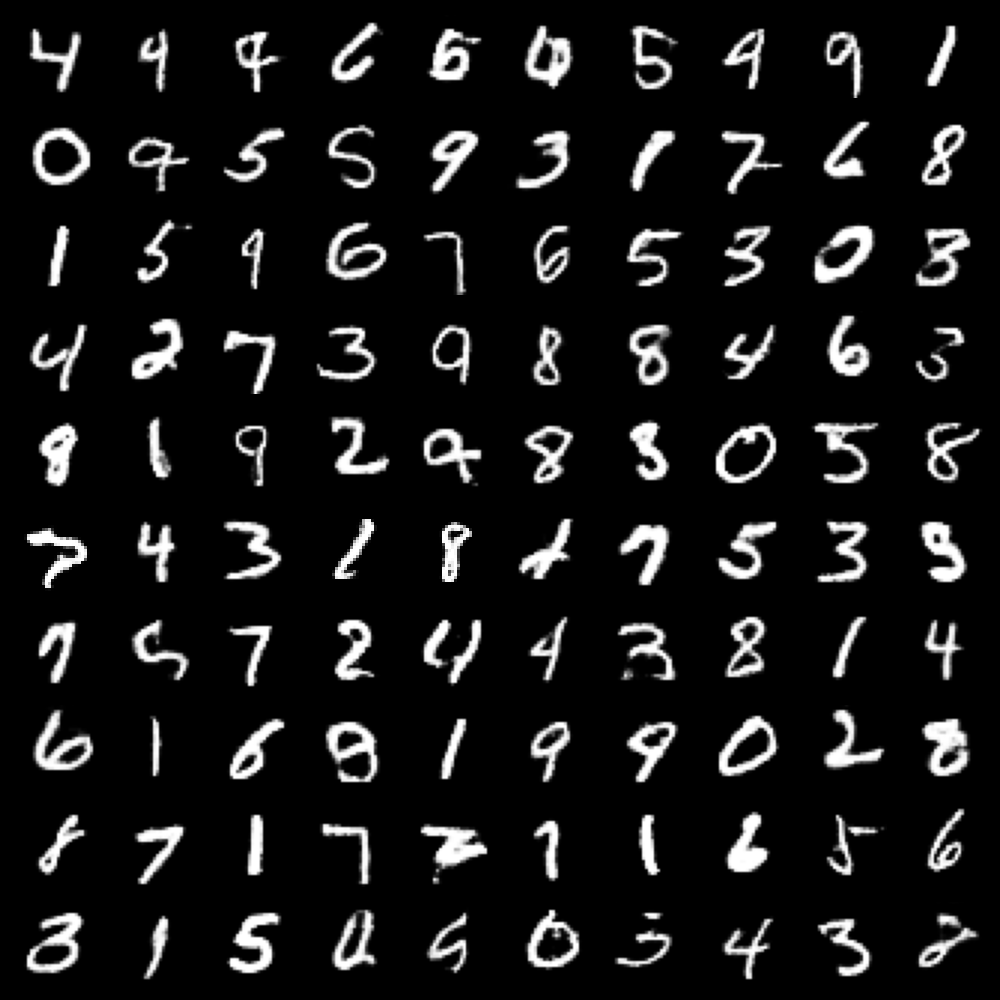}
    }
    \hfill
    \subfloat[HessianFR-CG1 (Standard CNN).] {
    \label{vis_mnist_gen_data:b}     
    \includegraphics[width=0.45\textwidth]{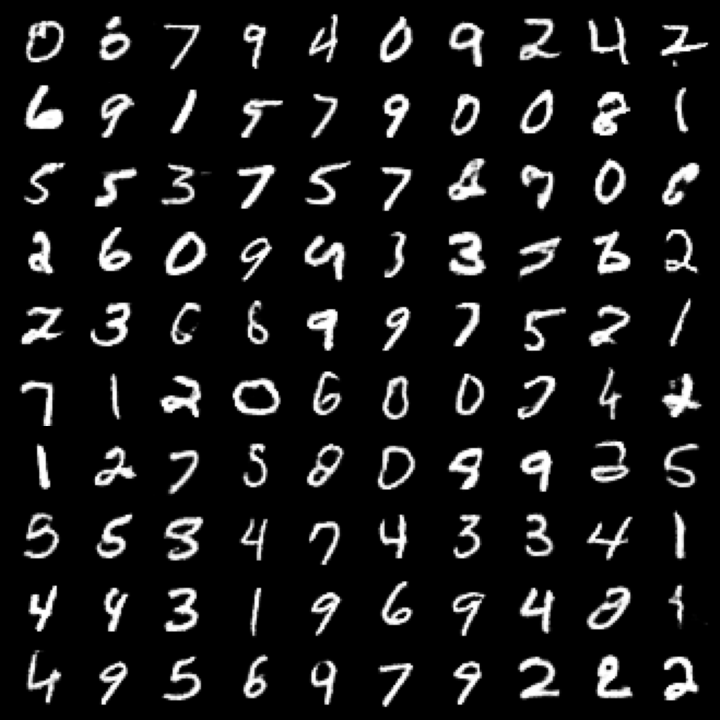}
    }
    \caption{Random samples from the generated distribution trained on MNIST dataset.}
    \label{vis_mnist_gen_data}
\end{figure*}

\begin{figure*}[h!]
  \centering
  \subfloat[HessianFR-DG (ResNet).] {
     \label{vis_celeba_gen_data:a}     
    \includegraphics[width=0.45\textwidth]{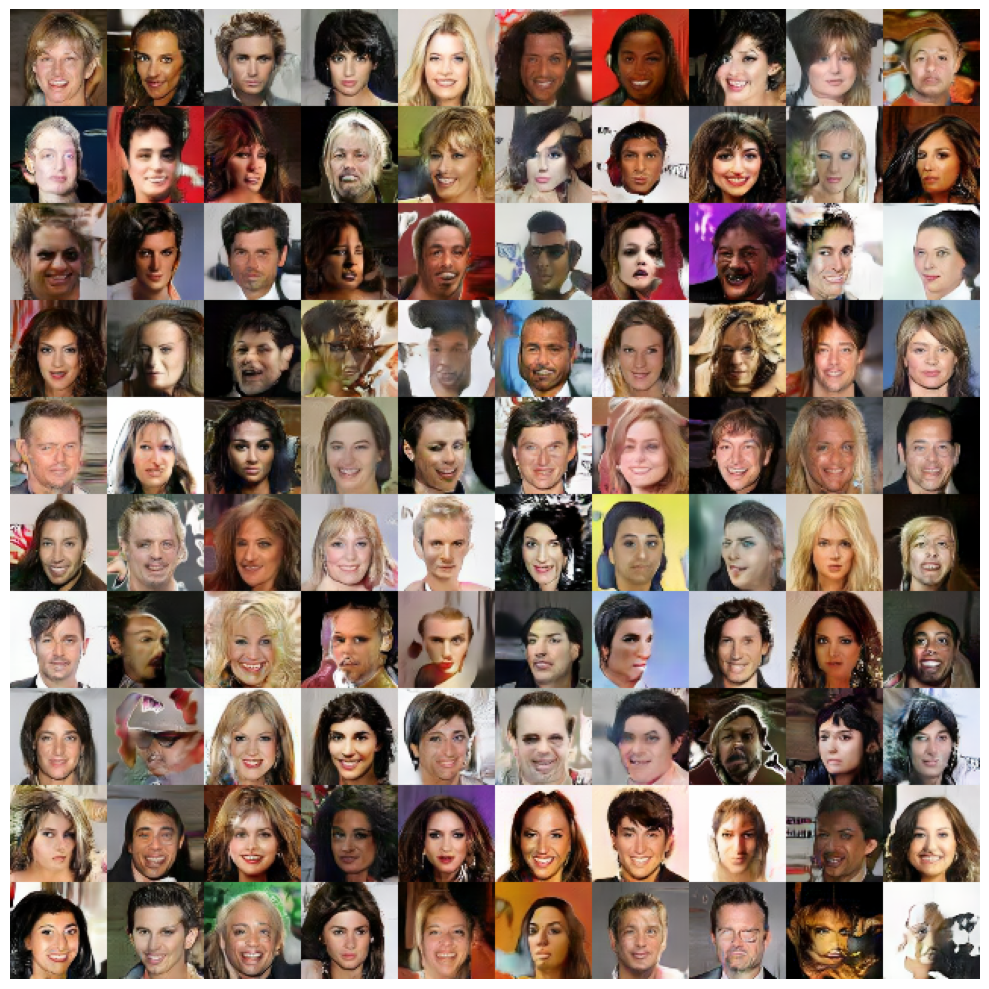}
    }
    \hfill
    \subfloat[HessianFR-CG1 (ResNet).] {
    \label{vis_celeba_gen_data:b}     
    \includegraphics[width=0.45\textwidth]{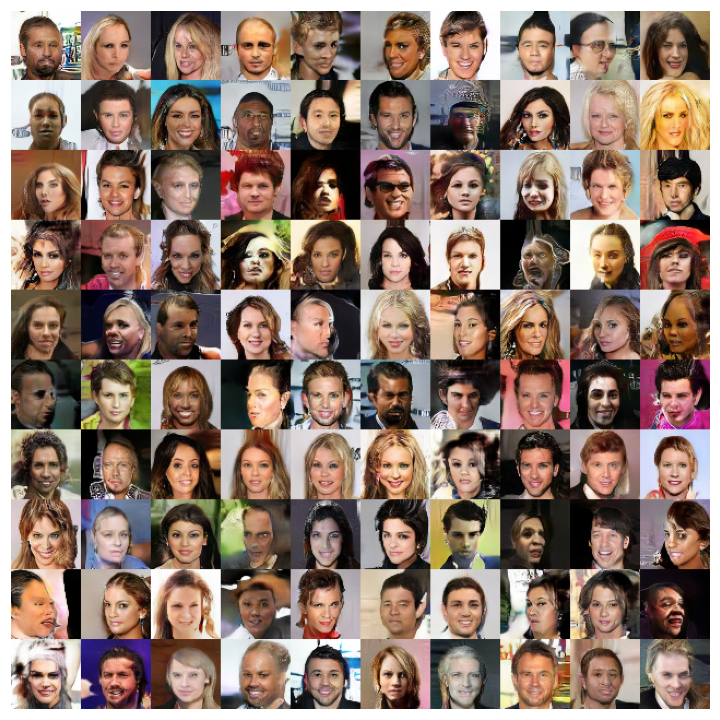}
    }
    \caption{Random samples from the generated distribution trained on CelebA dataset.}
    \label{vis_celeba_gen_data}
\end{figure*}

\section{Conclusion}
In this paper, we study a local minimax based algorithm (HessianFR) in solving sequential min-max problems. We first discuss the reasonableness and advantages of local minimax over local Nash equilibrium in sequential min-max problems. Then, to improve and generalize FR algorithm \cite{Wang2020On}, we propose the HessianFR which further utilizes Hessian information over FR.  Local convergence properties of HessianFR for both deterministic and stochastic optimizations are studied. Numerical experiments on toy examples and large scale datasets validate the effectiveness of the proposed HessianFR algorithm in training GANs for image generation. 

\appendix
\section{Local Convergence of GDA and its Variants}
\label{appendix_local_cong_GDA}
In this part, we briefly introduce the local convergence of GDA and its variant in terms of local minimax. The work has been done in \cite{zhang2020newton}. For completeness, we put the results here. 

Algorithmic update for Two Time-Scale GDA:
\begin{equation*}
    \begin{array}{ll}
        \mathbf{x}_{t+1} = \mathbf{x}_{t} - \eta_{\mathbf{x}} \nabla_{\mathbf{x}}f(\mathbf{x}_t,\mathbf{y}_t),\\ 
        \mathbf{y}_{t+1} = \mathbf{y}_{t} + \eta_{\mathbf{y}}\nabla_{\mathbf{y}}f(\mathbf{x}_t,\mathbf{y}_t).
    \end{array}
\end{equation*}

\begin{theorem}
[Local convergence of Two Time-Scale GDA] Let $\mathbf{z}^{*}=(\mathbf{x}^{*},\mathbf{y}^{*})$ to be a strict local minimax. For any $\delta > 0$, $\exists c_0 > 0$, such that if $c = \eta_{\mathbf{y}}/\eta_{\mathbf{x}}>c_0$, Two Time-Scale GDA has asymptotic linear convergence rate $\sigma=\max\{\sigma_1, \sigma_2\}$, where $\sigma_1= \max \{|1-\eta_{\mathbf{x}}\lambda_{\text{min}}(\mathbf{H}_{\mathbf{x x}}^{*} - \mathbf{H}_{\mathbf{x y}}^{*}\mathbf{H}_{\mathbf{y y}}^{*-1}\mathbf{H}_{\mathbf{y x}}^{*})|, |1-\eta_{\mathbf{x}}\lambda_{\text{max}}(\mathbf{H}_{\mathbf{x x}}^{*} - \mathbf{H}_{\mathbf{x y}}^{*}\mathbf{H}_{\mathbf{y y}}^{*-1}\mathbf{H}_{\mathbf{y x}}^{*})| \}+ \eta_{\mathbf{x}}\delta$ and $\sigma_2= \max \{|1-\eta_{\mathbf{y}}\lambda_{\text{min}}(\mathbf{H}_{\mathbf{y y}}^{*})|, |1-\eta_{\mathbf{y}}\lambda_{\text{max}}(\mathbf{H}_{\mathbf{y y}}^{*})| \}+ \eta_{\mathbf{y}}\delta$. To guarantee the convergence that $\sigma < 1$, we let $\delta$ to be small enough and properly select the learning rates $\eta_{\mathbf{x}}$ and $\eta_{\mathbf{y}}$.
\end{theorem}

Sketch of proof: We consider the Jacobian of Two Time-Scale GDA (TTSGDA) at strict local minimax $\mathbf{z}^{*}=(\mathbf{x}^{*},\mathbf{y}^{*})$:
\begin{equation}
\label{jacobian_ttsgda}
    \mathbf{J}_{\text{TTSGDA}}^{*} = \mathbf{I} - \eta_{\mathbf{x}} \left[ \begin{matrix} \mathbf{H}_{\mathbf{x x}}^{*} & \mathbf{H}_{\mathbf{x y}}^{*}\\ -c\mathbf{H}_{\mathbf{y x}}^{*} & -c\mathbf{H}_{\mathbf{y y}}^{*}   \end{matrix} \right] =:\mathbf{I} - \eta_{\mathbf{x}} \mathbf{U}^{*}_{\text{TTSGDA}}.
\end{equation}We let $c$ to be large enough such that the eigenvalues of $\mathbf{U}^{*}_{\text{TTSGDA}}$ is close to eigenvalues of $\mathbf{H}_{\mathbf{x x}}^{*} - \mathbf{H}_{\mathbf{x y}}^{*}\mathbf{H}_{\mathbf{y y}}^{*-1}\mathbf{H}_{\mathbf{y x}}^{*}$ and $\mathbf{H}_{\mathbf{y y}}^{*}$. For more details, please refer to \cite{jin2020local}.

Update rule for GDA-k:
\begin{equation*}
    \begin{array}{l}
        \mathbf{x}_{t+1} = \mathbf{x}_{t} - \eta_{\mathbf{x}} \nabla_{\mathbf{x}}f(\mathbf{x}_t,\mathbf{y}_t),\\ 
        \mathbf{y}_{t}^{(i+1)} = \mathbf{y}_{t}^{(i)} + \eta_{\mathbf{y}}\nabla_{\mathbf{y}}f(\mathbf{x}_t,\mathbf{y}_t^{(i)}), i=1 \cdots k-1,\\
        \mathbf{y}_{t+1} = \mathbf{y}_{t}^{(k)} + \eta_{\mathbf{y}}\nabla_{\mathbf{y}}f(\mathbf{x}_t,\mathbf{y}_t^{(k)}).
    \end{array}
\end{equation*}

\begin{theorem}
[Local convergence of GDA-k] GDA-k achieves an asymptotic linear convergence rate at the strict local minimax $\mathbf{z}^{*}$:
$\sigma = \max \{|1-\eta_{\mathbf{x}}\lambda_{\text{min}}(\mathbf{H}_{\mathbf{x x}}^{*} - \mathbf{H}_{\mathbf{x y}}^{*}\mathbf{H}_{\mathbf{y y}}^{*-1}\mathbf{H}_{\mathbf{y x}}^{*})|, |1-\eta_{\mathbf{x}}\lambda_{\text{max}}(\mathbf{H}_{\mathbf{x x}}^{*} - \mathbf{H}_{\mathbf{x y}}^{*}\mathbf{H}_{\mathbf{y y}}^{*-1}\mathbf{H}_{\mathbf{y x}}^{*})| \}$ when $k \to \infty$ and $\eta_{\mathbf{y}}<2/\lambda_{\text{min}}(\mathbf{H}_{\mathbf{y y}}^{*})$. To guarantee the convergence that $\sigma < 1$, we properly select the learning rates $\eta_{\mathbf{x}}$ and $\eta_{\mathbf{y}}$.
\end{theorem}

Entra-gradient (EG), as an extention of Two Time-Scale GDA, is derived from proximal point approach in simultaneous games. Mathematically, it is formulated as
\begin{equation*}
    \begin{array}{ll}
        \mathbf{x}_{t+1/2} = \mathbf{x}_{t} - \eta_{\mathbf{x}} \nabla_{\mathbf{x}}f(\mathbf{x}_t,\mathbf{y}_t),\\ 
        \mathbf{y}_{t+1/2} = \mathbf{y}_{t} + \eta_{\mathbf{y}}\nabla_{\mathbf{y}}f(\mathbf{x}_t,\mathbf{y}_t),\\
        \mathbf{x}_{t+1} = \mathbf{x}_{t} - \eta_{\mathbf{x}} \nabla_{\mathbf{x}}f(\mathbf{x}_{t+1/2},\mathbf{y}_{t+1/2}),\\ 
        \mathbf{y}_{t+1} = \mathbf{y}_{t} + \eta_{\mathbf{y}}\nabla_{\mathbf{y}}f(\mathbf{x}_{t+1/2},\mathbf{y}_{t+1/2}).\\
    \end{array}
\end{equation*}

Unfortunately, we find that EG may have worse convergence than Two Time-Scale GDA theoretically. Numerical experiments on toy examples in \cref{numerical_experiments} also show the worse performance of EG in solving sequential min-max problem \cref{min-max}. The next theorem illustrates that EG is not suitable to solve sequential min-max problems. 
\begin{theorem}
[Local convergence of EG] EG either suffers from rotational behavior (instability) or behaves worse than Two Time-Scale GDA almost surely. 
\end{theorem}

Sketch of proof: We consider the Jacobian of Extra Gradient (EG) at a strict local minimax $\mathbf{z}^{*}=(\mathbf{x}^{*},\mathbf{y}^{*})$. We relate EG to Two Time-Scale GDA:
\begin{equation*}
\mathbf{J}^{*}_{\text{EG}} = \mathbf{I} - \eta_{\mathbf{x}} \mathbf{U}^{*}_{\text{TTSGDA}} + \eta_{\mathbf{x}}^2 \mathbf{U}^{*2}_{\text{TTSGDA}}, 
\end{equation*}where $\mathbf{U}^{*}_{\text{TTSGDA}}$ is defined in \cref{jacobian_ttsgda}. Suppose that the eigenvalues of $\eta_{\mathbf{x}}\mathbf{U}^{*}_{\text{TTSGDA}}$ is $\lambda_1, \lambda_2,$ $\cdots, \lambda_{d_1+d_2}$ and let $\lambda_j = a_j + b_j i$ where $i$ is the imaginary unit. For the convergence of Two Time-Scale GDA, we require $|1-\lambda_j| \in [0,1)$. Similarly, to guarantee the convergence of EG, we have $|1-\lambda_j+\lambda_j^2|=|(1+a_j^2+b_j^2-a_j)+(2a_j-1)b_j i| \in [0,1)$. The rotational behavior arises if the Jacobian matrix has eigenvalues with large imaginary parts. Therefore, if $a_j \neq 1/2$ (almost surely), then $(2a_j-1)b_j \approx 0$ iff $b_j \approx 0$ because the boundedness of $1+a_j^2+b_j^2-a_j$ implies the boundedness of $a_j$. When $b_j \approx 0$, we have $|(1+a_j^2+b_j^2-a_j)+(2a_j-1)b_j i| \approx |1+a_j^2-a_j| \geq |1-a_j|$, which implies that EG converges slower than Two Time-Scale GDA.

\bibliographystyle{siamplain}
\bibliography{references}

\begin{thebibliography}{10}

\bibitem{adler2018banach}
{\sc J.~Adler and S.~Lunz}, {\em Banach wasserstein gan}, in Advances in Neural
  Information Processing Systems, vol.~31, Curran Associates, Inc., 2018.

\bibitem{arjovsky2017towards}
{\sc M.~Arjovsky and L.~Bottou}, {\em Towards principled methods for training
  generative adversarial networks}, CoRR, abs/1701.04862 (2017),
  \url{https://arxiv.org/abs/1701.04862}.

\bibitem{arjovsky2017wasserstein}
{\sc M.~Arjovsky, S.~Chintala, and L.~Bottou}, {\em Wasserstein generative
  adversarial networks}, in International Conference on Machine Learning, PMLR,
  2017, pp.~214--223.

\bibitem{balduzzi2018mechanics}
{\sc D.~Balduzzi, S.~Racaniere, J.~Martens, J.~Foerster, K.~Tuyls, and
  T.~Graepel}, {\em The mechanics of n-player differentiable games}, in
  International Conference on Machine Learning, PMLR, 2018, pp.~354--363.

\bibitem{bekas2007estimator}
{\sc C.~Bekas, E.~Kokiopoulou, and Y.~Saad}, {\em An estimator for the diagonal
  of a matrix}, Applied Numerical Mathematics, 57 (2007), pp.~1214--1229.

\bibitem{bowman2015generating}
{\sc S.~R. Bowman, L.~Vilnis, O.~Vinyals, A.~Dai, R.~Jozefowicz, and
  S.~Bengio}, {\em Generating sentences from a continuous space}, in
  Proceedings of The 20th {SIGNLL} Conference on Computational Natural Language
  Learning, Berlin, Germany, Aug. 2016, Association for Computational
  Linguistics, pp.~10--21, \url{https://doi.org/10.18653/v1/K16-1002}.

\bibitem{brockLRW17}
{\sc A.~Brock, T.~Lim, J.~M. Ritchie, and N.~Weston}, {\em Neural photo editing
  with introspective adversarial networks}, in 5th International Conference on
  Learning Representations, {ICLR} 2017,Toulon, France, April 24-26, 2017,
  Conference Track Proceedings, OpenReview.net, 2017.

\bibitem{cao2018improving}
{\sc Y.~Cao, G.~W. Ding, K.~Y.-C. Lui, and R.~Huang}, {\em Improving {GAN}
  training via binarized representation entropy ({BRE}) regularization}, in
  International Conference on Learning Representations, 2018.

\bibitem{dai2018sbeed}
{\sc B.~Dai, A.~Shaw, L.~Li, L.~Xiao, N.~He, Z.~Liu, J.~Chen, and L.~Song},
  {\em Sbeed: Convergent reinforcement learning with nonlinear function
  approximation}, in International Conference on Machine Learning, PMLR, 2018,
  pp.~1125--1134.

\bibitem{daskalakis2018training}
{\sc C.~Daskalakis, A.~Ilyas, V.~Syrgkanis, and H.~Zeng}, {\em Training {GAN}s
  with optimism}, in International Conference on Learning Representations,
  2018.

\bibitem{daskalakis2018limit}
{\sc C.~Daskalakis and I.~Panageas}, {\em The limit points of (optimistic)
  gradient descent in min-max optimization}, in Advances in Neural Information
  Processing Systems, S.~Bengio, H.~Wallach, H.~Larochelle, K.~Grauman,
  N.~Cesa-Bianchi, and R.~Garnett, eds., vol.~31, Curran Associates, Inc.,
  2018.

\bibitem{echenique2003equilibrium}
{\sc F.~Echenique}, {\em The equilibrium set of two-player games with
  complementarities is a sublattice}, Economic Theory, 22 (2003), pp.~903--905.

\bibitem{evtushenko1974iterative}
{\sc Y.~G. Evtushenko}, {\em Iterative methods for solving minimax problems},
  USSR Computational Mathematics and Mathematical Physics, 14 (1974),
  pp.~52--63.

\bibitem{evtushenko1974some}
{\sc Y.~G. Evtushenko}, {\em Some local properties of minimax problems}, USSR
  Computational Mathematics and Mathematical Physics, 14 (1974), pp.~129--138.

\bibitem{fiez2020implicit}
{\sc T.~Fiez, B.~Chasnov, and L.~Ratliff}, {\em Implicit learning dynamics in
  stackelberg games: Equilibria characterization, convergence analysis, and
  empirical study}, in International Conference on Machine Learning, PMLR,
  2020, pp.~3133--3144.

\bibitem{gidel2018a}
{\sc G.~Gidel, H.~Berard, G.~Vignoud, P.~Vincent, and S.~Lacoste-Julien}, {\em
  A variational inequality perspective on generative adversarial networks}, in
  International Conference on Learning Representations, 2019.

\bibitem{goodfellow2016nips}
{\sc I.~Goodfellow}, {\em Nips 2016 tutorial: Generative adversarial networks},
  arXiv preprint arXiv:1701.00160,  (2016).

\bibitem{goodfellow2014generative}
{\sc I.~Goodfellow, J.~Pouget-Abadie, M.~Mirza, B.~Xu, D.~Warde-Farley,
  S.~Ozair, A.~Courville, and Y.~Bengio}, {\em Generative adversarial nets},
  Advances in Neural Information Processing Systems, 27 (2014).

\bibitem{gross2011recovering}
{\sc D.~Gross}, {\em Recovering low-rank matrices from few coefficients in any
  basis}, IEEE Transactions on Information Theory, 57 (2011), pp.~1548--1566.

\bibitem{gulrajani2017improved}
{\sc I.~Gulrajani, F.~Ahmed, M.~Arjovsky, V.~Dumoulin, and A.~C. Courville},
  {\em Improved training of wasserstein gans}, in Advances in Neural
  Information Processing Systems, vol.~30, Curran Associates, Inc., 2017.

\bibitem{heusel2017gans}
{\sc M.~Heusel, H.~Ramsauer, T.~Unterthiner, B.~Nessler, and S.~Hochreiter},
  {\em Gans trained by a two time-scale update rule converge to a local nash
  equilibrium}, Advances in Neural Information Processing Systems, 30 (2017).

\bibitem{horn2012matrix}
{\sc R.~A. Horn and C.~R. Johnson}, {\em Matrix analysis}, Cambridge University
  Press, 2012.

\bibitem{jin2020local}
{\sc C.~Jin, P.~Netrapalli, and M.~Jordan}, {\em What is local optimality in
  nonconvex-nonconcave minimax optimization?}, in International Conference on
  Machine Learning, PMLR, 2020, pp.~4880--4889.

\bibitem{kingma2014adam}
{\sc D.~P. Kingma and J.~Ba}, {\em Adam: A method for stochastic optimization},
  in ICLR (Poster), 2015.

\bibitem{kohler2017sub}
{\sc J.~M. Kohler and A.~Lucchi}, {\em Sub-sampled cubic regularization for
  non-convex optimization}, in International Conference on Machine Learning,
  PMLR, 2017, pp.~1895--1904.

\bibitem{korpelevich1976extragradient}
{\sc G.~M. Korpelevich}, {\em The extragradient method for finding saddle
  points and other problems}, Matecon, 12 (1976), pp.~747--756.

\bibitem{krizhevsky2009learning}
{\sc A.~Krizhevsky, G.~Hinton, et~al.}, {\em Learning multiple layers of
  features from tiny images},  (2009).

\bibitem{lecun1998mnist}
{\sc Y.~LeCun}, {\em The mnist database of handwritten digits}, http://yann.
  lecun. com/exdb/mnist/,  (1998).

\bibitem{liu1989limited}
{\sc D.~C. Liu and J.~Nocedal}, {\em On the limited memory bfgs method for
  large scale optimization}, Mathematical Programming, 45 (1989), pp.~503--528.

\bibitem{liu2015faceattributes}
{\sc Z.~Liu, P.~Luo, X.~Wang, and X.~Tang}, {\em Deep learning face attributes
  in the wild}, in Proceedings of International Conference on Computer Vision
  (ICCV), December 2015.

\bibitem{madry2018towards}
{\sc A.~Madry, A.~Makelov, L.~Schmidt, D.~Tsipras, and A.~Vladu}, {\em Towards
  deep learning models resistant to adversarial attacks}, in International
  Conference on Learning Representations, 2018.

\bibitem{mescheder2017numerics}
{\sc L.~Mescheder, S.~Nowozin, and A.~Geiger}, {\em The numerics of gans}, in
  Advances in Neural Information Processing Systems, I.~Guyon, U.~V. Luxburg,
  S.~Bengio, H.~Wallach, R.~Fergus, S.~Vishwanathan, and R.~Garnett, eds.,
  vol.~30, Curran Associates, Inc., 2017.

\bibitem{miyato2018spectral}
{\sc T.~Miyato, T.~Kataoka, M.~Koyama, and Y.~Yoshida}, {\em Spectral
  normalization for generative adversarial networks}, in International
  Conference on Learning Representations, 2018.

\bibitem{odena2017conditional}
{\sc A.~Odena, C.~Olah, and J.~Shlens}, {\em Conditional image synthesis with
  auxiliary classifier gans}, in International Conference on Machine Learning,
  PMLR, 2017, pp.~2642--2651.

\bibitem{park2019sphere}
{\sc S.~W. Park and J.~Kwon}, {\em Sphere generative adversarial network based
  on geometric moment matching}, in Proceedings of the IEEE/CVF Conference on
  Computer Vision and Pattern Recognition, 2019, pp.~4292--4301.

\bibitem{saad2003iterative}
{\sc Y.~Saad}, {\em Iterative methods for sparse linear systems}, SIAM, 2003.

\bibitem{salimans2016improved}
{\sc T.~Salimans, I.~Goodfellow, W.~Zaremba, V.~Cheung, A.~Radford, and
  X.~Chen}, {\em Improved techniques for training gans}, Advances in Neural
  Information Processing Systems, 29 (2016), pp.~2234--2242.

\bibitem{schaul2013no}
{\sc T.~Schaul, S.~Zhang, and Y.~LeCun}, {\em No more pesky learning rates}, in
  International Conference on Machine Learning, PMLR, 2013, pp.~343--351.

\bibitem{shaham2018understanding}
{\sc U.~Shaham, Y.~Yamada, and S.~Negahban}, {\em Understanding adversarial
  training: Increasing local stability of supervised models through robust
  optimization}, Neurocomputing, 307 (2018), pp.~195--204.

\bibitem{song2019generative}
{\sc Y.~Song and S.~Ermon}, {\em Generative modeling by estimating gradients of
  the data distribution}, in Advances in Neural Information Processing Systems,
  vol.~32, Curran Associates, Inc., 2019.

\bibitem{song2020sliced}
{\sc Y.~Song, S.~Garg, J.~Shi, and S.~Ermon}, {\em Sliced score matching: A
  scalable approach to density and score estimation}, in Uncertainty in
  Artificial Intelligence, PMLR, 2020, pp.~574--584.

\bibitem{tropp2012user}
{\sc J.~A. Tropp}, {\em User-friendly tail bounds for sums of random matrices},
  Foundations of Computational Mathematics, 12 (2012), pp.~389--434.

\bibitem{Wang2020On}
{\sc Y.~Wang*, G.~Zhang*, and J.~Ba}, {\em On solving minimax optimization
  locally: A follow-the-ridge approach}, in International Conference on
  Learning Representations, 2020.

\bibitem{wei2018improving}
{\sc X.~Wei, Z.~Liu, L.~Wang, and B.~Gong}, {\em Improving the improved
  training of wasserstein {GAN}s}, in International Conference on Learning
  Representations, 2018.

\bibitem{xu2020second}
{\sc P.~Xu, F.~Roosta, and M.~W. Mahoney}, {\em Second-order optimization for
  non-convex machine learning: An empirical study}, in Proceedings of the 2020
  SIAM International Conference on Data Mining, SIAM, 2020, pp.~199--207.

\bibitem{yao2020adahessian}
{\sc Z.~Yao, A.~Gholami, S.~Shen, M.~Mustafa, K.~Keutzer, and M.~W. Mahoney},
  {\em Adahessian: An adaptive second order optimizer for machine learning},
  arXiv preprint arXiv:2006.00719,  (2020).

\bibitem{zhang2020newton}
{\sc G.~Zhang, K.~Wu, P.~Poupart, and Y.~Yu}, {\em Newton-type methods for
  minimax optimization}, In ICML workshop on Beyond First-Order Methods in ML
  Systems, arXiv:2006.14592,  (2021).

\end{thebibliography}

\end{document}